\documentclass[10pt]{amsart}
\usepackage[libertine]{newtxmath}
\usepackage{amsmath}
\usepackage{amsthm}
\usepackage{mathrsfs}
\usepackage{amsfonts}
\usepackage{graphicx}
\usepackage{enumerate}
\graphicspath{{../../drawings/}}
\usepackage[dvipsnames]{xcolor}

\definecolor{limegreen}{rgb}{0.196,0.804,0.196}
\definecolor{darkgreen}{rgb}{0.0,0.5,0.0}
\definecolor{darkbluegreen}{rgb}{0,0.3,0.6}
\definecolor{badgerred}{rgb}{0.715,0.004,0.004}
\usepackage[font=small,labelfont={bf,sf}, textfont={sf},margin=0em]{caption}
\usepackage{url,hyperref}
\hypersetup{colorlinks,
		citecolor=badgerred,
		filecolor=black,
		linkcolor=blue,
		urlcolor=darkgreen,
		bookmarksopen=false,
		pdftex}

\newcommand{\bom}{{\boldsymbol\Omega}}
\newcommand{\be}{\begin{equation}}
\newcommand{\ee}{\end{equation}}
\newcommand{\ba}{\begin{aligned}}
\newcommand{\ea}{\end{aligned}}
 \newcommand{\bee}{\begin{equation*}}
\newcommand{\eee}{\end{equation*}}
\newcommand{\R}{{\mathbb R}}
\newcommand{\bq}{{\bar q}}
\newcommand{\brho}{{\bar \sigma}}
\newcommand{\bu}{{\bar u}}
\newcommand{\cA}{\mathcal{A}}
\newcommand{\cB}{\mathcal{B}}
\newcommand{\cE}{\mathcal{E}}
\newcommand{\cF}{\mathcal{F}}
\newcommand{\cG}{\mathcal{G}}
\newcommand{\cI}{\mathcal{I}}
\newcommand{\cL}{\mathcal{L}}
\newcommand{\cO}{\mathcal{O}}
\newcommand{\cQ}{\mathcal{Q}}
\newcommand{\cT}{\mathcal{T}}
\newcommand{\hhm}{H}
\newcommand{\hilb}{\mathcal H}
\newcommand{\hm}{H}
\newcommand{\hv}{\mathcal{D}}
\newcommand{\mc}[1]{{\mathcal #1}}
\newcommand{\sk}{\smallskip}
\newcommand{\tw}{{\tilde w}}
\newcommand{\bell}{{\bar  \ell}}
\newcommand{\ueo}{u_{\epsilon,\bom}}

\newtheorem{theorem}{Theorem}[section]
\newtheorem{proposition}[theorem]{Proposition}
\newtheorem{lemma}[theorem]{Lemma}
\newtheorem{definition}[theorem]{Definition}
\newtheorem{prop}[theorem]{Proposition}
\newtheorem{corollary}[theorem]{Corollary}

\theoremstyle{remark}

\newtheorem{remark}[theorem]{Remark}

\newtheorem{claim}[theorem]{Claim}

\newtheorem{step}{Step}

\numberwithin{equation}{section}
\numberwithin{theorem}{section}

\begin{document}

\title{Mean Curvature Flow near a Peanut Solution}
\author{S.  B.  Angenent}\thanks{The first author is the corresponding author.  On behalf of all authors, the corresponding author states that there is no conflict of interest. }
\author{P.  Daskalopoulos}
\author{N.  Sesum}
\begin{abstract}
It was shown in  \cite{AAG} and \cite{AV} that there exist closed mean curvature flow solutions that extinct to a point in finite time, without ever becoming convex prior to their extinction.  These solutions develop a {\it degenerate neckpinch} singularity, meaning that the tangent flow at a singularity is a round cylinder, but at the same time for each of these solutions there exists a sequence of points in space and time, so that the pointed blow up limit around this sequence is the Bowl soliton.  These solutions are called  {\it peanut} solutions and they were first conjectured to exist by Richard Hamilton, while the existence of those solutions was shown in \cite{AAG}.
In this paper we show that this type of solutions are highly unstable, in the sense that in every small neighborhood of any such peanut solution we can  find   a perturbation so that the mean curvature flow starting at that perturbation develops spherical singularity, and at the same time we can find a perturbation so that the mean curvature flow starting at that perturbation develops a nondegenerate neckpinch singularity.  We also show that appropriately rescaled subsequence of any sequence of solutions whose initial data converge to the peanut solution, and all of which develop spherical singularities, converges to the Ancient oval solution.

\end{abstract}

\maketitle
\setcounter{tocdepth}{2}
%\begingroup\footnotesize\sffamily\tableofcontents \endgroup
\section{Introduction}

%\subsection{Parametrized families of Mean Curvature Flows}

We consider  families of compact hypersurfaces $\bar M_{\theta}(t)\subset\R^{n+1}$ that evolve by Mean Curvature Flow, and which depend continuously on the parameter $\theta\in\Theta$; the parameter belongs to some topological space $\Theta$, which in our examples will always be an open subset of $\R^{m}$ for some $m\geq 1$.  
These solutions become singular at a finite time $T(\theta)$ which may vary with the parameter $\theta\in\Theta$.  
Such solutions have a parametrization $(p, t, \theta)\in\mc M^n\times[0, \infty)\times\Theta\mapsto \hat F(p, t, \theta)\in\R^{n+1}$ whose domain is an open subset of $\mc M^n\times[0, \infty) \times \Theta$ given   by
\[
\mc O = \bigl\{ (p, t, \theta)\in \mc M^n\times[0, \infty) \times \Theta \mid 0\leq t<T(\theta) \bigr \}.  
\]
For each $\theta\in\Theta$ the immersion $p\mapsto \hat F(p, t, \theta)$ satisfies the Mean Curvature Flow equation
\begin{equation}		\label{eq-mcf}
\bigl(\partial_t\hat F\bigr)^{\perp} = \Delta_{\hat F}(\hat F), 
\tag{MCF}
\end{equation}
in which $(\partial_t\hat F)^\perp$ is the component perpendicular to $T_{\hat F(p, t, \theta)}\bar M_\theta(t)$ of $\partial_t\hat F(p, t, \theta)\in T_{\hat F(p,t,\theta)}\R^{n+1}$, and $\Delta_{\hat F}$ is the Laplacian of the pullback of the Euclidean metric under the immersion $p\mapsto \hat F(p, t, \theta)$.  

There have been many works towards understanding the formation of singularities in the mean curvature flow, that is classifying all possible singularity models.  It is a very hard, if not even impossible question to answer in its full generality.  To understand the singularities, which inevitably happen for closed mean curvature flows, one parabolically dilates around the singularity in space and time.  Huisken's monotonicity formula (\cite{Hu90}, \cite{Ilm95}) guarantees that a subsequential limit of such dilations will weakly limit to a tangent flow which will be a weak solution to \eqref{eq-mcf}, evolving only by homothety.  These solutions are called self-shrinking solutions.  We need to understand these tangent flows better in order to either continue the flow past singularities via surgery, or by showing some regularity for weak solutions past the singular time.  The problem is that tangent flow can come with multiplicity, and also that its mean curvature may change sign.

On the other hand, in \cite{CM} Colding and Minicozzi introduced the notion of entropy and showed that the only entropy stable shrinkers are the generalized cylinders.  These singularities if they occur with multiplicity one behave very well, and nice regularity results or well posedness of weak solutions were for example shown in \cite{CHHW22}, \cite{CM16} and \cite{Wh}.  Thanks to results in \cite{BK} and \cite{CCS} we know the mean curvature flow of a generic initial surface $M_0 \subset \mathbb{R}^3$ encounters only spherical and cylindrical singularities, and the flow is well-posed and is completely smooth for almost every time.  Combining the results in \cite{CCS} with the surgery construction in \cite{DH22}, the authors in \cite{CCS} construct a mean curvature flow with surgery for a generic initial $M(0) \subset \mathbb{R}^3$.  

\begin{figure}[t]
\centering
\includegraphics[width=\textwidth]{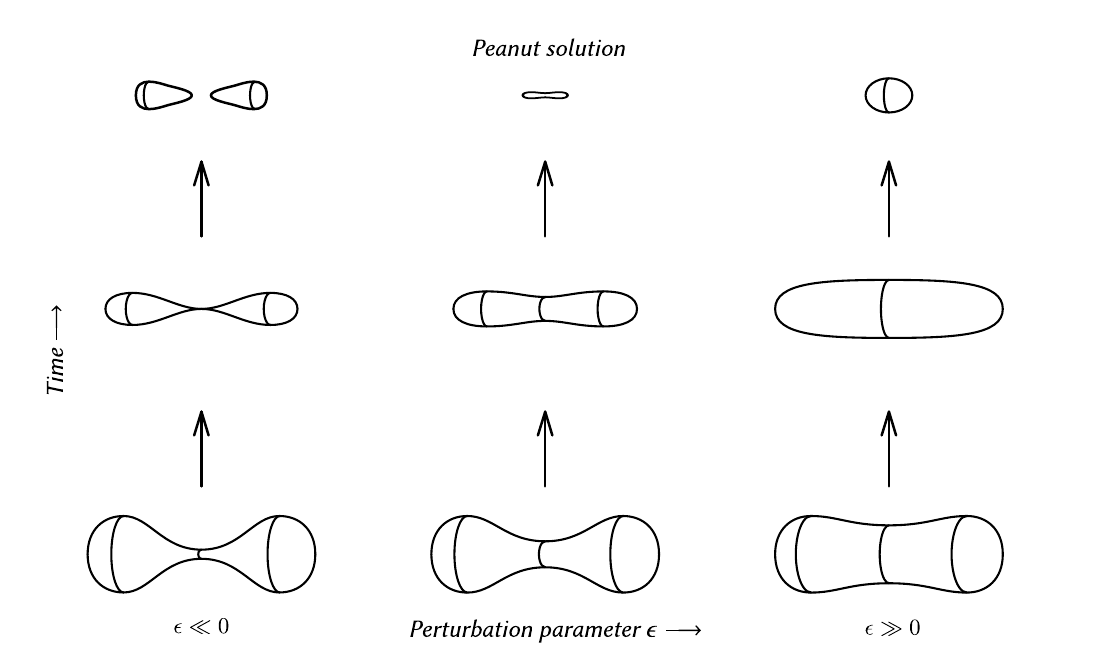}
\caption{\textbf{The peanut and neighbors.}  In the 1980ies Hamilton suggested an atypical singularity in a solution to MCF.  He considered a one-parameter family of initial surfaces $\{M_\epsilon \mid \epsilon\in\R\}$ each of which is a sphere.  For $\epsilon\ll0$ the surface $M_\epsilon$ has a neck that is so narrow that it will pinch before the whole solution can vanish in a point.  For $\epsilon\gg0$ the surface $M_\epsilon$ is convex, or so close to being convex that it quickly becomes convex and shrinks to a point, obeying Huisken's theorem \cite{Hu}.  Observing that continuous dependence on initial data should imply the existence of at least one parameter value $\epsilon_*\in\R$ whose corresponding solution $M_{\epsilon_*}(t)$  forms a neck pinch, but does not ever become convex.  A rigorous version of Hamilton's arguments appeared in \cite[Section 8]{AAG}}
\label{fig:peanut}
\end{figure}

The question is what happens in higher dimensions.  In \cite{CMS} it was shown that a generic closed mean curvature flow in $\mathbb{R}^{n+1}$, for $3 \le n \le 5$,  under certain entropy assumptions on an initial  hypersurface, develops only generic singularities, meaning the generalized cylinders. Roughly speaking, singularities modeled on generalized cylinders $\Sigma^k := \mathbb{S}^{n-k} \times \mathbb{R}^k$ are called in short {\it neckpinch} singularities.

Even among neckpinch singularities, there are different types of neckpinch singularities, i.e.  the nondegenerate and degenerate neckpinches.  We expect the former one to be generic, that is, if a mean curvature flow starting with some initial hypersurface $M_0$ develops a degenerate neckpinch singularity, we expect to find a sequence of perturbations converging to $M_0$, whose mean curvature flows all develop nondegenerate neckpinch singularities. 

\begin{definition}		\label{def-neckpinch}
We say that the Mean curvature flow  $\{\tilde{M}_t\}_{t\in [-1,0)}$ has a {\em neckpinch} singularity at $(0,0)$ if the singularity is modeled on $\Sigma^k:=
\mathbb{S}^{n-k} \times \mathbb{R}^k$  at $(0,0)$, and the associated Rescaled Mean Curvature flow  $\{M_{\tau}\}_{\tau\in [0,\infty)}$ converges to $\Sigma^k$  in the $C^{\infty}_{loc}$ sense.  
\end{definition}

In this paper we restrict our attention to the case $k = 1$, that is, when the singularity model is a round cylinder $\mathbb{S}^{n-1} \times \mathbb{R}.$  Assume the first singularity of the flow \eqref{eq-mcf} occurs at the spacetime point $(0,0) \in \R^n\times \R$.  
Consider then the rescaled mean curvature flow (RMCF) defined by $M_{\tau} = e^{\tau/2}\, \tilde{M}_{-e^{-\tau}}$, for $\tau\in [0,\infty)$.

To distinguish, at least in the geometric sense, between degenerate and nondegenerate neckinches  we introduce the following definition.  

\begin{definition}		\label{def-neckpinch-general}
A neckpinch singularity is called {\bf nondegenerate} if every pointed singularity model, that is, a smooth limit of any sequence of blow ups around $(x_i,t_i) \to (0,0)$, is a round cylinder $\Sigma^1$, and is called {\bf degenerate} if there is at least one blowup sequence around some $(x_i,t_i)\to (0,0)$ with a pointed limit that is not $\Sigma^1$.
\end{definition}

In this paper we focus on so-called \emph{peanut solutions} whose existence was first suggested by Richard Hamilton, and then established in~\cite{AAG,AV}.  In \cite{AV} the asymptotics of these solutions have been also established.  These are  examples of closed mean curvature flow solutions that contract to a point at the singular time, without ever becoming convex prior to that.  At the same time these are  examples of degenerate neckpinches.

In this paper we will restrict to the case where $\Theta$ is a {\em two dimensional set of parameters}  and consider perturbations of one of the  $4$-peanut solutions
(c.f.  in section \ref{sec:peanut properties}), that is a two parameter family of solutions $\{M_{\theta}(t)\,\,\, |\,\,\, \theta \in \Theta\}$ so that each $M_{\theta}(t)$ is a smooth MCF solution for $t\in [0,T(\theta))$, and so that for $\theta = {\bf 0} := (0,0)$ we have that $\bar{M}_{\bf 0}(t)$ is one of the $4$-peanut solutions.  More details on the peanut solutions  will be discussed in section \ref{sec:peanut}.
 
If for some parameter value $\theta\in\Theta$ the initial hypersurface $\bar M_\theta(0)$ is convex, then Huisken's theorem \cite{Hu} implies that $\bar M_\theta(t)$ contracts to a point as $t\nearrow T(\theta)$.  For other values of parameter $\theta$ the solution $\bar M_\theta(t)$ may become singular without shrinking to a point.  If, for example, the initial surface $\bar M_\theta(0)$ has a ``dumbbell shape,'' then this will happen for all $\theta$ in a subset of the parameter space $\Theta$ with nonempty interior.

We show that degenerate neckpinch type behavior exhibited by any of peanut solutions in consideration is highly \emph{unstable}, in the sense that there exist $\theta'$ arbitrarily close to ${\bf 0}$ for which $\bar M_{\theta'}(t) $ forms a qualitatively different kind of singularity than $\bar M_\theta(t)$.  More precisely, our goal in this paper is to prove the following result, which is well illustrated by figure \ref{fig:peanut} above.

\begin{theorem}\label{thm-main}
Let $\bar{M}_{\bf 0}(t)$ be one of the  $4$-peanut solutions  as discussed above, and let $\bar{T}$ be its first singular time.  There exists a $t_0$ close to $\bar{T}$, so that in every sufficiently small neighborhood of $\bar{M}_{\bf 0}(t_0)$, there exist perturbations $\bar{M}_{\theta_s}(t_0)$ and $\bar{M}_{\theta_c}(t_0)$ with the following property.  The MCF starting at $\bar{M}_{\theta_s}(t_0)$ as its initial data develops a spherical singularity, while at the same time the MCF starting at $\bar{M}_{\theta_c}(t_0)$ as its initial data develops a nondegenerate neckpinch singularity.  Here $\theta_s$ and $\theta_c$ can be chosen arbitrarily small.
\end{theorem}

We will give more precise definition of our families $\bar{M}_{\theta}(t)$ later in the text.  

\smallskip 

In \cite{CHH} the authors showed that the ancient ovals occur as a limit flow of  a closed MCF $\{M_t\}$ if and only if there is a sequence of spherical singularities converging to a cylindrical singularity.  As a corollary of Theorem \ref{thm-main} we show an analogous result for a blow up limit of our families of MCF solutions that can be seen as perturbations of  the peanut solution.  More precisely, we have the following result.  

\begin{theorem}		\label{thm-families}
Appropriately rescaled subsequence of any sequence of solutions which belong to one of our families of solutions, whose initial data converge to the peanut solution, and all of which develop spherical singularities, converges to the Ancient oval solution
constructed in \cite{HH, Wh}. 
\end{theorem}

\medskip

\noindent  {\bf Acknowledgements:} The second author has been supported by the NSF grants DMS
1900702 and DMS 2454018. The third author has been supported by the NSF grants DMS
2105508 and DMS 2505574.

\section{The Peanut Solution}
\label{sec:peanut}

Peanut solutions are central to our story.  In this section we describe these axiomatically by listing the asymptotic properties we assume them to have.  In appendix~\ref{sec:remember the peanut} we recall the construction in \cite{AV}, which showed that $m$-peanuts do exist.  

We begin by restating the relevant coordinates in Space-time and evolution equations for rotationally symmetric Rescaled Mean Curvature Flow from \cite{AV} in our current notation.

\subsection{The outer scale}
We consider families of rotationally symmetric surfaces which, in terms of their profile function $r= U(x, t)$ are given by
\[
\bar M_t =\bigl\{(x, x')\in \R\times\R^{n} : -x_{\max}(t)\leq x\leq x_{\max}(t), \|x'\|=U(x, t)\bigr\}.  
\]
We assume the surfaces are defined throughout the time interval $t_0\leq t <T$, and that  they are reflection symmetric, i.e.
\begin{equation}		\label{eq:everything even}
        U(-x, t) = U(x, t)\quad \text{for all }x, t.  
\end{equation}
The family of hypersurfaces $\bar M_t$ evolves by MCF if
\begin{equation}		\label{eq:hu}
        U_t = \frac{U_{xx}}{1+U_x^2} - \frac{n-1}{U}.  
\end{equation}
\subsection{The inner, parabolic scale}
To describe the possible singularity that forms at $x=0$, $t=T$, we introduce new time and space variables
\begin{equation}
        \tau = - \log (T-t), \quad t
        = T-e^{-\tau}, \quad y = \frac x{\sqrt{T-t}}
\end{equation}
as well as the rescaled profile
\[
u(y, \tau) =e^{\tau/2}U(e^{-\tau/2}x, T-e^{-\tau}),
\]
or, equivalently,
\begin{equation}		\label{eq:cv}
        U(x,t) = \sqrt{T-t} \; u\Bigl( \frac x{\sqrt{T-t}},\log \frac 1{T-t}\Bigr).  
\end{equation}
For the rescaled profile $u(y, \tau)$ \eqref{eq:hu} is equivalent with the Rescaled Mean Curvature Flow equation
\begin{equation}		\label{eq:u}
        u_\tau = \frac{u_{yy}}{1+u_y^2} - \frac y2 u_y - \frac{n-1}u + \frac u2.  
\end{equation}
The cylinder soliton corresponds to the constant solution $u(y, \tau)\equiv \sqrt{2(n-1)}$.  

\subsection{The intermediate scale}
We use an intermediate horizontal scale, whose coordinate is
\begin{equation}		\label{eq:z defined}
        z = e^{-\tau/4}y =\frac{x}{(T-t)^{1/4}}.  
\end{equation}

% \begin{figure}[b]
%         \centering
%         \includegraphics[width=0.75\linewidth]{initialSurface.pdf}       
%         \caption{Asymptotic properties that characterize an $m$-peanut solution}
%         \label{fig:initial surface}
% \end{figure}

\subsection{Defining properties of a $4$-Peanut} \label{sec:peanut properties} The defining estimates for a $4$-peanut solution contain four positive parameters $(\tau_0, \rho, \delta, \ell_{\rm int})$.  We say that a solution $\bu(y, \tau)$ of \eqref{eq:u} is a $4$-peanut solution if it is defined for $\tau\geq \tau_0$ and satisfies the descriptions in \S~\ref{sec:peanut property inner}, \S~\ref{sec:peanut property intermediate}, and \S\ref{sec:peanut property monotone} below.

In \cite{AV} it was shown that if $\delta>0$, $\rho>0$, are small enough, and $\ell_{\rm int}>0$ is large enough, then for all sufficiently large $\tau_0>0$ one can construct a corresponding $4$-peanut solution.  In this paper we will focus on one such peanut solution, and then construct arbitrarily small perturbations of that chosen peanut solution.  At several points in our arguments we have to assume that $\rho>0$ and $\delta>0$ are sufficiently small, and that $\ell_{\rm int}>0$ is sufficiently large.  The existence result in \cite{AV} (see also appendix~\eqref{sec:proof of monotonicity}) allows us to make this assumption, provided we choose $\tau_0$ large enough, depending on the chosen values of $\delta, \rho, \ell_{\rm int}$.

Given the parameter $\rho$ we abbreviate
\begin{equation}		\label{eq:6}
L(\tau) = \rho \, e^{\tau/4}.
\end{equation}

\subsubsection{The inner scale}		\label{sec:peanut property inner}
For all $|y|\leq  2 L(\tau)$, $\tau\geq\tau_0$ one has
\begin{equation}		\label{eq:um2}
        \left| \sqrt{2(n-1)} - K_0e^{-\tau}\hhm_4(y) - \bu(y,\tau) \right|\leq \delta \, (1+|y|)^{4}e^{-\tau} 
\end{equation}
and
\begin{align}		\label{eq:um2-derivs}
        \big|\bu_\tau -  K_0e^{-\tau}&\hhm_4(y)\big| +
        (1+|y|) \left|\bu_y -  K_0e^{-\tau}\hhm_4'(y)\right|  \\
        \notag
        &+ \left|\bu_{yy} - K_0e^{-\tau}\hhm_4''(y)\right|
        \leq  \delta \, (1+|y|)^4 e^{-\tau }
\end{align}

\subsubsection{The intermediate scale}	\label{sec:peanut property intermediate}
For all $|y|\geq \ell_{\rm int}, \tau\geq\tau_0$ one has
\begin{equation}		\label{eq:u-outer}
        \sqrt{2(n-1) - (\tilde K_0+\delta)e^{-\tau} y^4}
        \leq \bu(y, \tau)
        \leq \sqrt{2(n-1) - (\tilde K_0-\delta)e^{-\tau} y^4}
\end{equation}
where we abbreviate $\tilde K_0 = 2\sqrt{2(n-1)} \, K_0$. 

These inequalities imply that for some constant $C$ one has
\begin{equation}
        \big|\bu(y, \tau)-\sqrt{2(n-1)} + K_0 e^{-\tau}y^4\big|
        \le C\delta e^{-\tau}y^4
\end{equation}
whenever $\tau\geq \tau_0$ and $\ell_{\rm int}\leq |y|\leq 2L(\tau)$.

The inequalities \eqref{eq:u-outer} also imply that  the  location $y_{\max}(\tau)$ of the tip of the peanut solution is bounded by
\[
(\tilde K_0 +\delta)^{-1/4} e^{\tau/4}
\leq y_{\max}(\tau) \leq
(\tilde K_0 -\delta)^{-1/4} e^{\tau/4}.
\]
\subsubsection{Monotonicity of the ends of the peanut}			\label{sec:peanut property monotone}
In Appendix~\ref{sec:proof of monotonicity} we show that, with the right choice of initial peanut, the construction in \cite{AV} yields a  solution that satisfies the following monotonicity.

For all $y\geq \ell_{\rm int}$ and $\tau\geq \tau_0$ we have
\begin{equation}		\label{eq:utau positive}
        \bu_\tau(y, \tau) > 0.
\end{equation}
We will use this fact to construct barriers that control perturbations of the peanut outside the parabolic region.  

\section{Set up and outline of the proof of Theorem \ref{thm-main}}	\label{sec-outline}

The goal in this section is to describe the set up, explain our choice of constants and outline the main steps in the proof of Theorem \ref{thm-main}.

We choose positive constants $K_0, \rho, \delta, \ell_{\rm int}, \tau_0$ and consider a peanut solution  $\bu(y, \tau)$ as in the previous section.

To introduce a family of perturbations of $\bu(y,\tau_0)$ in the direction of the lower eigenfunctions $\hm_0, \hm_2$ we let $\eta_0:\R\to\R$ be a smooth even cutoff function satisfying
\begin{equation}		\label{eqn-def-eta0}
\eta_0(y)=
\begin{cases}
0 & \text{for } |y|\geq 2\\
1 & \text{for }|y|\leq 1,
\end{cases}
\end{equation}
and we choose a length $\ell_0$ with
\begin{equation}		\label{eq:1}
\ell_{\rm int}<2\ell_0 \ll   L(\tau_0)
\end{equation}
For any given $\epsilon > 0$ and $\bom= (\Omega_0,\Omega_2) \in \mathbb{S}^1$ we then define
\begin{equation}		\label{eq-initial}
\ueo(y,\tau_0)
	= \bu(y,\tau_0)
	+  \epsilon\, \eta_0\Bigl(\frac{y}{\ell_0}\Bigr)\bigl\{\Omega_0 \hm_0(y) + \Omega_2 \hm_{2}(y)\bigr\}.
\end{equation}
Let $\ueo(y,\tau)$ be the rescaled mean curvature flow solution starting at $\ueo(y,\tau_0)$.  If there is no confusion, we abbreviate $\ueo(y,\tau_0)$ and $\ueo(y,\tau)$ to $u(y,\tau_0)$ and $u(y,\tau)$, respectively.

We will analyze the difference between the peanut solution $\bu(y, \tau)$ and the perturbed solution $u_{\epsilon, \bf \Omega}$.  Define
\begin{equation}		\label{eqn-defnW}
w(y,\tau) := u(y,\tau) - \bu(y,\tau)\quad \text{and } W(y,\tau) := w(y,\tau) \, \eta(y,\tau)
\end{equation}
We refer to $W(y, \tau)$as the \emph{truncated difference} of $u$ and $\bu$.  The time dependent cutoff function $\eta$ is defined by
\begin{equation}		\label{eq:eta y tau def}
\eta(y,\tau) := \eta_0\left(\frac{y}{L(\tau)}\right).
\end{equation}
Since the initial perturbed surface coincides with the peanut $\bu(y, \tau_0)$ when $y\geq 2\ell_0$, short time existence for MCF implies that $\ueo(y, \tau)$ is defined on some time interval $[\tau_0, \tau_{\epsilon,\bom})$ for some $\tau_{\epsilon,\bom}>\tau_0$.  The tip of the perturbed solution is located at $y_{\max,\epsilon,\bom}(\tau)$.  In the following sections we follow the perturbed solution $u(y, \tau)$ until some time $\tau'$ depending on $\epsilon, \bom$, and it will follow from our barrier arguments that $y_{\max, \epsilon,\bom}(\tau) > 2L(\tau)$ for all~$\tau\leq \tau'$.  This implies that the truncated difference $W(y, \tau)$ is well defined by \eqref{eqn-defnW} for all~$y\in\R$.

\subsection{Evolution of the difference $w=u-\bu$}

Since both $\bu$ and $u=\ueo$ satisfy equation~\eqref{eq:u}, a direct computation shows that $w$ satisfies
\[
  w_\tau =  w_{yy}  -\frac y2 w_y +  w
  -\frac{u_y^2}{1+ u_y^2} w_{yy}
  -  \frac{(u_y+\bu_y)\bu_{yy}}{(1+\bu_y^2)(1+u_y^2)} w_y
  +\frac{2(n-1)-u\bu}{2\bu u}w,
\]
which we write as
\begin{equation}		\label{eqn-w}
w_\tau= \cL w +  \cE(w) 
\end{equation}
where $\cL$ is the drift Laplacian from \eqref{eq:L}, and 
\begin{equation}		\label{eqn-E1}
\begin{split}
\cE (w)  :&=    - \frac{u_y^2}{1+u_y^2} \, w_{yy} -  \frac{(u_y+\bu_y)\bu_{yy}}{(1+\bu_y^2)(1+u_y^2)}\, w_y +    \frac{2(n-1)-u\bu}{2u \bu} \, w\\
&=  c_2(y, \tau)\, w_{yy} +c_1(y, \tau)\, w_y+c_0(y, \tau)\, w.  
\end{split}
\end{equation}
Thus, the equation of $w$ may be expressed in the form 
\begin{equation}		\label{eqn-w2}
w_\tau= (1+ c_2(y, \tau))\, w_{yy} + \Bigl(- \frac y2 + c_1(y,\tau)\Bigr) \, w_y + (1+ c_0(y, \tau)) \, w.  
\end{equation}

The function  $W(y,\tau)$ satisfies
\begin{equation}		\label{eqn-W}
W_\tau =  \cL \, W +  \eta \, \cE(w)  + \cE(w,\eta)
\end{equation}
where $\cE(w)$ is as in \eqref{eqn-E1} and where $ \cE(w,\eta)$ is the error term coming from commuting the cutoff function with $\partial_\tau-\cL$:
\begin{equation}		\label{eqn-E2} 
\cE(w,\eta) = \mu_1  w + \mu_2   w_y, \quad \mu_1(\eta) = \eta_\tau - \eta_{yy} - \frac y2 \, \eta_y, \quad    \mu_2(\eta) =- 2  \eta_y.  
\end{equation}

\subsection{Outline of arguments that are common in both cases, finding  spherical and cylindrical singularities.  }
\label{ss-HHH}
\label{sec-common}
It is well known that the drift Laplacian $\cL$ is a self adjoint operator on the Hilbert space $\hilb=L^2(\R; e^{-y^2/4}dy)$ with discrete spectrum, and whose eigenfunctions are Hermite polynomials (see Appendix~\ref{sec:appendix hilbert spaces}).  
The space  $\hilb:= L^2\big(\mathbb{R}, e^{-y^2/4}\, dy\big)$ is a Hilbert space with respect to the norm and inner product
\[
\|f\|^2 = \int_{\mathbb R} f(y)^2 e^{-y^2/4}\, dy, \qquad \langle f, g \rangle = \int_{\mathbb R} f(y) g(y) e^{-y^2/4}\, dy.
\]
To facilitate future notation, we  define yet another Hilbert space $\hv$ by
\begin{equation}		\label{eqn-hv}
\hv = \{f\in \hilb \,\,\, :\,\,\, f, f_y \in \hilb\},
\end{equation}
equipped with a norm
\[
\|f\|^2_\hv = \int_{\mathbb{R}} \{f(y)^2 + f'(y)^2\| e^{-y^2/4}\, dy.
\]
If it were not for the error terms $\eta\cE(w)$ and $\cE(w, \eta)$, we could solve equation \eqref{eqn-W} for the truncated difference $W$ in terms of eigenfunctions of the drift Laplacian $\cL$.  Unfortunately, one cannot ignore $\eta\cE (w) + \cE(w, \eta)$ without further justification, and most of the analysis in this paper is meant to provide such justification. 

We deal with the error terms by means of an ``inner-outer estimate'' for the non-truncated difference $w$ in the transition region $L(\tau)\leq y\leq 2L(\tau)$ in terms of the size of $w$ at some fixed point $y = \ell$ in the inner region.  Our proof of the inner-outer estimate relies on the monotonicity of the peanut solution (see \S~\ref{sec:peanut property monotone}) to use different time translates of the Peanut as barriers for the perturbed solutions.  Because of this approach we need to assume the peanut solution $\bu(y, \tau)$ exists during some time interval $[\tau_0-N, \tau_0]$ \emph{before} the initial time $\tau_0$.

We decompose the truncated difference $W$ into its unstable and stable components, i.e.~we write
\begin{equation}		\label{eq-stable-unstable}
W=W^u+W^s, \quad\text{with}\quad
W^u := \pi^u (W), \quad W^s := \pi^s ( W)
\end{equation}
where $\pi^u, \pi^s$ are the projections of $\hilb$ onto the invariant subspaces
\[
\hilb^u = \mathrm{span}\{\hhm_0, \hhm_2\},\qquad
\hilb^s = \overline{\mathrm{span}\{\hhm_4, \hhm_6, \hhm_8, \dots\}}
\]
corresponding to the eigenvalues $\lambda_0, \lambda_2$ and $\lambda_4, \lambda_6, \dots $ of $\cL$ respectively.  More specifically,
\[
\pi^u\phi := \sum_{j=0}^1 \frac{\langle \hhm_{2j}, \phi\rangle}{\|\hhm_{2j}\|^2} \hhm_{2j},\qquad
\pi^s\phi := \phi-\pi^u\phi = \sum_{j=2}^\infty \frac{\langle \hhm_{2j}, \phi\rangle}{\|\hhm_{2j}\|^2} \hhm_{2j}.
\]

Combining the definitions \eqref{eqn-defnW} and~\eqref{eq-initial} of $w$ and $W$ we find
\[
W(y, \tau_0) = \epsilon\eta_0\Bigl(\frac{y}{L(\tau)}\Bigr) \eta_0\Bigl(\frac{y}{\ell_0}\Bigr) \bigl\{\Omega_0 \hhm_0(y) + \Omega_2\hhm_2(y)\bigr\}
\]
so we may regard the initial perturbation as a truncated linear combination of $\hhm_0, \hhm_2$.  The following lemma shows that $W(\cdot, \tau_0)$ almost lies in $\hilb^u$ in the sense that $W^s$ is much smaller than $W^u$.
\begin{lemma}		\label{lemma-lo}
There is a constant $C$ such that for any $\ell_0\in (\ell_{\rm int}, \tfrac12 L(\tau_0))$ and for all  $\epsilon>0$ and $\bf\Omega\in\mathbb S^1$ one has
\[
\|W^s(\cdot, \tau_0)\|  <   C\ell_0^{3/2}e^{-\ell_0^2/8}\,  \|W^u(\cdot, \tau_0)\| .  
\]
\end{lemma}
\begin{proof} The estimate is homogeneous in $W$, so we may assume $\epsilon=1$.

In \eqref{eq:1} we chose $\ell_0$ so that $2\ell_0 < L(\tau_0)$.  Since $\eta_0(s)=0$ for $s\geq 2$ and $\eta_0(s)=1$ for $s\leq 1$ we have $\eta_0 (y/\ell_0) = 1$ whenever $y\leq L(\tau_0)$.  Hence, with $\epsilon=1$,
\[
W(y,\tau_0) =  \eta_0(y/\ell_0)\,\big(\Omega_0 \hm_0(y)+\Omega_2\hm_{2}(y)\big).
\]
This implies that there is a $c_0>0$ such that
\begin{equation}		\label{eq:2}
\|W^u(\cdot, \tau_0)\| ^2  \geq \int_0^L |\Omega_0\hhm_0(y)+\Omega_2\hhm_2(y)|^2 e^{-y^2/4}dy \geq c_0>0,
\end{equation}
where $c_0$ does not depend on $\bom\in\mathbb S^1$.

Let $h(y)=  \Omega_0 \hm_0(y) + \Omega_2\hm_{2}(y)$.  Then $\pi^s h=0$, and hence
\[
\|W^s(\cdot, \tau_0)\|  = \|W^s(\cdot, \tau_0) - \pi^sh \|  = \|\pi^s(W(\cdot, \tau_0)-h)\|
\leq \|W(\cdot, \tau_0) - h\|  
\]
Since
\[
W(y, \tau_0)-h(y) =  (1-\eta_0(y/\ell_0)) \,\Bigl(\Omega_0\hm_0(y) + \Omega_2\hm_{2}(y)\Bigr)
\]
we have
\[
\| W(\cdot, \tau_0)-h\| ^2
\leq  \int_{\ell_0}^\infty \bigl(\Omega_0\hm_0(y) + \Omega_2\hm_{2}(y)\bigr)^2 \, e^{-y^2/4}\, dy.
\]
In view of $\Omega_0^2+\Omega_2^2=1$ and $|\hhm_j(y)| \lesssim y^{2j}$ we have for $y\geq 1$
\[
\bigl(\Omega_0\hhm_0(y) + \Omega_2\hhm_2(y)\bigr)^2 \lesssim C y^4.
\]
Therefore
\[
\| W(\cdot, \tau_0)-h\| ^2
\leq C  \int_{\ell_0}^\infty y^4 \, e^{-y^2/4}\, dy \leq C \ell_0^3 e^{-\ell_0^2/4}.
\]
Together with the lower bound $\|W^u(\cdot, \tau_0)\|^2\geq c_0$ this completes the proof.
\end{proof}

The $\hhm_4$ component of the peanut solution $\bu(y, \tau)$ is $-K_0 e^{-\tau}\hhm_4(y)$.  This term determines the shape of the peanut solution in the parabolic region, i.e.~for bounded $y$.  We will show that the perturbation $w(y, \tau)$ grows and  becomes large compared to $K_0 e^{-\tau}\hhm_4(y)$ at some time $\tau_1$, depending on the initial perturbation parameters~$\epsilon, \bom$.  The following definition contains a constant $M_1$, which we will specify later in Section~\ref{sec-sphere}.
\begin{definition} [The first exit time]
For any $\epsilon\geq 0$ and $\bom\in\mathbb S^1$we let $\tau_1(\epsilon,\bom)\in (\tau_0, \infty]$ be the maximal time for which
\begin{enumerate}[{\upshape (a)}]
\item the perturbed solution $\ueo$ is defined for all $\tau\in  [\tau_0, \tau_1(\epsilon,\bom))$ and $|y|\leq 2L(\tau)$
\item for all $\tau\in[\tau_0, \tau_1(\epsilon, \bom))$ the unstable component $W^u(\tau) = \pi^u W_{\epsilon,\bom}(\cdot, \tau)$ satisfies
\begin{equation}		\label{eqn-dfnfun}
\|W^u\| < M_1 e^{-\tau}
\end{equation}

\end{enumerate}
\end{definition}
\bigskip

We often abbreviate $\tau_1(\epsilon,\bom) = \tau_1$.

\begin{proposition} There is an $\bar\epsilon>0$ such that $\tau_1(\epsilon,\bom)<\infty$ for all $\epsilon\in(0, \bar\epsilon)$ and $\bom\in\mathbb{S}^1$.  Moreover, $W^u_{\epsilon,\bom}(\cdot, \tau)$ is still defined at $\tau=\tau_1$, at which time it satisfies
\begin{equation}		\label{eq:3}
\|W^u_{\epsilon,\bom}(\cdot, \tau_1)\|=M_1 e^{-\tau_1}.
\end{equation}
and
\begin{equation}		\label{eq:4}
\left(\frac{d}  {d\tau}e^\tau\|W^u(\tau)\|\right)_{\tau=\tau_1}>0.
\end{equation}
\end{proposition}
\begin{proposition} The function $(\epsilon,\bom)\mapsto \tau_1(\epsilon, \bom)$ is continuous on $(0, \bar\epsilon)\times\mathbb S^1$.
\end{proposition}

\begin{proposition} For all $(\epsilon, \bom) \in (0, \bar\epsilon)\times\mathbb S^1$ and $\tau\in[\tau_0, \tau_1(\epsilon,\bom))$ one has $\|W^u_{\epsilon,\bom}(\cdot, \tau)\| > 0$.
\end{proposition}

This implies that the map $H^u$
\[
H^u(\epsilon, \bom, \tau) := \frac{W^u_{\epsilon,\bom}(\cdot, \tau)}{\|W^u_{\epsilon,\bom}(\cdot, \tau)\|}
\in \mathbb S^1_\hilb
\]
is continuous, where $\mathbb S^1_\hilb\subset\hilb^u$ is the unit circle in $\hilb^u$.  The map is defined for all $\epsilon\in(0, \bar\epsilon)$, $\bom\in\mathbb S^1$, and $\tau\in [\tau_0, \tau_1(\epsilon, \bom)]$.

\begin{proposition} For all $\epsilon\in(0, \bar\epsilon)$ the map $\mathscr H^u_\epsilon:\mathbb S^1 \to \mathbb S^1_\hilb$ defined by $\mathscr H^u_{\epsilon}(\bom) = H^u(\epsilon, \bom, \tau_1(\epsilon, \bom))$ is surjective.
\end{proposition}
%\[
%\star\quad\star\quad\star
%\]
%%%% CONTINUE HERE %%%%%%%%%%%%%%%%%%%%%%%%%%%%%%%%%%%%%%%%%%%%%%%%%% 

\begin{definition}[The Funnel]		\label{def-funnel}
For any $\tau>\tau_0$ we consider the set 
\[
\cF_\tau = \{ \Gamma\subset\R^{n+1} | \Gamma\text{ satisfies \eqref{eq:Gamma smooth rot surface} and \eqref{eqn-dfnfun}}\}
\]
Here the first condition is
\begin{equation}		\label{eq:Gamma smooth rot surface}
\parbox[c]{0.8\textwidth}{$\Gamma$ is a smooth rotationally symmetric hypersurface in $\R^{n+1}$ whose
  profile function $y\mapsto u(y)$ is defined for $|y|\leq 2L(\tau)$.}
\end{equation}
This condition implies that the difference $w(y)=u(y)-\bu(y, \tau)$ and truncated difference $W(y)= \eta(y, \tau)w(y)$ are defined, which allows us to state the second condition: namely, the unstable component $W^u=\pi^u W$ is bounded by
\begin{equation}	
\|W^u(\cdot, \tau)\|_\hilb\leq M_1 e^{-\tau}.
\end{equation}
The constant $M_1$ will be chosen later on in Section \ref{sec-sphere}.
\end{definition}

To prove instability of the peanut solution, we show that for arbitrarily small $\epsilon>0$ there exist $\bom_n, \bom_c\in\mathbb S^1$ such that the solutions starting from either  $u_{\epsilon,\bom_n}$or $u_{\epsilon,\bom_c}$ moves away from the peanut, and in the case of $\bom_n$ forms a generic neck pinch, or, in the case of $\bom_c$ becomes convex before becoming singular. Since we will prove this for sufficiently small $\epsilon$, we may assume that
\begin{equation}		\label{eqn-epsilon}
0 <  \epsilon \ll   e^{- \tau_0}.  
\end{equation}
Under this assumption~\eqref{eq:um2} and \eqref{eq-initial} imply that at time $\tau_0$, the projection of $u(y,\tau_0) - \sqrt{2(n-1)}$ onto the unstable modes $\hm_0(y)$ and $\hm_2(y)$ is negligible compared to the projection of $u(y,\tau_0) - \sqrt{2(n-1)}$ on $\hm_4(y)$.  The linear part of equation~\eqref{eq:u} suggests that the $\hilb^s$ component will decay at least as fast as $e^{-\tau+o(\tau)}$, while the $\hilb^u$-component will not decay faster than $e^{o(\tau)}$

One of our first goals is to make sure that we can find the time at which the projection of $u(y,\tau) - \sqrt{2(n-1)}$ onto unstable modes will start catching up and hopefully start dominating the other projections.

Note that \eqref{eqn-epsilon} implies that all MCF solutions with initial data \eqref{eq-initial} belong to $\mc F_{\tau_0,\epsilon}$, so it makes sense to talk about our \emph{first goal  } we want to achieve, that is, to find the \emph{exit time}  for each of our solutions,  as described below.  

Our \emph{first goal}  is to show that  there exist a small $\epsilon > 0$, and $\tau_0 \gg 1$ so that for every $\bom \in \mathbb{S}^{1}$, every MCF solution with initial data \eqref{eq-initial} has the following property: there exists a time $\tau_1 = \tau_1(\epsilon, \bom)$, at which the solution hits the boundary of the funnel $\mc F_{\tau_1,\epsilon}$, meaning that $\| W^u(\cdot, \tau_1)\|  = M_1 e^{-\tau_1}$, where $M_1$ is a uniform constant, independent of $\epsilon$ and ${\bf \Omega}$.  We call this time $\tau_1$ the {\it exit time} for our solution.  
Note also that $\tau_1 = \tau_1(\epsilon, {\bf \Omega})$ is a continuous function of $\epsilon$ and $\bom$,  which follows by the continuous dependence of the MCF of initial data, and by the ``exit condition,'' which in this case says that $\frac{d}{d\tau}e^{\tau}\|W^u(\cdot, \tau)\| >0$ at $\tau=\tau_1$.

In order to prove the existence of the exit time, we need to look at the linearized equation satisfied by $W(y,\tau)$, around the round cylinder.  There will be two types of error terms in the equation for $W$, the ones that are roughly speaking, of quadratic nature, as in \eqref{eqn-E1}, and the others coming from the cut off functions, as in \eqref{eqn-E2}.  In order to deal with the errors coming from cut-off functions we need: (i)  the  inner-outer estimate
shown in Proposition \ref{prop-inout}, and (ii)   the $L^{\infty}$ estimates on $w(y,\tau)$ and its derivatives holding on large, time dependent sets.  We prove both (i)-(ii)  \emph{ simultaneously}, that is, we show that as long as our solution stays in the funnel, and as long as we have the  $L^{\infty}$ estimates with an auxiliary constant on our solution and its derivatives, the inner-outer estimate holds, and vice versa as long  as our solution stays in the funnel, and as long as the inner-outer estimate holds, our solution and its derivatives satisfy the $L^{\infty}$ estimates with sharper constants than what we have used in the previous step.  The fact that we get sharper constants in the latter step, enables us to run the argument which shows that as long as our solution stays in the funnel, both, (i) and (ii) hold simultaneously.  

In the course of proving (i) and (ii) above,  we also prove  that as long as our solution stays in the funnel $\mathcal F_{\tau,\epsilon}$, the unstable projection $\|W^u(\cdot, \tau)\| $ is much bigger than the stable projection $\| W^s(\cdot, \tau)\|$, for all $\tau \ge \tau_0$.  We actually show in Lemma~\ref{cor-prop-funnel} that 
\begin{equation}		\label{eqn-deltaW}
\| W^s(\cdot, \tau)\|  \le e^{-\ell_0^2/8} \, \| W^u(\cdot, \tau)\|
\end{equation}
for $\tau \ge \tau_0$.

Once we know that for every solution  we can find its exit time, we employ degree theory arguments to show that for every $\epsilon > 0$ small and every ${\bf\bar{\Omega}}\in \mathbb{S}^1$ there exists an $\bom\in \mathbb{S}^1$, so that 
\begin{equation}		\label{eqn-exit-omega}
W_{\epsilon,\bom}^u  = M_1 \,e^{-\tau_1} \big(\bar{\Omega}_0\, \hm_0 + \bar{\Omega}_{2} \, \hm_{2}\big).  
\end{equation}

\subsection{Case of spherical singularities}
\label{sec-case-sphere}
To show that in every small neighborhood of peanut solution we can find initial data whose mean curvature flow develops spherical singularities,  we choose  in  our discussion above,  ${\bf \bar{\Omega}} = (\Omega_0,\Omega_2) \in \mathbb{S}^{1}$, so that $\Omega_2 < 0$ and $\Omega_0^2 + \Omega_2^2 =1$ (that is, we have \eqref{eqn-exit-omega} with such a choice of ${\bf \bar{\Omega}}$).  In this case we show our solution at time $\tau_1$ is actually convex, and hence Huisken's result about convex solutions implies that our solution develops spherical type of singularity.  It will be clear from the proof that in order to prove   convexity  at time $\tau_1$, in a set where $\big(\hm_4(y)\big)_{yy}$ is possibly negative, we need to pick sufficiently big $M_1$ (it will be clear from the proof that we will need $-K_0\min_{|y| \le 2\ell_0}(\hm_4)_{yy}(y)- M_1/2 < 0$).  

\subsection{Case of cylindrical singularities}\label{sec-case-cyl}
 We will next outline the additional steps that are needed in order to show that for any $\epsilon >0$  we can pick initial data of the form \eqref{eq-initial} whose mean curvature flow   develops  a  nondegenerate neckpinch singularity.  

In the  definition of the funnel in \eqref{eqn-dfnfun} we can choose the same $M_1$, as in the spherical case.  
Let $\tau_1$ be as before the exit time, meaning that $\|W^u(\cdot,\tau_1)\| = M_1 e^{-\tau_1}$.  As a result we get that
\[
w(y, \tau) =  d_0(\tau)\, \hm_0(y) \, e^{- \tau} + d_2(\tau)\, \hm_2(y)\, e^{-\tau} +  o_{\tau} ( e^{-\tau}),
\]
on a compact set $|y| \le \ell$, $\tau\in [\tau_0, \tau_1]$, where $  |d_{0}(\tau)| + |d_2(\tau)| \leq C(M_1)$.

A new important ingredient in the cylindrical case versus the spherical case is a subtle construction of another family of barriers from below and above for $q(y,\tau) = u^2(y,\tau) - 2(n-1)$, whose purpose is to be able to guarantee that after  enough time has elapsed, the neutral mode $\hm_2(y)$ would start dominating in the asymptotic expansion of our solution around the singularity.  As we will describe in more detail below, these barriers provide a good asymptotic description of our entire surface up to some large time $\tau_2 \gg  \tau_1$
 enabling  us to then pass from below and above  some rough barriers at time $\tau_2$, and use the avoidance principle to conclude that our flow will develop a singularity at the origin, which will split the hypersurface into two disjoint parts, none of which disappears at the singular time.  This will imply the singularity is a nondegenerate neckpinch.  

To be more precise, call the supersolutions and subsolutions that we find in section \ref{sec-sub-super}, for $|y| \ge \ell_1$ and for $\tau \le\tau_2$, $\cQ^+_{\varepsilon,K}(y,\tau)$ and $\cQ^-_{\varepsilon, K}(y,\tau)$, respectively.  Here $\ell_1 > 0$ is a fixed constant and it will turn out that $\ell_0 \le \ell_1 \le 1000 \ell_0$, and $\ell_0$ is the constant that defines the support of the cut off function $\eta_0(y)$ that appears in the definition of our perturbations at time $\tau_0$, see \eqref{eq-initial}).  We would like to show they are actually the upper and the lower barriers for our solution, outside a large set $|y| \ge \ell_1$, for $\tau\in [\tau_1,\tau_2]$,  and for that we need to use the maximum principle with boundary.  In order to do that, we need to be able to compare our solution $q(y,\tau_1)$ with $\cQ^+_{\varepsilon,K}(y,\tau_1)$ and $\cQ^-_{\varepsilon, K}(y,\tau_1)$, for $|y|\ge \ell_1$.  In order to be able to do that we first use  the translates in time of the peanut solution as barriers from $\tau_0$ to $\tau_1$, as we did in section \ref{sec-inner-outer}.  This yields a good asymptotic description of our entire surface at time $\tau_1$, outside a large set $|y| \ge \ell_1$.  After that we use $\cQ^+_{\varepsilon, K}(y,\tau)$ and $\cQ^-_{\varepsilon,K}(y,\tau)$ as barriers from $\tau_1$ to $\tau_2 \gg \tau_1$, where $\varepsilon = M_1 \sqrt{2(n-1)}\, e^{-\tau_1}$.  The goal of these barriers is to show that our solution exists up to some time $\tau_2 \gg \tau_1$, and that thanks to the barriers we have a good asymptotic description of our solution at time $\tau_2$.  More precisely, combining the barriers and the $L^2$ theory we can finally guarantee that the neutral mode dominates over all other modes on a large parabolic neighborhood, implying that the maximum of  our  profile function $u_{\max} ( \cdot, \tau_2) $ can be made a  large absolute constant depending only on dimension $n$.

This fact and our barriers  then allow us to put the Angenent torus as an outer  barrier around the origin (where the radius of our solution $u(y,\tau_2)$ is close to $\sqrt{2(n-1)}$) 
  and simultaneously  large balls of radius $R$  inside of our solution, on both sides of the origin,  starting at time $\tau_2$.  The Angenent torus will shrink in a time  $T_a(n)$ that depends only on the dimension $n$, and hence by the avoidance principle for the MCF our solution needs to develop a singularity at time $T \le T_a(n)$.  On the other hand,  our subtle barriers show that at time $\tau_2$  we can place balls of  large radius $R$ (depending only on dimension $n$) inside of our solution on both sides of the origin (where the singularity happens).  The balls take time $\frac{R^2}{2n}$ to shrink to a point, and hence if we choose $R$ sufficiently big so that $\frac{R^2}{2n} > T_a(n)$, we will know that at the singular time our surface pinches off at the point and disconnects into two parts, none of which disappears at the singular time.  

The supersolutions and subsolutions $\cQ^+_{\varepsilon,K}(y,\tau)$ and $\cQ^-_{\varepsilon, K}(y,\tau)$, respectively, are constructed only outside a large compact set, hence to show they are the actual barriers  for our solution we need the comparison principle with the boundary, and for that we need to  show the right behavior of our solution at the boundary  $|y| = \ell_1$, for all times $\tau \in [\tau_1, \tau_2]$.  To show we have the right behavior of our solution on the boundary $|y| = \ell_1$, for all $\tau\in [\tau_1,\tau_2]$, we combine the $L^2$ theory for projections of our solution on different eigenspaces of linearlized operator around the cylinder, and  we also need to have good $L^{\infty}$ estimates on large sets in order to control the error terms.  More precisely we proceed as follows.  

First we assume we have all $L^{\infty}$ estimates that we need, on a large set whose size depends on time exponentially, with some auxiliary constant.  Then we employ suitable  $L^2$ arguments for the projections of $q(y,\tau): =u^2(y,\tau) - 2(n-1)$ onto stable, neutral and unstable modes of the linearized mean curvature operator around the cylinder to be able to say that at least for some large times the neutral mode starts dominating, which then yields the right behavior of our solution on a set $|y| \le \ell$.  In these $L^2$ arguments, we need the $L^{\infty}$ estimates on the solution and its derivatives, to be able to control the errors coming form cut off functions and estimate them by exponentially small terms, which become negligible in our analysis.  

In the second step, assuming that we have the right behavior of our solution at the boundary $|y| = \ell_1$, and that we have the upper and lower barriers for $|y| \ge \ell_1$, as described above, we show we have the $L^{\infty}$ estimates on the solution and its derivatives, with a sharper constant than the one we used in the $L^{\infty}$ estimates mentioned in the previous paragraph.  The sharper constant depends only on $\delta$, $n$ and $\ell_0$.  The fact we can get the  $L^{\infty}$ estimate  with the sharper constant allows us to close the argument, and tell us we have all we need in order to apply the maximum principle with the boundary.  This yields  the supersolutions and subsolutions we construct in section \ref{sec-sub-super}  are indeed the barriers for quite a long time, as we have wanted.  In turn, the precise asymptotics on compact sets, and the sharp $L^{\infty}$ estimates hold for quite a long time as well.  This has been shown in Proposition \ref{prop-barriers}.  

\section{Inner-Outer $L^2$ estimate}
\label{sec-inner-outer}

In order to be able work with ODEs for different projections of a difference of the perturbed solution and the peanut solution itself on eigenspaces of the linearized mean curvature flow operator around the round cylinder, we need first to prove  an inner-outer estimate.  It is needed in order to deal with the errors coming from the cut off functions in our ODE arguments for $L^2$ norms of the projections.  The goal in this section is to prove desired inner-outer estimate by using the translates in time of peanut solution itself as upper and lower barriers.  

We will assume throughout this section that for some auxiliary constant $\Lambda$, the perturbed solution $\ueo(\cdot, \tau)$ satisfies the $L^\infty$-bounds
\begin{equation}		\label{eqn-u-der-bounds}
|u(y,\tau) - \sqrt{2(n-1)} |+ |u_y(y,\tau)| + | u_{yy} (y,\tau)| \leq \Lambda\, e^{- \tau} (1+|y|^4)
\end{equation}
for all $ |y| \leq 2 L(\tau)$, and all $\tau\in[\tau_0, \tau_1(\epsilon, \bom)]$.  This estimate is analogous to the properties \eqref{eq:um2}, \eqref{eq:um2-derivs} for the peanut solution $\bu$.
This assumption will be removed in section \ref{sec-main-thm} (see Proposition \ref{prop-barriers}) where it will be shown that \eqref{eqn-u-der-bounds} holds with $\Lambda$ replaced by a constant that depends only on $M_1$ and the initial data. 

Our goal is to prove the following inner-outer $L^2$-estimate.  Recall that $L(\tau)= \rho e^{\tau/4}$, that our initial perturbation $w(y, \tau_0)$ is supported in the interval $|y|\leq 2\ell_0$ (see \eqref{eq-initial}), and that $\tau_0$ will be taken sufficiently large. Also, recall the $L^2$ norm with respect to the Gaussian weight defined
in subsection \ref{ss-HHH} and for any numbers $a < b$ we define 
\[
\| f \|_{\hilb[a, b]} = \int_a^b f(y)^2 \, e^{-y^2/4}\, dy.
\]

\begin{proposition}		\label{prop-inout} 
Assume that for all $y\in[\ell_0 , 1000 \ell_0]$ and $\tau \in [\tau_0, \tau_1(\epsilon,\bom)]$ one has
\begin{equation}		\label{eqn-uuu} 
        \frac 23   K_0 \, y^4 \, e^{-\tau} \leq     \sqrt{2(n-1)} -  \ueo(y,\tau) \leq  \frac 32  K_0 \, y^4 \, e^{-\tau}.  
\end{equation}
Then, for any $\tau \in [ \tau_0, \tau_1(\epsilon,\bom)] $,
\begin{equation}		\label{eqn-inout2} 
        \| e^\tau w(y, \tau)\|_{\hilb[L(\tau), 2L(\tau)]}^2 
        \leq C\,  e^{-\frac{1}{16}L(\tau)^2}\int_{\tau_0}^\tau \| e^{\tau'} w(y, \tau')\|_{\hilb[0, 4\ell_0]}^2   d\tau' 
\end{equation}
\end{proposition}  

We first use the time translates $\bu(y, \tau+\alpha)$, $\bu(y, \tau-\alpha)$, $\alpha >0$   of the peanut solution as barriers from  above and below, 
which provides a pointwise estimate for $w(y, \tau)$ in the interval $L(\tau)\leq y\leq 2L(\tau)$ in terms of the values of $w$ at $y=2\ell_0$ at all previous times $\tau'\in[\tau_0, \tau]$.  Then we use the $\hv$ norm of $w(y, \tau')$ on the interval $[\ell_0, 2\ell_0]$ to bound $w(2\ell_0, \tau')$.  Finally, a standard parabolic regularity argument allows us to bound the $\hv$ norm of $w$ in terms of its $\hilb$ norm.  Combining these steps then leads to the estimate~\eqref{eqn-inout2} above.

\subsection{A pointwise estimate via the maximum principle}
\begin{lemma}		\label{lem:alpha}
Suppose that for some $\tau\in[\tau_0, \tau_1(\epsilon, \bom)]$ there is an $\alpha\in(0, \ln 2)$ such that 
\begin{equation}		\label{eq:alpha definition} 
  \bu(2\ell_0, \tau'-\alpha)\leq \ueo(2\ell_0, \tau') \leq \bu(2\ell_0, \tau'+\alpha)
\end{equation}
for all $\tau'\in[\tau_0, \tau]$, and let $\alpha(\tau)$ be the smallest $\alpha\in(0, \ln 2)$ with this property.  Then
\begin{equation}		\label{eqn-upm2}
\bu(y, \tau-\alpha(\tau)) \leq \ueo(y, \tau) \leq \bu(y, \tau+\alpha(\tau)).
\end{equation}
and
\begin{equation}		\label{eq:5}
\bu(y, \tau-\alpha(\tau))-\bu(y, \tau) \leq w(y, \tau) \leq \bu(y, \tau+\alpha(\tau))-\bu(y, \tau)
\end{equation}
for all $y\geq 2\ell_0$.
\end{lemma}
\begin{proof}
By definition $w(y, \tau) = \ueo(y, \tau)-\bu(y,\tau)$, so the inequalities~\eqref{eqn-upm2} and~\eqref{eq:5} are equivalent.

Since $\bu(y, \tau)$ is a continuous and strictly increasing function of $\tau$ there is a smallest $\alpha=\alpha(\tau)$ for which \eqref{eq:alpha definition} holds for all $\tau'\in[\tau_0, \tau]$.  We claim that for this $\alpha$ the maximum principle implies
\begin{equation}		\label{eq:8}
\bu(y, \tau'-\alpha) \leq \ueo(y, \tau') \leq \bu(y, \tau'+\alpha)
\end{equation}
for all $y\geq 2\ell_0$ and $\tau'\in[\tau_0,\tau]$.

Indeed, $\bu(y, \tau\pm\alpha)$ and $\ueo(y, \tau)$ are both solutions of \eqref{eq:u} so we must verify \eqref{eq:8} on the parabolic boundary of the region $\{y\ge 2\ell_0, \tau_0\leq\tau'\le\tau\}$.  This boundary consists of a time-like edge $\{y= 2\ell_0, \tau_0\leq\tau'\leq \tau\}$ and a space-like initial edge $\{y\geq 2\ell_0, \tau'=\tau_0\}$.  Our assumption \eqref{eq:alpha definition}  implies that \eqref{eq:8} holds on the time-like edge so we only have to check \eqref{eq:8} at the initial edge.  The definition~\eqref{eq-initial} of $\ueo$ implies $\ueo(y, \tau_0)=\bu(y, \tau_0)$ for all $y\geq  2\ell_0$, so that \eqref{eq:8} with $\tau'=\tau_0$ follows from monotonicity of the peanut solution $\bu$.
\end{proof}

\begin{lemma}		\label{lem:w pointwise}
There is a constant $C>0$ depending only on  the dimension $n$ and  the parameters $(K_0, \delta)$ in the definition of the peanut, such that for all $y\in[2\ell_0, 2L(\tau)]$ we have
\begin{equation}		\label{eq:9}
e^\tau|w(y, \tau)|\leq C  \frac{y^4}{\ell_0^4}\sup_{\tau_0\leq \tau'\leq \tau} e^{\tau'} |w(2\ell_0, \tau')|.
\end{equation}
\end{lemma}
\begin{proof}
We begin by estimating $\alpha(\tau)$, which, by definition, is the smallest $\alpha\in(0, \ln 2)$ satisfying \eqref{eq:alpha definition}, or, equivalently,
\begin{equation}		\label{eq:12}
\forall \tau'\leq\tau:\quad
\bu(2\ell_0, \tau'-\alpha)-\bu(2\ell_0, \tau')\leq w(2\ell_0, \tau') \leq
\bu(2\ell_0, \tau'+\alpha)-\bu(2\ell_0, \tau').
\end{equation}
We can find upper and lower bounds for the differences on the left and right in this inequality by recalling that 
the peanut solution $\bu$ satisfies \eqref{eq:um2-derivs}, which implies that for some constant $C_0$ that
\[
\frac{1}{C_0}e^{-\tau}y^4 \leq \bu_\tau(y, \tau) \leq C_0 e^{-\tau}y^4
\]
holds for all $y\in[\ell_0, 2L(\tau)]$ and all $\tau\geq \tau_0$.  Hence,  for all $y\in[\ell_0, 2L(\tau)]$, all $\tau\geq \tau_0$, and any $\tilde\alpha$ with $|\tilde\alpha|\leq \ln 2$ we have
\begin{equation}		\label{eq:7}
\frac{1}{2C_0}e^{-\tau}y^4 \leq \bu_\tau(y, \tau+\tilde\alpha) \leq 2C_0 e^{-\tau}y^4.
\end{equation}
Since $\alpha\in (0,\ln 2)$ the Mean Value Theorem combined with the lower bound for $\bu_\tau$ from \eqref{eq:7} tells us
\begin{equation}		\label{eq:11}
        \begin{split}
                \bu(2\ell_0, \tau'+\alpha) - \bu(2\ell_0, \tau') \geq 
                \frac {1}{2C_0} e^{-\tau'}(2\ell_0)^4\alpha,
                \\
                \bu(2\ell_0, \tau'-\alpha) - \bu(2\ell_0, \tau')
                \leq -\frac{1}{2C_0}e^{-\tau'}(2\ell_0)^4\alpha
        \end{split}
\end{equation}
for all $\tau'\in[\tau_0, \tau]$.  Hence $\alpha = 2C_0(2\ell_0)^{-4}\sup_{\tau_0\leq \tau'\leq \tau} e^{\tau'}|w(2\ell_0, \tau')|$ satisfies \eqref{eq:12}.  This implies
\[
\alpha(\tau) \leq  2C_0(2\ell_0)^{-4}\sup_{\tau_0\leq \tau'\leq \tau} e^{\tau'}|w(2\ell_0, \tau')|.
\]
Next by \eqref{eqn-uuu} we see that we indeed have $2 C_0(2\ell_0)^{-4}\sup_{\tau_0\leq \tau'\leq \tau} e^{\tau'}|w(2\ell_0, \tau')| \leq \ln 2$, as claimed above. Note that \eqref{eqn-uuu} will be shown to hold in Lemma \ref{cor-prop-funnel}.
 
In the region $2\ell_0\leq y\leq 2L$ we apply the Mean Value Theorem to \eqref{eq:5} and conclude that for some $\alpha^*$ with $|\alpha^*|\leq \alpha(\tau)\leq \ln 2$
\[
|w(y, \tau)| \leq \alpha(\tau) \bu_\tau(y, \tau+\alpha^*)
\]
The upper bound for $\alpha(\tau)$ combined with the upper bound \eqref{eq:7} for $\bu_\tau$ then implies
\begin{align*}
|w(y, \tau)|
&\leq \left(\frac 18C_0\ell_0^{-4}\sup_{\tau_0\leq \tau'\leq \tau} e^{\tau'}|w(2\ell_0, \tau')| \right)\cdot \left(C_0 e^{-\tau-\alpha^*} y^4\right)\\
&\leq \frac 14C_0^2 e^{-\tau}y^4\ell_0^{-4}  \sup_{\tau_0\leq \tau'\leq \tau} e^{\tau'}|w(2\ell_0, \tau')|,
\end{align*}
which completes the proof of \eqref{eq:9}.
\end{proof}

 \subsection{Estimate in terms of the $\hv[\ell/2, \ell]$ norm}
 In this section we prove an $L^2$ version of the pointwise bound in Lemma~\ref{lem:w pointwise}.  For any real numbers $a<b$ the $\hv[a, b]$ norm of a function $h:\R\to\R$ is by definition given by
\begin{equation}		\label{dfn-hv}
\|h\|_{\hv[a,b]}^2 = \int_a^b \left\{ h'(y)^2 + h(y)^2\right\}e^{-\frac{y^2}{4}}dy.
\end{equation}
In the proof we use this elementary estimate:
\begin{lemma}		\label{lemma:w  pointwise from D norm}
{\upshape (Calculus inequality)}
For any $\tau\in[\tau_0, \tau_1]$
\begin{equation}
         w(2\ell_0,\tau)^2 \leq 5\ell_0^{-1}\, e^{\ell_0^2} \|w(\tau)\|_{\hv[\ell_0, 2\ell_0]}^2.
\end{equation}
\end{lemma}
\begin{proof}
We recall the bound 
\begin{equation}
                \label{eqn-good}
        \ell\, e^{-\ell^2/4} f(\ell)^2 %+ \frac 1{4} \int_{0}^\ell y^2 \, f(y)^2 \, e^{-\frac{y^2}{4}} \, dy         \leq \int_{0}^\ell  \bigl[4f'(y)^2 + f(y)^2 \bigr]\, e^{-\frac{y^2}{4}} \, dy,
\end{equation}
which holds for any $\ell>0$ and any function $f \in C^1([0,\ell])$ (e.g., see \cite{ADS1}).   Apply \eqref{eqn-good} with $\ell=2\ell_0$ to $f(y)= w(y, \tau') \, \zeta(y)$, where $\zeta\in C^1([0,\ell])$ is a cut off function with $|\zeta|\leq 1, |\zeta_y|\leq 1$ and
\[
\zeta(y) = 0 \text { for } y\in [0,\ell_0] \text{ and } \zeta(y) = 1 \text{ for } y\in [\tfrac 43\ell_0,2\ell_0].  
\]
The Lemma now follows from $w(2\ell_0,\tau')=f(2\ell_0)$ and $f_y^2 \leq 2 w_y^2 \, \zeta^2 + 2 w^2 \, \zeta_y^2\leq 2w^2+2w_y^2$.
\end{proof}

\begin{lemma}		\label{lem:w D bound}
There is a constant, depending on $\ell_0$, such that for all $\tau\geq \tau_0$ one has
\[
\|e^\tau w(\tau)\|_{\hilb[L, 2L]}^2
\leq Ce^{-\frac {1}{16} L^2}   \sup_{\tau_0\leq \tau'\leq \tau} e^{2\tau'}\|w(\tau')\|_{\hv[\ell_0, 2\ell_0]}^2.
\]
\end{lemma}
Here and in the following proof we abbreviate $L=L(\tau)=\rho e^{\tau/4}$.

\begin{proof}
For $y\in [L,  2L]$  Lemmas~\ref{lem:w pointwise} and~\ref{lemma:w  pointwise from D norm} imply
\begin{align*}
  e^{2\tau}w(y, \tau)^2& \leq C y^8\ell_0^{-8}\sup_{\tau'}e^{2\tau'}w(2\ell_0, \tau')^2 \\
&\leq C y^8\ell_0^{-8}\sup_{\tau'} e^{2\tau'} \ell_0^{-1}e^{\ell_0^2}\|w(\tau')\|_{\hv[\ell_0,2\ell_0]}^2\\
  &= C y^8\ell_0^{-9} e^{\ell_0^2} \sup_{\tau'} e^{2\tau'}\|w(\tau')\|_{\hv[\ell_0,2\ell_0]}^2
\end{align*}
where all the suprema are over $\tau'\in[\tau_0, \tau]$.  Integrate the inequality over $L\leq y\leq 2L$, using
\[
\int_{L}^{2L} y^8 e^{-\frac{y^2}{4}}dy
< e^{-\frac{1}{8}L^2} \int_0^\infty y^8 e^{-\frac{y^2}{8}}\; dy
= C e^{-\frac{1}{8}L^2}.
\]
We get
\begin{align*}
\|e^\tau w(\tau)\|_{\hilb[L, 2L]}^2
&=  \int_L^{2L} e^{2 \tau} |w(y,\tau)|^2 e^{-\frac{y^2}{4}} dy\\
&\leq C  e^{-\frac 18 L^2+ \ell_0^2} 
\ell_0^{-9}\sup_{\tau_0\leq \tau'\leq \tau} \,  e^{2 \tau'}  \|w(\tau')\|_{\hv[\ell_0, 2\ell_0]}^2.
\end{align*}
Since $L=L(\tau) = \rho e^{\tau/4}\geq \rho e^{\tau_0/4} \geq 4\ell_0$ we have $\frac{1}{16}L^2 \geq \ell_0^2$ and  we conclude
\begin{equation}		\label{eqn-boundw12} 
\|e^\tau w(\tau)\|_{\hilb[L, 2L]}^2
 \leq C  \, e^{-\frac{1}{16}L^2}   \ell_0^{-9}\sup_{\tau_0\leq \tau'\leq \tau} e^{2\tau'}\|w(\tau')\|_{\hv[\ell_0,2\ell_0]}
\end{equation}
for a fixed $C$.  Since $\ell_0>1$ always, we may discard the factor $\ell_0^{-9}$, which then leads to the estimate in the Lemma.
\end{proof}

\subsection{Estimate in terms of the $\hilb$ norm}
To replace the $\hv$ norm in \eqref{eqn-boundw12} by an $\hilb$ norm, we use the following energy estimate.
\begin{lemma}		\label{lem-energy}
For any $\tau \geq \tau_0$, we have
\begin{equation}
 \label{eqn-333}
\sup_{\tau' \in [\tau_0,\tau]}  \|e^{\tau'}  w(\tau')  \|_{\hv[\ell_0,2\ell_0]}^2 \leq 
C\int_{\tau_0}^\tau e^{2\tau'} \| w(\tau') \|_{\hilb[0, 4\ell_0]} ^2 d\tau' .
\end{equation}
The constant $C$ in this lemma may depend  on $\ell_0$.
\end{lemma}

\begin{proof}We differentiate equation \eqref{eqn-w2} with respect to $y$ and use the fact that
\[
|c_2| + |c_1| + |c_0|+ |c_{1y} | + |c_{0y}| \leq o(1)
\]
in the considered region (which follows by the fact that $u, \bu$ satisfy the bounds \eqref{eqn-u-der-bounds}).  Standard parabolic energy estimates then imply
\[
\sup_{\tau'}  \|e^{\tau'} w(\tau')  \|_{\hv[\ell_0,2\ell_0]}^2 \leq 
\| e^{\tau_0} w(\tau_0) \|_{\hv[\ell_0,2\ell_0]}^2 + C\int_{\tau_0}^\tau e^{2\tau'} \| w(\tau') \|_{\hilb[0, 4\ell_0]} ^2 d\tau'.
\]
In \eqref{eq-initial} we  have chosen the initial condition of the perturbation so that $w(y, \tau_0)=0$ for $y\geq \ell_0$ and therefore the first term in this estimate vanishes.
\end{proof}

\smallskip

\begin{proof}[Proof of Proposition ~\ref{prop-inout}] 
The proof directly follows by combining Lemma \ref{lem:w D bound}  and Lemma \ref{lem-energy}. 
\end{proof}

%%%%%%%%%%%%%%%%%%%%%%
%%%%%%%%%%%%%%%%%%%%%

\section{The growth of  the unstable mode of $w$  inside the funnel}
\label{sec-growth}

Recall that $\ueo(y,\tau)$ denotes the solution of \eqref{eq:u} with initial data $\ueo(y,\tau_0)$ defined by \eqref{eq-initial} and $\bu(y,\tau)$ denotes the peanut solution.  For now we drop the index and call $\ueo$ simply by $u$. The function $w:=u-\bu$ satisfies
\begin{equation}		\label{eqn-wtau0}
w(\cdot, \tau_0) =  \epsilon\,\eta_0(y/\ell_0) \, \big(\Omega_0 + \Omega_2 \hm_2(y)\big),
\end{equation}
where $\eta_0(y/\ell_0)$ is the cutoff function supported on a set $|y| \le 2\ell_0$ as defined in \eqref{eq-initial}.  Recall that $W(y,\tau) = \eta(y, \tau)w(y, \tau)$ was defined in \eqref{eqn-defnW}, where the cutoff function $\eta(y,\tau) = \eta_0(y/L(\tau))$, with $L(\tau) = \rho e^{\tau/4}$, satisfies
\begin{equation}		\label{eqn-eta}
|\eta_y| < c_0 e^{-{\tau/4}}, \qquad  |\eta_{yy}| < c_0 e^{-{\tau/2}}, \qquad |\eta_{\tau}| \le c_0 e^{-{\tau/4}},
\end{equation}
for a uniform constant $c_0  > 0$, and where all these derivatives are supported in $L(\tau) \le |y| \le  2L(\tau)$.

Since we can take $\tau_0$ large enough so that $e^{\tau_0/4} \geq 2\ell_0 $ (that is $\eta(y, \tau_0) \equiv 1$ on the support of the initial perturbation $w(\cdot, \tau_0)$) we have
\begin{equation}		\label{eqn-Wtau0}
W(\cdot, \tau_0) = \epsilon\, \eta_0 \, \big(\Omega_0 + \Omega_{2}\hm_{2}\big) .  
\end{equation}
The function  $W(y,\tau)$ satisfies \eqref{eqn-W}, where the error term  $g:=  \eta \, \cE(w)  + \cE(w,\eta)$  is given by \eqref{eqn-E1} and \eqref{eqn-E2}.  
\sk

Recall the definitions of $W^u(y,\tau)$ and $W^s(y,\tau)$, the projections of $W(y,\tau)$ onto subspaces $\hilb^u := \langle H_0, H_{2} \rangle$ and $\hilb^s = \langle H_{2i}, i= 2, \cdots \rangle$, respectively, as in \eqref{eq-stable-unstable}.  We can express $W = W^u +W^s$.

\subsubsection*{Notation: the dependence of constants $c_0, C_0$ and $c, C$ on $\ell_0$}
% \label{notation}

Throughout this section we will denote by $c_0, C_0$ positive universal constants depending only on dimension,
and $C, c$ positive constants that  depend on dimension and they may also depend on $\ell_0$.
\bigskip

Recall the Definition \ref{def-funnel} of the funnel $\mc F_{\tau}$.  Our  {\em goal in this section}  is to show that  the solution 
$u(y,\tau)$    hits the boundary of the funnel at some time $\tau_1$ that we call the exit time.  More precisely, we prove the following Proposition.

\begin{proposition}		\label{lemma-big-unstable} 
Let $M_1 >0$ be an arbitrary fixed constant.  There exists $\tau_0 \gg 1$ so that the following holds.  For every $\epsilon$ sufficiently small, where $\epsilon> 0$, and every $\bom\in \mathbb{S}^1$, there exists the first time $\tau_1 = \tau_1(\epsilon, \bom)$ so that
\begin{equation}		\label{eq-big-unstable} 
\|W^u(\cdot, \tau_1)\|  = M_1 \, e^{-\tau_1}
\end{equation}
Note that here we write shortly $W^u$ for $W_{\epsilon,\bom}^u$.  
\end{proposition}

\medskip

In order to show  Proposition \ref{lemma-big-unstable} we first need to show a series of other related results.  We first make the \emph{a priori} assumption that  $w(y,\tau)$ and its derivatives satisfy the $L^\infty$-bounds 
\begin{equation}		\label{eq-apriori-Linfty}
|w(y,\tau)| + |w_y(y,\tau)| + |w_{yy}(y,\tau)|  + |w_{yyy}(y,\tau)| \leq \Lambda   \, e^{-\tau} (1+|y|^4)
\end{equation}
for all $ y \in  [-2L(\tau) , 2L(\tau) ]$, $\tau \in [\tau_0, \tau_1]$, as long as $\|W^u(\cdot,\tau)\| \le M_1\, e^{-\tau}$,
for all $\tau \in [\tau_0,\tau_1]$.  Here $\Lambda$ is an auxiliary  constant.  We address this a'priori assumption on $L^{\infty}$-bounds in section \ref{sec-Linfty-w}.  Note that by our choice of initial data  this condition holds at $\tau_0$ for $|y| \le 2L(\tau_0)$, and by short time 
regularity it also holds on $\tau \in [\tau_0,\tau_0+\theta_0]$, $|y| \le2 L(\tau)$,   for some $\theta_0 >0$ small, when we replace constant $\Lambda$ by a slightly larger constant.  Note also that due to the peanut asymptotics \eqref{eq:um2-derivs},  the bound  \eqref{eq-apriori-Linfty} also implies:
\begin{equation}		\label{eq-v-small}
|u(y,\tau) - \sqrt{2(n-1)}| + |u_y(y,\tau)| + |u_{yy}(y,\tau)| + |u_{yyy}(y,\tau)| \leq \Lambda\,(1 + |y|^4)\, e^{-\tau}
\end{equation}
for $|y| \le 2\,L(\tau)$, and $\tau\in [\tau_0,\tau_1]$.  Here the constant $\Lambda$ can be chosen the same as in \eqref{eq-apriori-Linfty}.

Recall the definitions of the Hilbert spaces $\hilb$ and $\hv$ and their norms, given in subsection \ref{ss-HHH}.  Denote by $\hv^*$ the dual of $\hv$. Since we have a dense inclusion $\hv \subset \hilb$, we also get a dense inclusion $\hilb \subset \hv^*$ where every $f \in \hilb$ is interpreted as a functional on $\hv$ via
\[
g\in \hv \to \langle f,g\rangle.
\]
Because of this we will also denote the duality between $\hv$ and $\hv^*$ by
\[
(f,g) \in \hv\times\hv^* \to \langle f, g\rangle.
\]
Since $\hilb \subset \hv^*$, for every $f \in \hilb$ we define the dual norm as usual by
\[
\|f\|_{\hv^*} := \sup\{ \langle f,g\rangle\,\,\, : \,\,\, \|g\|_\hv \le 1\}.
\]
\smallskip 
The next Lemma tells how we estimate error term $g:=  \eta \, \cE(w)  + \cE(w,\eta)$ in \eqref{eqn-W}  with respect to appropriately chosen norms. 

\begin{lemma}\label{lemma-error-small10}
Assume that $u \in \cF_{\epsilon, \tau_1}$ and $w$ satisfies \eqref{eq-apriori-Linfty}, which implies that $u$ satisfies \eqref{eq-v-small}.  Let $g:=\cE(w) + \cE(w,\eta)$ be  the error term  given by  equation \eqref{eqn-W}.  For $\tau_0 \gg 1$ chosen large as above (independent of the considered  initial data)  we have that for any $\delta > 0$ there exists a $\rho$ sufficiently small so that 
\begin{equation}		\label{eq-error-g}
\|g\|_{\hv ^*}^2 \le \delta\, \|W\|_\hv^2 + \delta \int_{L(\tau)}^{2L(\tau)} w^2 \, d\mu,
\end{equation}
for all $\tau\in [\tau_0,\tau_1]$, where  $d\mu = e^{-y^2/4} dy$, and  $L(\tau)= \rho\, e^{\tau/4}$. 
\end{lemma}

\begin{proof}
Having \eqref{eq-v-small}, the proof of Lemma is analogous to the proofs of Lemma 6.8 and Lemma 6.9 in \cite{ADS}.  Here we use the estimate \eqref{eq-v-small}, and the fact that we cut very far away, of the order of $L(\tau) = \rho e^{{\tau/4}}$, so that the errors coming from cut off functions are exponentially small.  To be more precise, by \eqref{eq-v-small}, \eqref{eqn-E1} and \eqref{eqn-E2}, keeping in mind that time and space derivatives of $\eta(y,\tau)$ are supported on $L(\tau) \le |y| \le 2L(\tau)$, and that $\rho > 0$ is sufficiently small we get \eqref{eq-error-g}.
\end{proof}

\sk
\sk  

\subsection{Estimating the quotient $\|W^s(\cdot, \tau)\| /\|W^u(\cdot, \tau)\| $}  In the following Lemma we show that as long as our solution stays inside the funnel, the $L^2$ norm of unstable projection $\|W^u(\cdot,\tau)\|$  dominates  the $L^2$ norm of stable projection, $\|W^s(\cdot,\tau)\|$.  Recall that $ L = L(\tau)= \rho\, e^{\tau/4}$ and that the support of the initial perturbation $w(y, \tau_0) = \eta_0(y/\ell_0)\bigl(\Omega_0 \hm_0(y) + \Omega_2 \hm_2(y)\bigr)$ is contained in the interval $(-2\ell_0, 2\ell_0)$ (See \eqref{eqn-wtau0}).

\begin{lemma}\label{lemma-quotient1}\label{cor-quotient}
There exists a large constant $\ell_0$, and small $\rho = \rho(\ell_0)$ so that if $u \in \cF_{\tau_1}$ satisfies \eqref{eqn-uuu} and $w$ satisfies \eqref{eq-apriori-Linfty}, then
\[
\|W^s(\cdot, \tau)\|   <   e^{-\ell_0^2/8}\,  \|W^u(\cdot, \tau)\| 
\]
for all $\tau\in [\tau_0,\tau_1]$. 
\end{lemma}

\begin{proof}
Throughout the proof we assume that $u\in \mc F_{\tau_1}$.  The proof of the Lemma will immediately follow from the two steps below.  
\begin{step}		\label{step-one-side}
As long as $\|W^s(\cdot, \tau)\|  < e^{-\ell_0^2/8} \|W^u(\cdot, \tau)\| $, we have
\begin{equation}		\label{eq-cond}
\int_{\tau_0}^{\tau} \|W^u(\cdot, s)\| ^2  e^{2 s} \, ds  < 2 e^{2\tau}\|W^u(\cdot, \tau)\|^2 .
\end{equation}
This inequality holds at $\tau=\tau_0$.  By continuity \eqref{eq-cond} also holds for $\tau>\tau_0$ sufficiently close to $\tau_0$.  To show that \eqref{eq-cond} holds for all $\tau < \tau_1$, we argue by contradiction and assume that there is a first time $\bar \tau_1 \in (\tau_0, \tau_1)$ at which \eqref{eq-cond} fails while $\|W^s(\cdot, \tau)\| \le e^{-\ell_0^2/8}\|W^u(\cdot, \tau)\| $ holds for $\tau\in [\tau_0,{\bar \tau_1}]$.  Thus, for all $\tau\in [\tau_0,\bar \tau_1)$ the inequality~\eqref{eq-cond} holds, while at $\tau=\bar\tau_1$ we have equality in \eqref{eq-cond}.

Since $\|W^u(\cdot, s)\|>0$ for all $s\in[\tau_0, \bar\tau_1)$ that are sufficiently close to $\tau_0$, the integral on the left in \eqref{eq-cond} is always positive, and therefore
\begin{equation}		\label{eq:Step 1 Wu not zero}
\|W^u(\cdot, s)\|>0\text{ for all }s\in[\tau_0, \bar\tau_1).
\end{equation}

From the equation $\frac{\partial W}{\partial\tau} = \mc{L} W + g$ for $W$, using that the unstable eigenspace is finite dimensional, similarly as in section 6 in \cite{ADS}, using \eqref{eq-apriori-Linfty}, after integration by parts and using the Cauchy Schwartz inequality we get 
\begin{equation}		\label{eq-unstable-W1}
\frac{d}{ds}\|W^u\|^2 \ge -\delta\, \|W^u\|^2 -\delta \int_L^{2L} w^2\, d\mu \ge -c_0 \, \delta \, \|W^u\|^2 -\delta \int_L^{2L} w^2\, d\mu,
\end{equation}
%where we have used that $\|W^u\| \le \|W^u\|_\hv \le c_0 \|W^u\| $, which holds because the unstable eigenspace is finite dimensional.  
Here $\delta>0$ is a constant that can be made as small as we want by taking $\rho$ small.  We can rewrite the previous inequality as
\[
\frac{d}{ds} \Big(e^{2s} \|W^u\|^2 \Big) \ge (2 - c_0\delta) \,   e^{2s}\|W^u\|^2 - \delta \, e^{2s} \|w\|_{\hilb[L, 2L]}^2 .  
\]
Since we assume \eqref{eqn-uuu} holds, Proposition \ref{prop-inout} tells us that
\[
\frac{d}{ds} \Big(e^{2s} \|W^u\|^2  \Big)\geq(2-c_0\delta) \, \|W^u\|^2  e^{2s} - \delta \int_{\tau_0}^s e^{2s'} \|W\|^2  ds'.
\]
By $\|W^s\|^2 \le \delta \|W^u\| ^2$, and \eqref{eq-cond}, for all $\tau\in [\tau_0,{\bar \tau_1})$ we  have
\[
\frac{d}{ds} \Bigl(e^{2s} \|W^u\|^2  \Bigr)
\geq (2-c_0\delta-\delta) e^{2s} \|W^u\|^2  
\geq \frac{3}{2} e^{2s} \|W^u\|^2  ,
\]
provided we make sure $\delta$ is sufficiently small.
For any $s\in[\tau_0, \bar\tau_1)$ we integrate this differential inequality over  $[s,\bar \tau_1]$ to get
\[
e^{2s} \|W^u(\cdot,s)\|^2  \le e^{2{\bar \tau_1}}\|W^u(\cdot, {\bar \tau_1})\|^2  e^{-3/2\,({\bar \tau_1} - s)},
\]
for all $s\in [\tau_0,{\bar \tau_1}]$.  Integrate this in $s$, from $\tau_0$ to ${\bar \tau_1}$ to get
\begin{equation}		\label{eq:13}
\int_{\tau_0}^{\bar \tau_1} e^{2s}\|W^u(\cdot, s)\|^2 ds
\leq e^{2\bar\tau_1}\|W^u(\cdot, \bar\tau_1)\|^2 \int_{\tau_0}^{\bar\tau_1}e^{-\frac 32 (\bar\tau_1-s)}ds
\leq \frac 23 e^{2\bar\tau_1}\|W^u(\cdot, \bar\tau_1)\|^2.
\end{equation}
On the other hand we began the proof of Step 1 with the  assumption that
\begin{equation}		\label{eq:14}
\int_{\tau_0}^{\bar\tau_1} \|W^u(\cdot, s)\| ^2  e^{2 s} \, ds  = 2 e^{2\bar\tau_1}\|W^u(\cdot, \bar\tau_1)\|^2 .
\end{equation}
We have shown in~\eqref{eq:Step 1 Wu not zero} that $\|W^u(\cdot, \bar\tau_1)\|\neq 0$, so \eqref{eq:13} and \eqref{eq:14} cannot both be true, which completes the proof of Step~\ref{step-one-side}.
\end{step}

\begin{step}		\label{step-two-side}
As long as \eqref{eq-cond} holds we have $\|W^s(\cdot, \tau)\|^2  < e^{-3\ell_0^2/16} \|W^u(\cdot, \tau)\|^2 $.  

To show this, we begin by observing that similarly to the proof of Lemma 6.6 in \cite{ADS}, we have
\[
\frac{d}{d\tau}\|W^s\|^2  \le - 2 c  \, \|W^s\|^2_{ \hv} + \|g\|_{\hv^*} \|W^s\|_\hv
\]
for some $c >0$. 

Using the above inequality,  Lemma \ref{lemma-error-small10}, and Cauchy-Schwarz inequality,  we get
\[
\frac{d}{d\tau}\|W^s\|^2  \le  -\frac 53  \, c \, \|W^s\|^2_\hv + \delta \|W\|^2_\hv + \delta\int_L^{2L}w^2 \, d\mu.
\]
Using $\|W\|_\hv^2 = \|W^s\|_\hv^2 + \|W^u\|_\hv^2$,  the fact that $\|W^u\|_\hv\leq C \|W^u\|$ (because $\hilb^u$ is finite dimensional), and assuming that $\delta<\frac 16$, we get
\begin{equation}		\label{eq-stable-W1}
\frac{d}{d\tau} \|W^s\|^2  \le -\frac32 \, c \, \|W^s\|^2_\hv + \delta \|W^u\|^2  + \delta \int_L^{2L} w^2 \, d\mu,
\end{equation}
Since we are assuming \eqref{eqn-uuu}, by \eqref{eq-stable-W1}, \eqref{eq-unstable-W1} and Proposition \ref{prop-inout},  we have
\[
\frac{d}{d\tau} \|W^s\|^2  \le -\frac{3}{2} \, c \, \|W^s\|^2  + \delta\| W^u\|^2  + \delta\, \int_{\tau_0}^{\tau} e^{-2(\tau-s)}\, \|W\|^2 \, ds
\]
and 
\[
\frac{d}{d\tau} \|W^u\|^2  \ge -\delta \|W^u\|^2  - \delta\, \int_{\tau_0}^{\tau} e^{-2(\tau-s)}\, \|W\|^2 \, ds.  
\]
The above hold for  a tiny constant $\delta$. 

Assume $\bar \tau_1 < \tau_1$ is the first time so that
\begin{itemize}
\item  \eqref{eq-cond} holds for all $\tau\in [\tau_0, \bar \tau_1]$, 
\item $\|W^s(\cdot, \tau)\|  <  e^{-3\ell_0^2/16} \|W^u(\cdot, \tau)\| $ holds for all $\tau\in [\tau_0, \bar \tau_1)$, 
\item $\|W^s(\bar \tau_1)\| = e^{-3\ell_0^2/16} \|W^u(\bar \tau_1)\| $
\end{itemize}
Using the same notation $\delta$ for a tiny constant that can vary from line to line, the above estimates yield, for all $\tau\in [\tau_0, \bar \tau_1]$, the bounds
\begin{align}
  \frac{d}{d\tau} \|W^s\|^2 & \le - c \,  \|W^s\|^2  + \delta \|W^u\|^2 ,    \\
  \frac{d}{d\tau} \|W^u\|^2 & \ge - 2\delta\, \|W^u\|^2. \label{eq-unstable-useful}
\end{align}
Set $Q(\tau) := \frac{\| W^s (\cdot, \tau)  \|^2}{\| W^u (\cdot, \tau)  \|^2}.$ The last two inequalities yield
\[
\frac{d}{d\tau} Q \le - (c - 2 \delta ) \, Q  + \delta = -c_0 Q+\delta,
\]
where $c_0:=c-2\delta $.  Assume $\delta>0$ is so small that $c_0>0$, and use variation of constants to integrate the  inequality from $\tau_0$ to any $\tau\in (\tau_0,\bar \tau_1]$.  This leads to
\[
Q(\tau) \leq e^{-c_0(\tau-\tau_0)}Q(\tau_0) + \frac{\delta}{c_0}\big(1-e^{-c_0(\tau-\tau_0)})
<Q(\tau_0) + \frac{\delta}{c_0}.
\]
By Lemma \ref{lemma-lo} we have $Q(\tau_0) \leq e^{- 5\ell^2_0/16}$.  Recalling that $\delta$ can be made as small as we want by taking $\rho$ small,  we conclude we can choose $\rho = \rho(\ell_0)$ so small that
\[
  Q(\tau) \le   e^{-5\ell_0^2/16} + \frac{\delta}{c_0} < e^{-3\ell_0^2/16} 
\]
for all $\tau\in [\tau_0, \bar \tau_1]$.  This concludes the proof of the statement in Step \ref{step-two-side}.
\end{step}

To conclude the proof of Lemma \ref{lemma-quotient1}, we claim that Step \ref{step-one-side} and Step \ref{step-two-side}, together with the assumptions in the Lemma, imply that $\|W^s(\cdot, \tau)\| \le e^{-\ell_0^2/8}\, \|W^u(\cdot, \tau)\| $ and \eqref{eq-cond} hold for all $\tau\in [\tau_0,\bar{\tau_1}]$.  To justify this we argue as follows.  At time $\tau_0$ by Lemma \ref{lemma-lo} we have $\|W^s(\cdot,\tau_0)\| \le e^{-5\ell_0^2/32}\, \|W^u(\cdot,\tau_0)\| $.  Let $\tilde{\tau}_1 \in (\tau_0,\bar{\tau_1}]$ be the maximal time so that $\|W^s(\cdot,\tau)\| \le e^{-\ell_0^2/8}\|W^u(\cdot,\tau)\| $ for all $\tau\in[\tau_0, \tilde\tau_1]$.  Assume $\tilde{\tau}_1 < \bar{\tau}_1$, since otherwise there is nothing to prove.  By Step 1 we have~\eqref{eq-cond} for all $\tau\in [\tau_0,\tilde{\tau}_1)$.  We can now apply Step 2 to conclude that $\|W^s(\cdot,\tau)\| < \, e^{-3\ell_0^2/16}\,\|W^u(\cdot,\tau)\| \ll e^{-\ell_0^2/8} \|W^u(\cdot,\tau)\| $, for all $\tau\in [\tau_0,\tilde{\tau}_1)$, hence contradicting the maximality of $\tilde{\tau}_1$.  This implies $\tilde{\tau}_1 = \bar{\tau}_1$ as claimed.  The proof of Lemma \ref{lemma-quotient1} is now complete.
\end{proof}

\begin{remark}		\label{rem-gtau}
Lemma \ref{lemma-quotient1} implies that if $u \in \cF_{\tau_1}$ satisfies \eqref{eqn-uuu}, $w$ satisfies \eqref{eq-apriori-Linfty}, and if $\|W^u(\cdot, s)\| > 0$ for all $\tau\in [\tau_0,\bar{\tau}_1]$, then $\|W^s(\cdot, \tau)\| < e^{-\ell_0^2/8}\, \|W^u(\cdot, \tau)\| $, for all $\tau\in[\tau_0,\bar{\tau}_1]$, where $\bar{\tau}_1 \le \tau_1$.  Step~1 now implies that~\eqref{eq-cond} holds for all $\tau\in [\tau_0,\bar{\tau}_1]$ as well. 
\end{remark}

\subsection{$L^\infty$ estimates}		\label{sec-Linfty-w}

Lemma~\ref{lemma-quotient1} plays a crucial role in the proof of Proposition \ref{lemma-big-unstable} and consequent results.  In the proof of Lemma \ref{lemma-quotient1} we made two \emph{a priori} assumptions:
\begin{enumerate}[\upshape(i)]
\item the bound \eqref{eqn-uuu} that was used in the proof of Proposition \ref{prop-inout} which in turn was used in the proof of Lemma \ref{lemma-quotient1}, and
\item the \emph{a priori} $L^{\infty}$ estimate \eqref{eq-apriori-Linfty} that was used directly in Lemma \ref{lemma-quotient1}.
\end{enumerate}
Assuming \eqref{eqn-uuu} and \eqref{eq-apriori-Linfty}, we will now improve these two bounds in such a way that at the end we will get that the conclusion of Lemma \ref{cor-quotient} holds independently of these \emph{a priori} assumptions.

We start with the following  consequence of  Lemma~\ref{lemma-quotient1}.

\begin{lemma}\label{cor-expansion} 
Assume that $u \in \mc F_{\tau_1}$ and that $ \|W^s (\cdot, \tau)\| \le e^{-\ell_0^2/16}\, \|W^u(\cdot, \tau)\| $ holds for all $\tau\in[\tau_0,\tau_1]$.  Let $\ell$ be any large constant satisfying  $ 1 \ll \ell \leq {\ell_0}/{10} \ll e^{\tau_0/4}$.  For all $\tau \in [\tau_0, \tau_1]$ with $\tau_0$ sufficiently large, we have
\begin{equation}		\label{eqn-wexpansion} 
\big \| w(\cdot, \tau) -  W^u(\cdot,\tau) \big \|_{C^2[0,\ell]} \le  e^{-\ell_0^2/30}  \, e^{- \tau}.
\end{equation}
\end{lemma}

\begin{proof} 
We write $W^u(\cdot, \tau)\in\hilb^u$ as a linear combination of $\hm_0, \hm_2$:
\[
W^u(y, \tau) = e^{-\tau}d_0(\tau)\hm_0(y) + e^{-\tau}d_2(\tau)\hm_2(y).
\]
Since $u \in \mc F_{\tau_1}$ we have $\|W^u(\cdot, \tau) \|  \leq M_1 \, e^{-\tau}$ and $ |d_0(\tau)|+ |d_{2}(\tau)| \leq C(M_1)$ for all $\tau \in [\tau_0, \tau_1]$.

By Lemma \ref{cor-quotient} and the definition \ref{def-funnel} of the funnel, we have 
\[
\|W(\cdot,\tau)  -  W^u(\cdot, \tau)   \|_\hilb   \leq    M_1 e^{-\ell_0^2/16} \, e^{-\tau}
\]
for all $\tau \in [\tau_0, \tau_1]$.   Since $w(y,\tau)=W(y,\tau)$ for $|y| \leq L(\tau)$ (from the definition of $W$) and $4\ell < \ell_0 < L(\tau)$, we have $W=w$ on $[0, 4\ell]$, so 
that the previous estimate implies  
\begin{equation}		\label{eqn-w234} 
\|   w(\cdot,\tau)  -  W^u(\cdot, \tau) \|_{L^2([0, 4\ell])}    \leq   M_1 \,  e^{2\ell^2} e^{-\ell_0^2/16}  e^{-\tau}
\end{equation}
for all $\tau \in [\tau_0, \tau_1]$, where the exponential $e^{2\ell^2}$ on the right hand side comes from converting the weighted $\| \cdot\| $ norm to standard $L^2$ norm.\footnote{Namely,  for any $f\in L^2([0, 4\ell])$
\[
  \|f\|_{L^2([0, 4\ell])}^2 = \int_0^{4\ell} f(y)^2 \leq e^{\frac{(4\ell)^2}{4}}\int_0^{4\ell} f(y)^2 e^{-\frac{y^2}{4}}
  =e^{4\ell^2}\int_0^{4\ell} f(y)^2 e^{-\frac{y^2}{4}} = e^{4\ell^2} \|f\|_{\hilb[0,4\ell]}^2
\]
}

We next apply standard interior $L^\infty$ estimates on $f(y,\tau):= w(y,\tau) - W^u(y, \tau)$ to derive a bound on $\| f(\cdot, \tau) \|_{C^0([0, 2\ell])}$, $\tau \in [\tau_0, \tau_1]$.
To this end, let us recall that $W$ satisfies equation $W_\tau = \cL W + g$, where $g:= \eta\cE(w) + \cE(w, \eta)$ and $\cE(w)$, $\cE(w, \eta)$ are defined in \eqref{eqn-E1}, \eqref{eqn-E2} respectively.  It then follows that $W^u$ satisfies equation $(W^u)_\tau = \cL W^u + g^u$, where $g^u$ denotes the projection of $g$ onto the unstable subspace $\hilb^u := \langle H_0, H_{2} \rangle$ of $\hilb$.  Therefore, by the above equation for $W^u$ and \eqref{eqn-w}, $f := w-W^u$ satisfies
\begin{equation}		\label{eqn-Wu100}
f_\tau = \cL f  + \tilde g, \qquad \tilde g = \cE(w)  - g^u
\end{equation}
on $y \in [0, 4\ell]$, $\tau \in [\tau_0, \tau_1]$. 

From now on, one can argue similarly to  the proof of Lemma 4.2 in \cite{AV}. First,  one applies standard  interior $L^\infty$ estimates for the equation \eqref{eqn-Wu100}  for $\tau \in [\tau_0+1, \tau_1]$ (if $\tau_1 < \tau_0+1$
we ignore this case) to obtain the bound  
\begin{equation}		\label{eqn-f234}
\| f(\cdot, \tau) \|_{C^0([0, 2\ell])}   \leq   \sup_{\tau' \in [\tau-1, \tau]}  C(\ell)  \, \big ( \| f(\cdot, \tau') \|_{L^2([0, 4\ell])}  + \| \tilde g(\cdot, \tau') \|_{L^2([0, 4\ell])}  \big ). 
\end{equation}
all $\tau \in [\tau_0 +1, \tau_1]$, where here and below $C(\ell)$ denote constants that may vary from line to line, but they are at most polynomial in $\ell$.  

We next claim that  
\begin{equation}		\label{eqn-tg100}
\| \tilde g(\cdot, \tau) \|_{L^2 [0, 4\ell]} \leq C(\ell) \, \Lambda^2 \, e^{-2\tau}
\end{equation}
To prove this, we first write 
\begin{align*}
	\tilde g 
	&= \cE(w) - g^u \\
	&=  c_2(y, \tau)w_{yy} + c_1(y, \tau)w_y + c_0(y, \tau)w 
	-\langle g, \hm_0\rangle \frac{\hm_0}{\|\hm_0\|^2} 
	-\langle g, \hm_2 \rangle\frac{\hm_2}{\|\hm_2\|^2}.
\end{align*}
Therefore
\[
\|\tilde g\|_{L^2([0, 4\ell])} \lesssim \|c_2 w_{yy} + c_1 w_y + c_0 w\|_{L^2([0, 4\ell])} +  |\langle g, \hm_0\rangle|+  |\langle g, \hm_2\rangle|.
\]
The coefficients $c_j$ are defined in \eqref{eqn-E1}, and thus we can estimate them by
\begin{align*}
	|c_2(y, \tau)|&\leq u_y^2 \lesssim \bu_y^2 + w_y^2 
	\lesssim (\delta^2+\Lambda^2) (1+y^4)^2e^{-2\tau}
	\lesssim \Lambda^2 (1+y^4) e^{-2\tau}
	\\
	|c_1(y, \tau)| &\lesssim u_y^2 + \bu_y^2 + \bu_{yy}^2  \lesssim \Lambda^2 (1+y^4)^2 e^{-2\tau} 
	\\
	|c_0(y, \tau)|&\lesssim \Big|u-\sqrt{2(n-1)}\Big| + \Big|\bu-\sqrt{2(n-1)}\Big| 
	\leq  |w| + \Big|\bu-\sqrt{2(n-1)}\Big| \\
	&\lesssim \Lambda (1+y^4) e^{-\tau}.
\end{align*}
Hence 
\begin{equation}		\label{eq:cEw-pointwise}
	|\cE(w)| \lesssim \Lambda^2 (1+y^4)^2 e^{-2\tau} \text{ for }
|y|\leq 2L(\tau), 
\end{equation}
and thus $\|\cE(w)\|_{L^2([0,4\ell])} \lesssim \Lambda^2 e^{-2\tau}$.

To estimate $\tilde g = \cE(w) + g^u$ we still need to estimate $\langle g, \hm_j\rangle$ for $j=0, 2$.
Since
\begin{align*}
  g &= \eta \cE(w) + \cE(\eta, w) \\
    &= \eta \cE(w) + \bigl(\eta_\tau - \eta_{yy} - \tfrac y2 \eta_y\bigr)w - 2\eta_y w_y
\end{align*}
we have
\[
\|g\|_{\hilb}\leq \|\eta \cE(w)\|_{\hilb} + \|\cE(\eta, w)\|_{\hilb}.
\]
The first term is bounded by the pointwise estimate \eqref{eq:cEw-pointwise} which implies 
\[
\|\eta\cE(w)\|_{\hilb} \lesssim \Lambda^2 e^{-2\tau}.
\]
The term $\cE(\eta, w) =  \bigl(\eta_\tau - \eta_{yy} - \tfrac y2 \eta_y\bigr)w - 2\eta_y w_y$ is supported in the interval $L(\tau)\leq y\leq 2L(\tau)$, while the derivatives of the cut-off functions are bounded by
\[
 |\eta_\tau| + |\eta_{yy}| + |\tfrac y2 \eta_y| + |\eta_y| \leq C,
\]
for some constant $C$ (the largest term is $y\eta_y$).  It follows that
\[
|\cE(\eta, w)|\lesssim \Lambda (1+y^4) e^{-\tau} \chi_{[L(\tau), 2L(\tau)]}(y) 
\]
which implies
\[
\|\cE(\eta,w)\|_\hilb \lesssim \sqrt{\int_L^{2L} y^8 e^{-\frac{y^2}{4}} dy }
\lesssim \Lambda L^{\frac 72} e^{-\frac{L^2}{8}}.
\]
Hence we get
\[
\|g\|_\hilb  \lesssim \Lambda^2 e^{-2\tau} + \Lambda L(\tau)^{\frac 72} e^{-\frac{L(\tau)^2}{8}}.
\]
Since $L(\tau)=\rho e^{\tau/4}$ we can choose $\tau_0$ large enough, depending on $\Lambda$, to ensure that 
\[
L(\tau)^{\frac{7}{2}} e^{-\frac{L(\tau)^2}{8}} \leq \Lambda^2 e^{-2\tau}
\]
for all $\tau\geq \tau_0$.   It then follows that
\[
|\langle g, \hm_j \rangle| \lesssim\|g\|_\hilb \lesssim \Lambda^2 e^{-2\tau}
\]
which implies $\|\tilde g\|_{L^2([0, 4\ell])} \lesssim \Lambda^2 e^{-2\tau}$, that is \eqref{eqn-tg100}  holds. 

Inserting the bounds \eqref{eqn-w234} and $\| \tilde g(\cdot, \tau) \|_{L^2 [0, 4\ell]} \leq C(\ell) \, \Lambda^2 \, e^{-2\tau}$  in the estimate \eqref{eqn-f234}, we get 
\begin{equation}		
 \| f(\cdot, \tau) \|_{C^0([0, 2\ell])}
 \leq C(\ell) \, \bigl( M_1 \,  e^{2\ell^2} e^{-\ell_0^2/16}  e^{-\tau} +\Lambda^2 \, e^{-2\tau}\bigr)
 \leq C(\ell )  e^{2\ell^2} e^{-\ell_0^2/16}   e^{-\tau}
\end{equation}
provided that $\tau_0 $ is sufficiently large.    Under our assumption $1 \ll \ell \leq  \ell_0/10$, we  get   $C(\ell)  \, e^{2\ell^2} e^{-\ell_0^2 /16} \leq e^{-\ell^2_0/20}$. 
This gives the $L^\infty$  bound 
\begin{equation}		\label{eqn-f111}
 \| w - W^u \|_{C^0([0, 2\ell])}  =  \| f(\cdot, \tau) \|_{C^0([0, 2\ell])}  \leq e^{-\ell^2_0/20} \, e^{-\tau}
\end{equation}
To get the derivative bound in \eqref{eqn-wexpansion} one interpolates between the  uniform  $C^3$-bound $|e^{\tau} f_{yyy} | = |e^{\tau} w_{yyy} | \leq \Lambda \, (1 + (4\ell)^4) $
 that follows from \eqref{eq-apriori-Linfty}, 
and the $L^\infty$ bound in \eqref{eqn-f111}, and takes $\ell_0 \gg 1$. Hence, the Lemma readily follows in the case
that $\tau \in [\tau_0+1, \tau_1]$.  For more details, see Lemma 4.1 in \cite{AV}. 
 
In the case where $\tau \in [\tau_0, \tau_0 +1]$ one argues similarly to  the proof of Lemma 4.2 in \cite{AV} to obtain  similar  bounds.
\end{proof}

We will next show  an  improvement of \eqref{eqn-uuu}.

 \begin{lemma}\label{cor-lll}
Assume  that $u \in \mc F_{\tau_1}$ and that $
\|W^s (\cdot, \tau)\|    \le e^{-\ell_0^2/16}\, \|W^u(\cdot, \tau)\| 
$
holds for all $\tau\in[\tau_0,\tau_1]$ and $\ell_0 \gg 1$.  Then by taking $\ell_0=\ell_0(M_1, K_0)$ sufficiently large, we get 
\begin{equation}		\label{eqn-uuu23} 
 \frac 34  K_0 \, \ell^4 \, e^{-\tau} \leq     \sqrt{2(n-1)} -  u(\ell,\tau) \leq  \frac 54\,   K_0 \, \ell^4 \, e^{-\tau}  
\end{equation}
for all $\ell \in [\ell_0/10, 1000 \ell_0], \,\,\, \tau\in [\tau_0,\tau_1]$.
 \end{lemma} 

\begin{proof}  Fix  $ \bell:= \ell_0/10$ where $\ell_0$ is sufficiently large.  Lemma \ref{cor-expansion} says that   $w=u-\bu$ satisfies the bound 
\[
|w(\bell,\tau)| \leq  \big(C(M_1) \, \bell^2 + e^{-\ell_0^2/20}\big) e^{-\tau}
\]
for some constant $C(M_1)$ depending  on $M_1$.
Combining this with  the peanut asymptotics shown in subsection \ref{sec:peanut properties}  yield  
\begin{equation*}
\begin{split}
K_0 - C(M_1)  \bell^{-2}  - \bell^{-4} e^{-\ell_0^2/20} + o_\tau(1) &\leq \big (    \sqrt{2(n-1)} - u(\bell, \tau)  \big ) \, e^{\tau } \bell^{-4} \\
&\leq  K_0 +  C(M_1)   \bell^{-2}  + \bell^{-4} e^{-\ell_0^2/20} + o_\tau(1)
\end{split} 
\end{equation*}
for all $\tau \in [\tau_0,\tau_1]$.  By choosing 
 $\ell_0 = \ell_0(M_1, K_0)$  big so that $C(M_1)  \bell^{-2}  - \bell^{-4} e^{-\ell_0^2/20} + o_{\tau}(1) < 0.01\, K_0$,  we  readily see  that 
\begin{equation}		\label{eqn-uuu232} 
(1- 0.01)  \,  K_0 \, \bell^4 \, e^{-\tau} \leq     \sqrt{2(n-1)} -  u(\bell,\tau) \leq  (1+0.01)  \, K_0 \, \bell^4 \, e^{-\tau}.  
\end{equation}
and all $\tau \in [\tau_0, \tau_1]$ provided $\tau_0 \gg 1$. 

We will next use the peanuts as barriers to expand the behavior in \eqref{eqn-uuu232} from $\bell:=\ell_0/10\,\,\,$  to any 
$\ell \in [\ell_0/10, 1000\ell_0]$. To this end, we first observe that  the peanut asymptotics and \eqref{eqn-uuu232} imply that 
\[
\bu(\bell, \tau-\alpha) \leq u(\bell, \tau) \leq  \bu(\bell, \tau + \alpha)
\]
for all $\tau \in [\tau_0, \tau_1]$, if we choose $e^{-\alpha} = 1-0.05$.  Furthermore, we may assume that $\epsilon$ is sufficiently small
so that our initial profile $u(y, \tau_0) := \bu(y, \tau_0) + \epsilon \, ( \Omega_0 + \Omega_2 \, \hhm_2(y) ) \, \eta_0$ satisfies
\begin{equation}		\label{eq-est-between}
\bu(y, \tau_0-\alpha) \leq u(y, \tau_0) \leq  \bu(y, \tau_0 + \alpha)
\end{equation}
for $\ell_0 \ge |y| \geq \bell := \ell_0/10$.  Note that for this to hold we need to have $\epsilon < \frac{0.05}{200} K_0 e^{-\tau_0} \ell_0^2$. 

Recall that since  $\eta_0 =0$,  for $|y| \geq \ell_0$, we have  $u=\bu=0$,  for $|y| \geq \ell_0$,  at $\tau = \tau_0$, and that   the peanut solution satisfies  $\bu_\tau  >0$ outside a fixed compact set,  (see in subsection \ref{sec:peanut property monotone}). 
Now, we can apply the comparison principle with boundary on $|y| \geq \bell$ to conclude that 
\begin{equation}		\label{eqn-comp-peanuts}
\bu(y, \tau-\alpha) \leq u(y, \tau) \leq  \bu(y, \tau + \alpha), \qquad \mbox{on} \,\, |y| \geq \bell:=\frac{\ell_0}{10}, \,\, \tau \in [\tau_0, \tau_1].
\end{equation}

Using \eqref{eqn-comp-peanuts},   the peanut asymptotics on $[\ell_0/10, 1000\ell_0]$,  and the definition of $\alpha$ by $e^{-\alpha} = 1-0.05$,
we obtain that 
\begin{equation}
(1-0.1) \,  K_0 \, \ell^4 \, e^{-\tau} \leq     \sqrt{2(n-1)} -  u(\ell,\tau) \leq (1+0.1) \,   K_0 \, \ell^4 \, e^{-\tau}  
\end{equation}
holds for $y \in [\ell_0/10, 1000\ell_0]$, $\tau \in [\tau_0, \tau_1]$, provided that $\ell_0=\ell_0(M_1, K_0)$ is chosen sufficiently large and
also $\tau_0 \gg 1$.  In  particular this shows that \eqref{eqn-uuu23} holds, thus finishing the proof of the lemma.

\end{proof} 

\begin{remark}
\label{rem-ell}
We can see easily that the same proof as above yields for any fixed $\eta$ small (in the proof above we took $\eta = 0.05$ for simplicity), as long as $\epsilon < \frac{\eta K_0 e^{-\tau_0}\ell_0^2}{200}$ we have
\begin{equation}		\label{eqn-esti-ueta}
(1-\eta) K_0 \ell_0^4 e^{-\tau} \le \sqrt{2(n-1)} - u(y,\tau) \le (1 + \eta) K_0 \ell^4 e^{-\tau},
\end{equation}
for all $y \in [\ell_0/10, 1000 \ell_0]$, and all $\tau\in [\tau_0,\tau_1]$.
\end{remark}

We will now see the $L^\infty$-estimate \eqref{eq-apriori-Linfty} that we have assumed in proving Proposition \ref{lemma-big-unstable}
holds with the auxiliary constant    $\Lambda$ replaced by a constant $M_2$  that depends only on $M_1$, and the peanut constant $K_0$.  

\begin{prop}[$L^\infty$-estimate]\label{prop-Linfty-w}  There exists an $\ell_0 = \ell_0(M_1,K_0)$ large so that if $u \in \mc F_{\tau_1}$ satisfies \eqref{eqn-uuu}
and  $
\|W^s (\cdot, \tau)\|    \le e^{-\ell_0^2/16}\,\|W^u(\cdot, \tau)\| 
$
holds for all $\tau\in[\tau_0,\tau_1]$,
then, the bound 
\begin{equation}		\label{eqn-linftyw-final}
 |w(y,\tau)| + |w_y(y,\tau)| + |w_{yy}(y,\tau)| + |w_{yyy}(y, \tau) \leq M_2 \, (1+ |y|^4)  \,  e^{-\tau}
\end{equation}
holds for all $|y| \leq  2\rho e^{\tau/4}$, $\tau \in [\tau_0,\tau_1]$.  Here the constant $M_2$ depends only on $M_1$ and $K_0$. 
\end{prop}

\begin{proof} 

Our beginning point is \eqref{eqn-comp-peanuts} which implies that for $|y| \geq \ell_0/10$, $\tau \in [\tau_0, \tau_1]$ we have 
\begin{equation}		\label{eqn-est-w12} 
|w(y,\tau) | \leq \max \big ( \bu(y,\tau+\alpha) - \bu(y,\tau), \bu(y, \tau) - \bu(y,\tau-\alpha) \big )
\end{equation}
where $e^{-\alpha}=0.95$. Therefore, by the peanut asymptotics (see subsection \ref{sec:peanut properties})  we get the crude but sufficient bound $\sup_{\ell_0/10 \leq |y| \leq 4\rho e^{\tau/4}} |w(y,\tau)| \leq 4 K_0 \, e^{- \tau}\, |y|^4.$ On the other hand, estimate \eqref{eqn-wexpansion} in Lemma \ref{cor-expansion} shows $\sup_{0 \leq |y| \leq  \ell_0/10} |w(y,\tau)| \leq \Big ( C(M_1) \, |y|^2 + e^{-\ell_0^2/20} \big )\, e^{-\tau}.$ Combining these two bounds, while taking $\ell_0=\ell_0(M_1, K_0)$ sufficiently large, yields
\begin{equation}		\label{eqn-w200}
|w(y,\tau)| \leq C(M_1, K_0)  \,  e^{- \tau}\, |y|^4, \qquad \mbox{on} \,\, |y| \leq 4\rho e^{\tau/4}
\end{equation}
for some constant $C(M_1, K_0)$ that depends only on $M_1$ and $K_0$. For the rest of proof we will denote by $C(M_1, K_0)$ constants that may change from line to line but they only depend on 
$M_1$ and $K_0$. 
 
To estimate the derivatives $w_y$ and $w_{yy}$  for $|y| \leq  2\rho  e^{\tau/4}$ it is easier to prove the same bounds  for our  solution $u$ of \eqref{eq:u}, since 
then the bounds for $w$ will follow by the bounds on $u$ and the bounds on the peanut solution $\bu$. 
To this end, we recall first that $\bu$ satisfies 
\begin{equation}		\label{eqn-bu200}
 |\bu(y,\tau)- \sqrt{2(n-1)}| + |\bu_y(y,\tau)| + |\bu_{yy}(y,\tau)| + |\bu_{yyy}(y,\tau)|  \leq C(K_0) \, (1+ |y|^4)  \,  e^{-\tau}
\end{equation}
which also hold on $ |y| \leq 4\rho e^{\tau/4}$ for some constant $C(K_0)$, depending on $K_0$. Then, the $L^\infty$ bounds in \eqref{eqn-w200} and \eqref{eqn-bu200} imply that 
\begin{equation}		\label{eqn-u200}
|u(y,\tau) - \sqrt{2(n-1)}| \leq C(M_1, K_0)  \,  e^{- \tau}\, |y|^4, \qquad \mbox{on} \,\, |y| \leq 4\rho e^{\tau/4}
\end{equation}
for some other constant  $C(M_1, K_0)$ still depending on $K_0$ and $M_1$. 

One can  now use \eqref{eqn-u200} and derivative estimates for quasilinear equations to obtain the rest of the derivative bounds in \eqref{eqn-u200}.
Indeed this has been done in detail in Lemma 6.2 in \cite{AV} where it was shown the $L^\infty$ bound \eqref{eqn-u200} implies the derivative bounds 
\begin{equation}		\label{eqn-u300}
 |u_y(y,\tau)| + |u_{yy}(y,\tau)| + |u_{yyy}(y, \tau)| \leq C(M_1, K_0)  \, (1+ |y|^4)  \,  e^{-\tau}
\end{equation}
 holding in the region $ |y| \leq 2\rho e^{\tau/4}$. 
 Now combining \eqref{eqn-bu200} and \eqref{eqn-u300} readily give us the derivative bounds in  \eqref{eqn-est-w12}, thus  finishing the proof of the proposition.

%
%
%The estimates on the derivatives
%$w_y$ and $w_{yy}$  for $|y| \leq  2\rho \, e^{\tau/4}$ follow  from  the $L^\infty$-estimate and standard derivative estimates, by  scaling.    
%One first sees that the estimates hold on  some time  interval    $\tau_0 \leq \tau \leq \tau_0 + \sigma$,
%by our choice of initial data and standard short time existence and regularity theory.  Then one applies standard local derivative estimates
%to show that the desired  bounds  hold up to $\tau_1$. For more details, see Lemma 4.1 in \cite{AV}.  

\end{proof}

We can now combine the previous three lemmas to conclude the following result that justifies all our \textit{a priori} assumptions used in this section hence conclude the proof of of Proposition \ref{lemma-big-unstable}.

\begin{lemma}
\label{cor-prop-funnel}
Let $\ell_0 = \ell_0(M_1,K_0)$ be chosen as above so that all previous results hold. Let $\eta > 0$ be a small fixed number. There exists a $\tau_0$ sufficiently big and $\epsilon > 0$ sufficiently small ($\epsilon$  is the size of initial perturbation)  so that if $u \in \mc F_{\tau_1}$,  then $u(y,\tau)$ satisfies the following properties
\begin{enumerate}[\upshape(i)]
\item
$|u(y,\tau)| + |u_y(y,\tau)| + |u_{yy}(y,\tau)| + |u_{yyy}(y, \tau)| \leq 2M_2 \, (1+ |y|^4)  \,  e^{-\tau}$
holds,   $\forall y \in [-2\rho e^{{\tau/4}}, 2\rho e^{ \tau/4}]$, $\tau \in [\tau_0,\tau_1]$.  Here $M_2$ is the same constant as in \eqref{eqn-linftyw-final} and depends only on $M_1$ and $K_0$, and
\item $(1-\eta) K_0 \, \ell^4 \, e^{-\tau} \leq \sqrt{2(n-1)} - u(\ell,\tau) \leq (1+\eta) K_0 \, \ell^4 \, e^{-\tau}$ holds for $\tau\in [\tau_0,\tau_1]$, and $\ell_0 \le \ell \le 1000\ell_0 < e^{\tau_0/4}$.
\item $\|W^s (\cdot, \tau)\| \le e^{-\ell_0^2/8}\, \|W^u(\cdot, \tau)\| $ holds, for all $\tau\in[\tau_0,\tau_1]$.
\end{enumerate}
\end{lemma}

\begin{proof}
We first note that all (i), (ii) and (iii) are satisfied at time $\tau_0$.  More precisely, parts (i) and (ii) easily follow from $u(y,\tau_0) = \bu(y,\tau_0) + \eta_0(y) (\Omega_0 + \Omega_2 \hm_2(y))$ and the asymptotics of peanut solution (note that we can assume part (i) holds at time $\tau_0$ with $M_2$ replaced by $M_2/2$).  Furthermore, by Lemma \ref{lemma-lo} we have $\|W^s(\cdot,\tau_0)\|  \le e^{-5\ell_0^2/32}\, \|W^u(\cdot,\tau_0)\| $, which obviously implies (iii) at time $\tau_0$.  

Assume that $\tilde{\tau}_1 < \tau_1$ is the maximal time up to which (i) holds.  We will first show that (ii) and (iii) need to hold up to time $\tilde{\tau}_1$ as well.  In order to do that, we will first show that \eqref{eqn-uuu} needs to hold all the way to $\tilde{\tau}_1$.  Assume $\tilde{\tau}_2 < \tilde{\tau}_1$ is the maximal time up to which \eqref{eqn-uuu} holds.  Then by Lemma \ref{cor-quotient} we have that (iii) holds for all $\tau\in [\tau_0, \tilde{\tau}_2)$.  Now we can apply Lemma \ref{cor-lll} (more precisely see Remark \ref{rem-ell}) which implies that  actually (ii)  holds for some small but fixed $0 < \eta  < \frac14$, for all $\tau\in [\tau_0,\tilde{\tau}_2)$.  Since the constants in (ii) are sharper than in \eqref{eqn-uuu}, this contradicts the maximality of $\tilde{\tau}_2$.  All these imply that $\tilde{\tau}_2 = \tilde{\tau}_1$, and that \eqref{eqn-uuu}  holds for all $\tau\in [\tau_0,\tilde{\tau}_1]$.  Having (i)  and \eqref{eqn-uuu} for all $\tau\in [\tau_0,\tilde{\tau}_1]$, Lemma \ref{cor-quotient} implies we have (iii) on that time interval as well.  Lemma \ref{cor-lll} (more precisely Remark \ref{rem-ell}) and Proposition \ref{prop-Linfty-w} now imply that actually (i) and (ii) hold for all $\tau\in [\tau_0,\tilde{\tau}_1]$ as well.  

We claim that $\tilde{\tau}_1 = \tau_1$.  If not, let $\tilde{\tau}_1 < \tau_1$ be the maximal time so that (i) holds on $[\tau_0, \tilde{\tau}_1]$.  Similar argument as above implies that (ii) and (iii) hold on $[\tau_0,\tilde{\tau}_1]$.  By Proposition \ref{prop-Linfty-w} we now get that  actually (i) holds for $\tau\in [\tau_0,\tilde{\tau}_1]$ (with constant $M_2 < 2M_2$), hence contradicting the maximality of $\tilde{\tau}_1$.  This implies $\tilde{\tau}_1 = \tau_1$ as claimed.  
\end{proof}

\smallskip

We are now ready to give the proof of Proposition \ref{lemma-big-unstable}.  

\begin{proof}[Proof by contradiction of Proposition \ref{lemma-big-unstable}]
Assume the statement is not true, implying
\begin{equation}		\label{eq-contra}
\|W^u(\cdot, \tau)\|  < M_1 \, e^{-\tau},
\end{equation}
for all $\tau \in [\tau_0,\tau_{\max})$, where $\tau_{\max} \le \infty$ is the maximal existence time for $u(y,\tau)$.  
Having Lemma \ref{cor-prop-funnel}, similar computation to the one in the proof of Lemma \ref{lemma-quotient1} yields
\[
\frac{d}{d\tau} \|W^u\|^2  \ge -\delta \, \|W^u\|^2 
\]
for some very small $\delta >0$, and all $\tau\in [\tau_0,\tau_{\max})$.  
This yields
\begin{equation}		\label{eqn-theta2} 
\|W^u(\cdot, \tau)\|^2   \ge e^{- \delta (\tau -\tau_0)} \,  \|W^u(\cdot, \tau_0)\|^2 
\end{equation}
where $\|W^u(\cdot, \tau_0)\|^2 \geq c_n \, \epsilon^2$ for a dimensional constant $c_n >0$ (here we use Lemma \ref{lemma-lo}).  If $\tau_{\max} = \infty$, since the previous estimate would then hold for all $\tau \ge \tau_0$, and since $\delta$ can be taken to be a small number, inequality \eqref{eqn-theta2} would contradict \eqref{eq-contra} for large values of $\tau$.  Hence, $\tau_{\max} < \infty$.  By part (i) of Lemma \ref{cor-prop-funnel} we conclude that a singularity can not form in the region $|y| \le \rho \, e^{{\tau/4}}$.  On the other hand, by \eqref{eqn-upm2}, which holds for all $|y| \ge \ell_0$ and all $\tau\in [\tau_0,\tau_{\max})$, we conclude the singularity can not happen on $|y| \ge \ell_0$ either.  Hence we conclude there exists the exit time $\tau_1 = \tau_1(\epsilon, \bom) < \tau_{\max}$ satisfying \eqref{eq-big-unstable}.  This concludes the proof of Proposition \ref{lemma-big-unstable}.
\end{proof}

To summarize, let us recall that if $\epsilon > 0$ and $\bom = (\Omega_1,\Omega_2)\in \mathbb{S}^1$, we denote by $u_{\epsilon, \bf{\Omega}}(y,\tau)$ a solution to the Mean Curvature Flow that starts at $\tau_0$ as
\[
\ueo(y,\tau_0) = \bu(y,\tau_0) + \epsilon\, \eta_0(y)\,\big(\Omega_0 + \Omega_{2} \hm_{2}(y)\big).  
\]
In Proposition \ref{lemma-big-unstable} we showed that for every $\epsilon > 0$ small and every $\bom \in \mathbb{S}^1$,  there exists the first time $\tau_1 = \tau_1(\epsilon,\bom)$ so that 
\[
\|W^u(\cdot, \tau_1)\|  = M_1 \, e^{-\tau_1}.  
\]

Moreover we also observe the following.
\begin{lemma}
\label{lemma-hit-boundary}
Let $\tau_1$ be as above, that is the first time so that $\|(W_{\epsilon,\bom})^u(\cdot, \tau_1)\|  = M_1 \, e^{-\tau_1}$ holds.    Then we have that $\frac{d}{d\tau}\big(e^{\tau}\, \|W^u(\cdot,\tau)\|\big) > 0$ at $\tau = \tau_1$.
\end{lemma}

\begin{proof}
By Lemma \ref{cor-prop-funnel} and~\eqref{eq-unstable-useful}, we have that for all $\tau\in [\tau_0,\tau_1]$,
\[
\frac{d}{d\tau}  \big(e^{\tau} \|W_u\|^2\big) \ge (1 - 2\delta) \|W_u\|^2 > 0
\]
as claimed.
\end{proof}

As in \eqref{eqn-defnW}, we let $W_{\epsilon, \bom}(y,\tau) := (\ueo(y,\tau) - \bu(y,\tau))\, \eta(y,\tau)$.  The following result that uses a homotopy degree argument tells us that we can say something more precise about our solution at the exit time $\tau_1$.  

\begin{lemma}[Degree theory lemma] 

\label{lemma-homotopy} Let  $M_1 >0$ be an arbitrary constant.  For every $\epsilon > 0$ small, and every ${\bar\bom} \in \mathbb{S}^1$,   there exist  an $\bom \in \mathbb{S}^1$, and a $\tau_1 = \tau_1(\epsilon, \bom)$, such that
\[
\begin{split}
&\|W_{\epsilon,\bom}^u(\cdot, \tau_1)\|  = M_1 \, e^{-\tau_1} \qquad \mbox{and} \\  &W_{\epsilon,\bom}^u(\cdot,\tau_1)
= M_1 \,e^{-\tau_1}\, \left(\bar{\Omega}_0 \, \frac{\hm_0}{\|\hm_0\|}
+ \bar\Omega_2\, \frac{\hm_2}{\|\hm_2\|}\right).  
\end{split}
\]
\end{lemma}

\begin{proof}
 Let $\ueo(y,\tau)$ be the mean curvature flow solution starting at $u_{\epsilon,{\bf \Omega}}(y,\tau_0)$.  Let $\tau_1 = \tau_1(\epsilon,\bom)$ be chosen as in Proposition \ref{lemma-big-unstable}, so that $\|W_{\epsilon,\bom}^u(\cdot, \tau_1)\|  = M_1 \, e^{-\tau_1}$.  Fix $\epsilon > 0$ small. Let $\mathbb S^1_{\R^2}\subset\R^2$ and $\mathbb S^1_{\hilb^u}$ be the unit circles in $\R^2$ and $\hilb^u$, respectively, and define a map $\mc F: \mathbb{S}^1_{\R^2} \to \mathbb{S}^1_{\hilb^u}$ as follows
\[
\mc F(\Omega_0, \Omega_2) = \frac{W_{\epsilon, \bom}^u(\tau_1)}{\|W_{\epsilon, \bom}^u(\tau_1)\|_\hilb}.  
\]
The map $\mc F$ is well defined because Lemma~\eqref{lemma-quotient1} implies that $W_{\epsilon,\bom}(\tau)\neq0$ as long as $W(\tau)$ lies in the funnel.

Our goal is to show that the map $\mc F$ is surjective.  In order to do that, define a map $\mc G : \mathbb{S}^1  \to \mathbb{S}^1$ as follows
\[
\mc G(\bom) = \frac{W_{\epsilon, \bom}^u(\tau_0)}{\|W_{\epsilon, \bom}^u(\tau_0)\|}.  
\]
It is straightforward to check that the map $\mc{G}({\bf\Omega})$ is a bijection. Thus in order to achieve our goal it is sufficient to define a homotopy between the maps $\mc{F}({\bf\Omega})$ and $\mc{G}({\bf\Omega})$.  Lemma~\eqref{lemma-quotient1} tells us that $W_{\epsilon,\bom}^u(\tau) \neq 0$ for all $\epsilon, \bom$and for as long as $W_{\epsilon,\bom}(\tau)$ lies in the funnel.  Therefore we may define a map $\mc H_1: \mathbb{S}^1_{\R^2}\times[0,1] \to \mathbb{S}^1_{\hilb^u}$ by
\[
\mc H_1(\bom, s) = \frac{W_{\epsilon, \bom}^u(s\tau_0 + (1-s)\tau_1)}{\|W_{\epsilon, \bom}^u(s\tau_0 + (1-s)\tau_1)\|}.
\]
We claim $\mathcal{H}_1$ is indeed a homotopy between $\mc F$ and $\mc G$. To show this, fix an $\epsilon > 0$ small. We will first argue that the map $\mc H_1$ is continuous.  We claim that the exit time $\tau_1 = \tau_1(\epsilon, \bom)$ is a continuous function of ${\bf \Omega}$. To show this claim we use Lemma \ref{lemma-hit-boundary} and argue as in the proof of Lemma 3.2 in \cite{AV}.  It also has the property that $\mc{H}_1(\bom,0) = \mc F(\bom)$, and $\mc H_1(\bom,1) = \mc G(\bom)$, hence making it the homotopy between the maps $\mc F$ and $\mc G$.  This concludes the proof of the Proposition.
\end{proof}

Using Lemma \ref{lemma-homotopy} we have two goals: one  goals is to show that,  for every $\epsilon > 0$,   we can actually choose initial data at time $\tau_0$  so that our solution starting at that initial data {\em becomes convex at time $\tau_1$.}  Huisken's theorem then  implies this solution develops a singularity that is modeled on a round sphere $\mathbb S^n$.  Our second goal is to show
that at  the same time, for  every $\epsilon > 0$ we can choose initial data at $\tau_0$, so that the solution starting at that initial data develops a {\em nondegenerate neckpinch}  singularity.  We address the formation of spherical singularities in  Section \ref{sec-sphere}, and the formation of nondegenerate neckpinch singularity in Section \ref{sec-main-thm}.

%%%%%%%%%%%%%%%%%%%%%%%%%%%%%%%%%

\section{Initial data developing spherical singularities}
\label{sec-sphere}

 Let $\tau_0$, $\ell_0$ be chosen large enough so that all our previous estimates hold.   We start with the following observation.
 
\begin{lemma}		\label{lem-convexity}
For any $(\epsilon, \bom)$, let $u=\ueo$ and $\tau_1=\tau_1(\epsilon,\bom)$ 
be as in Proposition \ref{lemma-big-unstable}.   
Then, there exist $r_0 = r(K_0) \gg 1$ and $\tau_0 \gg 1$ sufficiently large so that, for every small $\epsilon$ and all $\tau_0 \leq \tau \leq \tau_1$, we have
\begin{equation}		\label{eqn-conv-r0}
u_{yy}(y,\tau) \leq 0, \qquad \mbox{for} \,\, |y| \geq  r_0.  
\end{equation}
This holds provided $\ell_0$ is chosen so that $\ell_0 \gg r_0$.  
\sk

In addition, if we assume that the unstable mode at time $\tau_1$, given by
\[
W^u(y,\tau_1) =\bigl(d_0(\tau_1) + d_2(\tau_1) \hm_2(y) \bigr) e^{-\tau} ,
\]
satisfies
\begin{equation}		\label{eq:10}
d_2(\tau_1)<0 \quad\text{ and }\quad |d_0(\tau_1)|\, \|\hm_0\| \leq M_1/2,
\end{equation}
and if we also assume that $M_1=M_1(K_0)$ was chosen sufficiently large, then we have
\[
u_{yy}(\cdot,\tau_1) <0 \text{ for all }y\in\R,
\]
i.e., our solution is convex at $\tau_1$.

\end{lemma} 

\begin{proof} 
Note that from the definition of our initial data $u(y,\tau_0)$, we have $w(y,\tau_0) =0$ for $|y| \geq 2\ell_0$, and that $ u_{yy}(y,\tau_0) = \bu_{yy}(y,\tau_0) \leq 0$ on $|y| \geq 2\ell_0$.   It was shown in \cite{AV} that there exists $r_0 >0$, depending only on $K_0$ and dimension, such that $\bu_{yy}(y,\tau) \le 0$, for all $|y| \ge r_0$ and all $\tau \ge \tau_0$, provided that $\tau_0 \gg1$.   At the end of the proof we will choose $M_1$  and $r_0$  to depend on  $K_0$ and universal constant.  We will also need to have $\ell_0 \geq 20 r_0$ so that we can apply  Lemma \ref{cor-expansion}.

We need to choose $r_0$ sufficiently large (still depending on $K_0$)  and $\ell_0 \gg r_0$.   Then, the maximum principle with boundary, applied to $u_{yy}$, shows that in order to prove \eqref{eqn-conv-r0}, it is sufficient to show that:
\begin{enumerate}[(i)]
\item $u_{yy}(r_0, \tau) \leq 0$, for all $\tau \in [\tau_0, \tau_1]$, and
\item $u_{yy}(y,\tau_0) \le 0$ for all $r_0 \le |y| \le 2\ell_0$ (since $u_{yy}(y, \tau_0) = \bu_{yy}(y, \tau_0) $ for $|y| \geq 2\ell_0$).
\end{enumerate}
Let us prove these last two claims next.

\smallskip

Claim (i) follows from Lemma \ref{cor-expansion}.  Indeed, since $w=u-\bu$ and $W^u (\cdot, \tau) = \bigl(d_0(\tau) + d_2(\tau) \, \hm_2(y)\bigr)e^{-\tau}$, the estimate \eqref{eqn-wexpansion} applied on $[0,2r_0]$ implies the bound
\[
\big|u_{yy}(y,\tau) - \bu_{yy} (y, \tau) - d_{2}(\tau)\hm_{2}'' (y) e^{-\tau}| \le e^{-\ell_0^2/30}\, e^{- \tau}.
\]
provided that $2r_0 \leq \ell_0/10$.  
This, combined with the peanut asymptotics~\eqref{eq:um2-derivs}, implies
\[
|\bu_{yy} (y, \tau) +  K_0 \, \hm_4''(y)  e^{-\tau}| \le  C_0 \, r_0^2  \,   \delta  \, e^{-  \tau}
\]
for some universal $C_0$.  
We conclude that for all $|y| \leq 2r_0$, and all $\tau \in [\tau_0, \tau_1]$, we have
\begin{equation}		\label{eqn-uuu2}
\big| e^{\tau}\, u_{yy}(y,\tau) +    K_0 \,  \hm_4'' (y)- d_2(\tau)\hm_2''(y)\big| \le   e^{-\ell_0^2/30} + \delta \, C_0 \, r_0^2.   
\end{equation}
Let us evaluate the above at $y = r_0$.   Note that we have that $r_0 \ll \ell_0$, but we can still take $r_0$ big enough so that%
\footnote{~Since, by \eqref{eq:Hermite polynomials explicit form}, $\hm_4(y) = y^4-12y^2+12$, we have $\hm_4''(y) = 12y^2-24$.   If $y\geq 2$ then $24<6y^2$ so $\hm_4''(y)\geq 6y^2$.   For the subsequent estimate involving $\hm_2$ we recall $\hm_2(y) = y^2 - 2$ and thus $\hm_2''(y) = -2$.}%
$\hm_4''(r_0 ) \geq 6 r_0^{2} $ and furthermore, because $d_0(\tau)^2\|\hm_0\|^2 + d_2(\tau)^2 \|\hm_2\|^2 \leq M_1^2$ for $\tau \in [\tau_0, \tau_1]$, we have $\big| d_{2}(\tau)\, \hm_{2}''(r_0)\big| = 2|d_2(\tau)| \leq \frac{2}{\|\hm_2\| } M_1$.   Hence, we can guarantee that $u_{yy}(r_0, \tau ) < 0$, for all $\tau \in [\tau_0, \tau_1]$ provided
\[
-  K_0   \hm_4''(r_0) + d_2(\tau)\hm_{2}''(r_0) \leq - cK_0 \, r_0^2   + \frac{2}{\|\hm_2\|} M_1 < - 2
\]
and
\[
 e^{-\ell_0^2/30} + \delta \, C_0 \, r_0^2<1.
\]

To satisfy the first inequality we need to take  $r_0^2 > C(K_0)  \, (M_1 +1)$ (where $M_1$ will be taken in the sequel  to be a  constant
depending on  $K_0$).  With such a choice
of $r_0$ we can guarantee the second inequality by taking  $\delta $ sufficiently small depending only on  $K_0$ and  $M_1$.

\smallskip 
Now let us  check that (ii) holds, that is $u_{yy}(y,\tau_0) \le 0$,  for all $|y| \geq r_0$.  By the definition of $u(\cdot, \tau_0)$ we have that $u=\bu$ for $|y| \geq 2\ell_0$.  Hence, for $|y| \geq 2\ell_0$ we have  $u_{yy}(y,\tau_0) = \bu_{yy}(y,\tau_0) \leq 0$ from our choice of $r_0$.  For 
$r_0 \le |y| \le 2\ell_0$, using $u(y,\tau_0) = \bu (y, \tau_0) + \epsilon\,\eta_0 \, \big(\Omega_0 + \Omega_2 \hm_2(y)\big)$, we compute 
\[
u_{yy}(y,\tau_0) = \bar{u}_{yy}(y,\tau_0) + 2\epsilon \, \Omega_2 \, \eta_0 + \epsilon (\eta_0)_{yy}(\Omega_0 + \Omega_2\hm_2)   + 4\epsilon y (\eta_0)_y\Omega_2 \le 0,
\]
using the asymptotics of $\bar{u}(y,\tau_0)$, if we choose $\epsilon$ sufficiently small (compared to $e^{-\tau_0}$).  Combining the two cases shows that $u_{yy}(y, \tau_0) \leq 0$, for $|y| \geq r_0$, thus concluding the proof that \eqref{eqn-conv-r0} holds.

\smallskip 
Together (i), (ii) and  the maximum principle  yield \eqref{eqn-conv-r0}, thus finishing the proof of the first part of the Lemma.  
\sk

To show the second part of the Lemma, assume that at $\tau_1$ we have $d_2(\tau_1) <0$, and $d_0(\tau_1)^2\|\hm_0\|^2 + d_2(\tau_1)^2\|\hm_2\|^2 = M_1^2$, implying that $|d_2(\tau_1) |\;\|\hm_2\| \ge M_1/2$, since we are assuming $|d_0(\tau_1)|\; \|\hm_0\|\le M_1/2$.   Then, $d_2(\tau_1) \hm_2''(y) < - M_1/\|\hm_2\|$, and thus, similarly as above, we obtain that for $|y| \le r_0$ and $r_0 \ll \ell_0$,
\begin{equation}		\label{eq:15}
 e^{\tau_1}\, u_{yy}(y,\tau_1) \leq -K_0 \hm_4''(y)- M_1  + e^{-\ell_0^2/30}
 \leq 24K_0  -  M_1 + e^{-\ell_0^2/30}
\end{equation}
because $\hm_4(y)=y^4-12y^2+12$ implies $\hm_4''(y) = 12y^2-24\geq -24$.    Thus the right hand side in \eqref{eq:15} is indeed negative provided
\begin{equation}		\label{eq:16}
M_1 > 24K_0 +  e^{-\ell_0^2/30}.
\end{equation}
We then conclude, from the above discussion, that if $d_2(\tau_1) <0$ and $|d_0(\tau_1)|\;\|\hm_0\| \le M_1/2$, then at $\tau_1$, $u_{yy} (y,\tau_1) <0$, everywhere on our hypersurface, meaning that our solution is convex at $\tau_1$.  

\smallskip
Note that since $\ell_0 \gg 1$, to guarantee that \eqref{eq:16} holds, it is sufficient to choose $M_1 = 24 K_0 + 1$ and with that choice of $M_1$, we required above that $r_0^2 > C(K_0) \, M_1$.  Further more we need to take (depending on $K_0$) $r_0$ so that $\bu_{yy}(y,\tau) \le 0$, for all $|y| \ge r_0$, and $\tau \ge \tau_0$, provided that $\tau_0 \gg1$ (the existence of such $r_0$ is shown in \cite{AV}.
\end{proof}

%We now prove the following Proposition.

Let $\bar M_{\epsilon,\bom}(t)$ ($0\leq t<T_{\epsilon,\bom}$) be the solution to MCF that starts from the hypersurface of rotation with profile $u_{\epsilon,\bom}(\cdot, \tau_0)$.  It is a surface of rotation with profile function
\[
r = \sqrt{1-t}\; u_{\epsilon, \bom_\epsilon}\Bigl( \frac{x}{\sqrt{1-t}}, \tau_0 - \ln(1-t)\Bigr)
\]
for $t<\min\{1, T_{\epsilon,\bom}\}$.  

\begin{prop}		\label{prop-sphere}
For every small enough $\epsilon>0$ there exists an $\bom_\epsilon \in \mathbb S^1$ such that the Mean Curvature Flow $\{\bar M_{\epsilon,\bom_\epsilon}(t)\}$ develops a spherical singularity.
\end{prop}
\begin{proof}
For any sufficiently small $\epsilon>0$ Lemma~\ref{lem-convexity} provides an $\bom_\epsilon\in\mathbb S^1$ such that the Rescaled MCF $u_{\epsilon, \bom_\epsilon}$ becomes convex at time $\tau_1(\epsilon, \bom_\epsilon)$.  Since the RMCF given by $u_{\epsilon,\bom}$ and the non-rescaled MCF $\bar M_{\epsilon,\bom}(t)$ differ only by rescaling and a time change, it follows that $\bar M_{\epsilon, \bom_\epsilon}(t)$ is convex when $t=t_1(\epsilon)$ where
\[
t_1(\epsilon) = 1- e^{-(\tau_1(\epsilon, \bom_\epsilon) - \tau_0)}.
\]
Huisken's theorem on convex Mean Curvature Flow \cite{Hu} then implies that $\bar M_{\epsilon, \bom_\epsilon}(t)$ is convex for all $t\in[t_1(\epsilon), T_{\epsilon,\bom_\epsilon})$ and shrinks to a ``round point'' as $t \nearrow T_{\epsilon, \bom_\epsilon}$.
\end{proof}

\section{Upper and lower barriers  needed in the cylindrical case}
\label{sec-sub-super}

To conclude the proof of Theorem \ref{thm-main} we still  need to find perturbations of peanut solution developing nondegenerate neckpinch singularities. The first step in finding those perturbations is the construction of families of subsolutions and supersolutions that will be used in Sections \ref{sec-right-bc}
and \ref{sec-main-thm}  as  lower and upper barriers, respectively, thus providing good asymptotic behavior of our perturbations.  In this section we will carry out the construction of those subsolutions and supersolutions.

\subsection{Set up for our barriers and formal analysis}

Let $\varepsilon > 0$ be a very small constant and assume that  $\tau_0$ is any  large
number so that  $\varepsilon \, e^{\tau_0} \ll 1$.  Denote by $\tau_2 > \tau_0$ the time when 
\begin{equation}		\label{eqn-tau2} \varepsilon \, e^{\tau_2/2 } =\sigma_n, \qquad \mbox {for fixed  large $\sigma_n$}.  
\end{equation}
The constant  $\sigma_n$ will be chosen in the proof of Theorem \ref{thm-cylinder} and will depend only on $K_0$ and dimension $n$.  
 We will use the notation 
\[
\sigma(\tau) :=\varepsilon \, e^{\tau/2 }, \qquad \mbox {where $\tau \leq \tau_2$, implying that $0 \leq  \sigma \leq \sigma_n$}
\]
and we simply use the notation $\sigma=\sigma(\tau)$ keeping in mind that $\sigma$ depends on $\tau$.  

For any $\ell_1, \ell_2$  large constants, we  define the \emph{intermediate region}   
\[
\cI_{\ell_1, \ell_2}  = \big  \{ (y, \tau):\,\, |y| \geq \ell_1, \,\,  2(n-1) + q_0 \geq \ell_2 \, e^{-\tau/2}, \,\,\,\,  \tau \in [\tau_0, \tau_2] \, \big \}
\]
and the \emph{tip region}   
\[
\cT_{\ell_2}  = \big  \{ (y, \tau):\,\, 2(n-1) + q_0 \leq \ell_2 \, e^{-\tau/2}, \,\,\,\,  \tau \in [\tau_0, \tau_2] \, \big \}.  
\]
Note that $\ell_1$ marks the end of parabolic region, while $\ell_2$ marks the end of the intermediate region and the beginning of the tip region.

Our goal is to define upper and lower barriers for solutions $u(y,\tau)$ to  the rescaled MCF equation \eqref{eq:u} 
that are perturbations of the peanut solution.  These barriers will be defined for all $|y| \geq \ell_1$, where $\ell_1$ is a sufficiently large uniform constant.  

Function   $q(y, \tau) := u(y, \tau)^2-2(n-1)$ satisfies a nice equation:
\[
q_\tau = \frac{q_{yy}-2u_y^2}{1+u_y^2} - \frac y2 q_y + q 
 = q_{yy} - \frac y2 q_y + q - \frac{q_{yy} + 2}{1+u_y^2}u_y^2
\]
Use $2uu_y = q_y$, so
\[
  \frac{u_y^2}{1+u_y^2} = \frac{q_y^2}{4u^2 + q_y^2} = \frac{q_y^2}{8(n-1) + 4q + q_y^2} 
\]
and hence
\begin{equation}		\label{eq:q}
q_\tau = q_{yy} - \frac y2 q_y + q  -  \frac{(q_{yy} + 2)q_y^2}{8(n-1) + 4q + q_y^2}.  
\end{equation}

Equation \eqref{eq:q} has an advantage over \eqref{eq:u}, namely, the nonlinear terms are bounded by $q_y^2$ and $q_y^2q_{yy}$, i.e. by the derivatives of the solution.  This means that they are still small in the intermediate region where the $y$-derivatives are small.  On the other hand $u-\sqrt{2(n-1)}$ is not small in the intermediate region, so the nonlinear terms in the usual linearization of \eqref{eq:u} cannot be ignored.  

\subsubsection{Guessing the shape of $q(y,\tau)$} Expand $q(y,\tau)$ in Hermite polynomials:
\[
  q(y, \tau) = a_0(\tau)H_0(y) + a_2(\tau)H_2(y) + a_4(\tau)H_4(y) +\cdots
\]
One formally expects to obtain the following  ODEs for projections $a_j(\tau)$:
\[
  a_0' =  a_0+\cdots, \quad
  a_2' = \cO(a_2^2) + \cdots, \quad
  a_4' = -a_4+\cdots
\]
For now assume $a_0=0$, since this will be controlled in terms of the dominating modes.  Assume that $a_2$ starts out with $a_2(\tau) = \varepsilon$.  Then $a_2(\tau) \sim \frac{\varepsilon}{1+C\varepsilon\tau}$, so that we may assume $a_2(\tau) \approx \varepsilon$ if $\tau\ll \varepsilon^{-1}$.  
The equation for $a_4$ tells us that we should expect  $a_4(\tau) \approx Ke^{-\tau}$.  Thus the  approximate solution is 
\begin{equation}
 q(y, \tau) \approx \varepsilon H_2(y) - Ke^{-\tau}H_4(y)  
 = \varepsilon(y^2-2) - Ke^{-\tau} (y^4-12y^2+12).  
\end{equation}
The approximation should be valid for $\tau \in [ \tau_0, \tau_2]$.  

Since we are interested only in  large values of $|y|$ we approximate $H_k(y)\approx y^k$ and get
\[
  q(y, \tau) \approx \varepsilon y^2 - K e^{-\tau}y^4 =: q_0.  
\]
This approximation cannot be good beyond the point where $q_0=-2(n-1)$, because $u^2=q_0+2(n-1) \geq 0$.  This happens when $ K y^4 - \varepsilon e^{\tau} y^2  -2(n-1)e^{\tau}=0$, i.e.  when $y \approx Y_0(\tau)$ where
\begin{equation}		\label{eqn-ytau}
  Y_0(\tau)^2  = \frac {e^\tau}{2K} \bigl\{\varepsilon + \sqrt{\varepsilon^2 + 8(n-1) K e^{-\tau}}\bigr\}
\end{equation}
We can simplify this when $\tau$ is not too large, and when $\tau$ is very large.  
If $\varepsilon^2 \ll e^{-\tau}$ then $Y_0(\tau)^2 \approx e^{\tau/2} \, \sqrt{\frac{2(n-1)}K}$.  If $\varepsilon^2 \gg e^{-\tau}$ then $Y_0(\tau)^2 \approx \frac{\varepsilon}K e^\tau$.  	

So we have
\begin{equation}		\label{eqn-tip0}
Y_0(\tau) \approx
\begin{cases}
\sqrt[4]{2(n-1)K^{-1 }}  \, e^{\tau/4}   & (\tau_0\leq \tau \ll 2\ln\frac{1}{\varepsilon})  \\ \
\sqrt{ \varepsilon K^{-1}}\, e^{\tau/2} & (2\ln\frac{1}{\varepsilon} \ll \tau \leq \tau_2)
\end{cases}
\end{equation}
where $\tau_2$ is given by \eqref{eqn-tau2}.  

On the interval $|y|\leq Y_0(\tau)$ our approximation for $q$  is  valid, except near the end point where we expect to find Bowl solitons.  Thus we get the following approximation for $u$:
\begin{equation}
	u(y, \tau) \approx  \sqrt{2(n-1) + \varepsilon y^2 - K e^{-\tau}y^4}
\end{equation}
for large $|y| \geq \ell_1 \gg 1$.  The maximum occurs at $y_{\max} =\sqrt{ \frac \varepsilon {2K}} e^{\tau/2}$; at this maximum we have
\begin{equation}		\label{eqn-height}
  u(y_{\max},\tau) \approx \sqrt{2(n-1) + \tfrac {\varepsilon^2 e^{\tau}}{4K}}.  
\end{equation}

This results in the following observation which turns out to be important for our purposes.

\begin{remark}\label{rem-height} [The significance of estimate \eqref{eqn-height}]
Thus by choosing $\tau_2 $ so that $\varepsilon \, e^{\tau_2/2 } =\sigma_n \gg 1$ we can guarantee using \eqref{eqn-height} that  $u_{\max} (\tau_2) \approx \frac {\sigma_n}{2\sqrt {K} } \gg 1$. This  would allow us at  the end of the paper when we prove Theorem \ref{thm-main} 
to pass  sufficiently large  rough barriers from below at  $\tau_2$,  which in turn would imply that our solution develops a non-degenerate neckpinch.  
\end{remark}

\subsubsection{ The scaling at the tip region}\label{ss-stip}

Assuming that $q(y,\tau) \approx q_0(y,\tau)$ for $|y| \geq \ell_1$, the two tips occur approximately at  $\pm Y_0(\tau)$, where $Y_0(\tau)$ satisfies \eqref{eqn-tip0}.  Let us  focus on the left tip $-Y_0(\tau)$ and find the scaling at
which we see the bowl soliton.  For the moment we call this $\alpha(\tau)$.  Note that the results in \cite{AV}  imply that when $\varepsilon=0$, then $\alpha(\tau) \sim e^{-\tau/4}$.  
However, as $\varepsilon$ increases $\alpha(\tau)$ changes depending on $\varepsilon$.

Assuming that $\alpha(\tau)$ is the suitable scaling at the tip, we define the rescaled solution  
$w(\xi,\tau), \xi >0$ by 
\begin{equation}		\label{eqn-wxi1}  w(\xi, \tau)= \alpha(\tau) \, u(y,\tau), \qquad \xi = (-Y_0(\tau) + y)\, \alpha(\tau) >0.  
\end{equation}
Note that by definition $w(0,\tau)=0$, for all $\tau$.  
Lets compute the equation of $w(\xi,\tau)$ from the equation of $u(y,\tau)$ in \eqref{eq:u}.  
We use:
\[
y=-Y_0(\tau) + \alpha^{-1}  \xi, \quad u(y,\tau)=\alpha^{-1} \, w(\xi,\tau), \quad u_y=w_\xi, \quad u_{yy} = \alpha \, w_{\xi\xi}
\]
and
\[
u_\tau = \frac 1{\alpha}\, w_\tau + \Big ( \frac{\alpha'}{\alpha^2} \, \xi + Y_0'(\tau) \Big ) \, w_\xi - \frac{\alpha'}{\alpha^2} \, w.  
\]
Plugging the above in \eqref{eq:u} we find 
\begin{equation*}
\begin{split}
\frac 1{\alpha}\, w_\tau + \Big ( \frac{\alpha'}{\alpha^2} \, \xi &+   Y_0'(\tau) \Big ) \, w_\xi - \frac{\alpha'}{\alpha^2} \, w  \\
&= \alpha \, \frac{w_{\xi\xi}}{1 + w_\xi^2} - \frac 12 \, \big ( -Y_0(\tau) + \frac \xi{\alpha}\big ) \, w_\xi + \frac 1{2\alpha} \, w - \alpha \, \frac{n-1}w.  
\end{split}
\end{equation*}
We divide by $\alpha$ and  rearrange terms to express the above equation as:
\begin{equation}		\label{eqn-w202} 
\frac 1{\alpha^2}\, w_\tau + \frac 1{\alpha^2} \Big ( \frac{\alpha'}{\alpha}  + \frac 12 \Big )  \big (  \xi \, w_\xi -   \, w \big ) =\frac{w_{\xi\xi}}{1 + w_\xi^2} -  \frac{n-1}w  +   \frac {  \frac 12 Y_0(\tau) - Y_0'(\tau) }{\alpha(\tau)}\, w_\xi.  
\end{equation}
We want to choose $\alpha(\tau)$ so that the rescaled solution  $w(\xi,\tau) \to   W(\xi)$, where $W(\xi), \xi >0$ denotes the 
profile (in our coordinates) of the  translating bowl soliton of speed one.  $W(\xi)$ satisfies the equation 
\begin{equation}		\label{eqn-Ws} 
\frac{W_{\xi\xi}}{1 + w_\xi^2} -  \frac{n-1}W +  1 \cdot W_\xi =0, \qquad W(0)=0
\end{equation}
and the asymptotic behavior
\[
W^2(\xi)   = 2(n-1)  \xi + (n-1)   \ln  \xi  + \cO(\ln \xi), \mbox{as} \,\, \xi \to +\infty
\]
which can be differentiated in $\xi$.  
%$$W_\xi^2  (\xi)    = 2 (n-1)  + (n-1) \xi^{-1}  + o(\xi^{-1})$$
In addition $W(\xi) \sim \sqrt{\xi}$ near $\xi \sim 0$.  All these  asymptotics  follow from the fact that,  after switching coordinates, that is, write $\xi=\cB(w)$,  the profile $\cB(w)$ of the   Bowl soliton of speed one
is smooth at the origin and satisfies the asymptotic behavior  
\begin{equation}		\label{eqn-Ba}
\cB(w) = \frac {w^2}{2(n-1)} - \ln w + \cO(w^{-2}), \qquad \mbox{as}\,\, w \to +\infty
\end{equation}
 which can be differentiated in $w$ (see Proposition 2.1 in \cite{AV}).

Going back to \eqref{eqn-w202}, in order to have $w(\xi,\tau)  \approx W(\tau)$ we would need to
 choose $\alpha(\tau) >0$ such that 
\[
\frac {   Y_0'(\tau)  -  \frac 12 Y_0(\tau)}{\alpha(\tau)} = - 1   \cdot  (1+ o(1)), \qquad \mbox{as} \,\, \tau \to \infty.  
\]
Note that 
\begin{equation}		\label{eqn-EW}
E[W] = \frac 1{\alpha^2} \, \Big ( \frac{\alpha'}{\alpha}  + \frac 12 \Big ) \, \big ( \xi \, W_\xi -  W ).  
\end{equation}
is expected to  be an error term.

\begin{remark}\label{rem-atau}  For  $\varepsilon^2 \ll e^{-\tau}$ we have $Y_0(\tau) \approx \sqrt[4]{\frac {2(n-1)}K } \, e^{\tau/4} $, hence we can take 
$\alpha(\tau) := \frac 12 Y_0(\tau) - Y'_0(\tau) = \frac 14 \, \sqrt[4]{\frac {2(n-1)}K } \, e^{\tau/4}$, that is $\alpha(\tau) \sim e^{\tau/4}$  (as in \cite{AV}).  On the other hand, for $\varepsilon^2 \gg e^{-\tau}$, we have $Y_0(\tau) \approx \sqrt{ \frac \varepsilon K}  e^{\tau/2}$ 
and at first order $Y_0'(\tau)-\frac 12 Y_0(\tau) =0$.  So we will need to compute  $Y_0'(\tau)-\frac 12 Y_0(\tau)$ up to the second order, and we will do that 
for both cases above together, that is for all $\tau \in [\tau_1, \tau_2]$ independently from  how $\varepsilon^2$ compares to $e^{-\tau}$.  
\end{remark}

\begin{lemma}\label{lemma-atau}  Let $\alpha(\tau) := \frac 12 Y_0(\tau) - Y'_0(\tau)$ and $\sigma:= \varepsilon e^{\tau/2}$, $\brho=\sqrt{\sigma^2  + 8(n-1) K}$.  Then,  
\begin{equation}		\label{eqn-atau}
\alpha(\tau)=  \frac 14 \, Y_0(\tau) \cdot   \frac{\brho- \sigma}{ \brho} = \frac{(n-1) \sqrt{2K} }{ \brho \, \sqrt{\sigma + \brho}  }  \, e^{\tau/4} = c_1 \, e^{\tau/4}
\end{equation}
where we have used that 
\begin{equation}		\label{eqn-Ytau}
  Y_0(\tau)  = \sqrt{ \frac {\sigma + \brho}{2 K}}\, e^{\tau/4}.  
\end{equation}
Note that 
$c_1= c_1(\sigma)  \sim 1$, since $\sigma \leq \sigma_n$.   
\end{lemma} 

\begin{proof} First note that \eqref{eqn-ytau} and $\sigma=\varepsilon e^{\tau/2}$, $\brho=\sqrt{\sigma^2  + 8(n-1)K}$ imply \eqref{eqn-Ytau}.  
Furthermore, since  the tip $Y_0(\tau)$  satisfies  equation 
\[
2(n-1) + \varepsilon \, Y_0(\tau)^2 - K e^{-\tau} \, Y_0(\tau)^4=0
\]
by differentiating  in $\tau$ and solving  for $Y_0'(\tau)$ we  obtain 
$Y_0' =   \frac  {Y_0^3} { 2  (2\,  Y_0^2 - \varepsilon\, K^{-1} \,  e^{\tau}  )}.  $
Hence, 
\[
Y_0' - \frac 12 Y_0 =  \frac  { Y_0 \big ( -  Y_0^2 + \varepsilon K^{-1} e^{\tau} \big )} { 2 \, \big (2\,  Y_0^2 - \varepsilon K^{-1} e^{\tau} \big )}= \frac  { Y_0\big ( -  Y_0^2 + \sigma  K^{-1} e^{\tau/2} \big )} { 2 \, \big (2\,  Y_0^2 - \sigma K^{-1} e^{\tau/2} \big )}  
\]
Using \eqref{eqn-Ytau} we  obtain \eqref{eqn-atau}.  
\end{proof} 

\smallskip

\begin{remark}[The range of the intermediate region] Since $8(n-1) + 4q_0 \approx 4u^2 := 4 \alpha^{-2} w(\xi, \tau)$, where $\xi := \alpha \, |\pm Y_0 -  y| $ and we expect $w(\xi,\tau) \approx W(\xi)$ (the Bowl soliton) 
and $q_{0y}^2 \sim \alpha(\tau)^{-2}$, we see that   $8(n-1) + 4q_0 \gg q_{0y}^2$ iff $\alpha^{-2} w(\xi, \tau) \gg \alpha^{-2}$, iff $W(\xi) \gg 1$.  
Using the behavior $W(\xi) \sim \sqrt{\xi}$ for $\xi \gg 1$, we conclude that the intermediate region $8(n-1) + 4q_0 \gg q_{0y}^2$ reaches up to the soliton region.  

\end{remark} 
\smallskip

\subsection{Statement of our result} 

The equation for $q$ is \eqref{eq:q}.  Ignoring the nonlinear terms we have a solution
\[
 \bq_0(y, \tau) = \varepsilon H_2(y) - K e^{-\tau} H_4(y).  
\]

Our formal analysis above indicates that  $\bq_0(y, \tau)$ is a good approximation in the intermediate  region $\cI_{\ell_1, \ell_2}$ 
which extends up to  order $\cO(e^{-\tau/4})$ close to the tip.  Thus we will seek for super and subsolutions that are small perturbations
of $\bq_0$.  

First,  by plugging $\bq_0$ in equation \eqref{eq:q} and using that $\bq_0 $ satisfies  $\bq_{0\tau} =  \bq_{0yy} - \frac y2 \bq_{0y} + \bq_{0}$ we conclude that 
\emph{$\bq_0$ is a supersolution of equation \eqref{eq:q}}  for all $|y| \geq \ell_1$.  
Hence, we will use this for  a supersolution,  
and we rename it $Q^+_{\varepsilon, K}$ to indicate its dependence on $\varepsilon, K$.  

To   construct  a subsolution  $q^{-}$  to \eqref{eq:q} for   $|y| \geq \ell_1 \gg 1$ we note that  in this region we have  
$\bq_0 (y, \tau) \approx \varepsilon \, y^2 - K e^{-\tau} y^4$.  Hence, we set $q_0 :=  \varepsilon \, y^2 - K e^{-\tau} y^4$ and we look for a sub-solution in the form 
\[
  q^-(y, \tau) = q_0(y, \tau) +   q_1(y, \tau),
\]
in which $q_1$ is suitable correction term.  To this end we will  consider separately   the intermediate and tip regions.  

\smallskip

We summarize the results in this section as follows.  

\begin{proposition}\label{prop-sub-sup}  Let $\varepsilon > 0$ be a very   small constant and assume that  $\tau_0$ is  any  large
number so that  $\varepsilon \, e^{\tau_0} \ll 1$.  Denote by $\tau_2 > \tau_0$ the time when 
$ \varepsilon \, e^{\tau_2/2 } =\sigma_n$, for any fixed large constant  $\sigma_n$.  Then, for any fixed $K > 0$ the following hold for $\tau\in [\tau_0,\tau_2]$,
\begin{itemize}
\item the function $\cQ^+_{\varepsilon, K} :=  \varepsilon \, H_2(y)- K e^{-\tau} H_4(y)$ is a supersolution of equation \eqref{eq:q} on $|y| \geq \ell_1$, that is on
$\cI_{\ell_1, \ell_2} \cup \cT_{\ell_2}$.  

\item the function $\cQ^-_{\varepsilon, K}$ defined in \eqref{eqn-barq} for $y <0$,  and extended by reflection $\cQ^-_{\varepsilon, K}(y,\tau)
:=\cQ^-_{\varepsilon, K}(-y,\tau)$ for $y >0$,  is a  subsolution of equation \eqref{eq:q} on $|y| \geq \ell_1$, that is on
$\cI_{\ell_1, \ell_2} \cup \cT_{\ell_2}$.  
\end{itemize} 
\noindent Both functions $\cQ^+_{\varepsilon, K}$ and $Q^-_{\varepsilon, K}$ satisfy the asymptotics 
\begin{equation}		\label{eqn-bound-QQQQ}
\cQ^\pm_{\varepsilon, K}(y,\tau) = ( \varepsilon y^2 - K e^{-\tau} y^4 ) (1+ o_\tau(1)), \qquad \mbox{on} \,\, \cI_{\ell_1, \ell_2}.  
\end{equation}
Equivalently $U^\pm_{\varepsilon, K} := \sqrt{ 2(n-1) + \cQ^\pm_{\varepsilon, K} }$ define super and subsolutions of equation \eqref{eq:u} in the
same region $\cI_{\ell_1, \ell_2} \cup \cT_{\ell_2}$.  Finally, the barriers are still valid for $\varepsilon=0$ in which case $\tau_2 > \tau_0$ 
can be any number.

\end{proposition}

\subsection{Subsolution   in the intermediate region $\cI_{\ell_1, \ell_2}$}\label{sec-sub1} 

We begin my observing that $q_0$ solves 
\begin{equation}		\label{eqn-q020}
q_{0\tau} +   \frac{y}{2} q_{0y} -q_0=0.  
\end{equation}
Therefore, the  function $q^-:= q_0+ q_1$ is a subsolution of \eqref{eq:q} if  $q_1$ satisfies 

\begin{equation}		\label{eqn-q12}
  q_{1\tau}   +  \frac{y}{2} q_{1y} -q_1 < -     \frac{(q^-_{yy} + 2)(q^-_y)^2}{8(n-1) + 4q^- + (q_y^-)^2}  +   q^-_{yy} := \cE(q^-).  
\end{equation}
As a temporary step we estimate the right hand side of \eqref{eqn-q12}  by substituting $\cE(q^-)$ by  $\cE(q_0):=  - \frac{(q_{0yy} + 2) \, q_{0y}^2}{8(n-1) + 4q_0 + (q_{0y})^2}    + q_{0yy}$.  

Formulas   \eqref{eqn-atau} and \eqref{eqn-Ytau}  imply that for $\sigma \leq \sigma_n$ we have 
$Y_0(\tau) \sim   \sqrt{\brho} \, e^{\tau/4}= \cO(e^{\tau/4})$, since $ \brho:=\sqrt{\sigma^2+ 8(n-1)K}.  $
Furthermore, 
\begin{equation}
 \label{eqn-der1}
q_{0y} \approx 2\varepsilon y - K e^{-\tau} 4y^3, \quad  |q_{0y}|\lesssim \varepsilon y + e^{-\tau}y^3, \quad q_{0y}^2 \lesssim \varepsilon^2 y^2 + e^{-2\tau}y^6 
\end{equation}
and
\begin{equation}		\label{eqn-der2}
q_{0yy}\approx 2\varepsilon - 12 K e^{-\tau}y^2, \qquad |q_{0yy}|\lesssim \varepsilon + e^{-\tau}y^2
\end{equation}
imply 
\[
 q_{0y}^2 = \cO(e^{-\tau/2}), \qquad q_{0yy} =\cO(e^{-\tau/2}).  
\]
It follows that 
\[
\cE(q_0) \geq -   \frac{3\,q_{0y}^2}{8(n-1) + 4q_0 } -  |q_{0yy}| \geq - B \, \Big (
 \frac{\varepsilon^2 y^2 + K^2 e^{-2\tau} y^6}{2(n-1) +  q_0} + \varepsilon +  e^{-\tau} y^2 \Big )
\]
 for some constant $B=B(\sigma_n, n)$.  

The above discussion leads to defining the correction of our subsolution  $q^-:= q_0+ q_1$ as $q_1 := \vartheta Q$ for  an appropriate  solution $Q$ of equation 
\begin{equation}		\label{eqn-Q}
Q_\tau + \frac y2 Q_y - Q =  -  \frac{\varepsilon^2 y^2 + K^2 e^{-2\tau} y^6}{2(n-1) +  q_0} - \varepsilon -  e^{-\tau} y^2.  
\end{equation}

To solve  \eqref{eqn-Q} we integrate along characteristics and find,  by  direct calculation,   that the general  solution of  \eqref{eqn-Q} is given by 
\begin{equation}		\label{eqn-QQ}
Q(y,\tau) =   \frac{\varepsilon^2 y^2 + K^2 e^{-2\tau} y^6}{2(n-1)} \,  \ln \big ( \frac{ 2(n-1) +q_0 }{e^{-\tau} y^4} \big )  + (\varepsilon + e^{-\tau} y^2) + y^2 h(\tau-2\ln |y|) 
\end{equation}
for any smooth parameter function $h$ (that is zero along characteristics).  We will momentarily choose  the function $h$ in such a way that $Q(y,\tau ) >0$ on $\cI_{\ell_1, \ell_2}$.  
With such a  choice of $Q$, the function $q_1:=\vartheta \, Q$ satisfies 
\begin{equation}		\label{eqn-q15}  q_{1\tau}   +  \frac{y}{2} q_{1y} -q_1 = -  \vartheta \,  \Big (  \frac{\varepsilon^2 y^2 + K^2 e^{-2\tau} y^6}{2(n-1) +  q_0} + \varepsilon + K e^{-\tau} y^2 \Big )
\end{equation}
and therefore  \eqref{eqn-q12} holds in the region $\cI_{\ell_1, \ell_2}$  provided 
\begin{equation}		\label{eqn-q13}
  \frac{(2+q^-_{yy})  (q^-_{y})^2}{8(n-1) + 4q^- + (q^-_y)^2}  -  q^-_{yy} < \vartheta \, \Big (  \frac{\varepsilon^2 y^2 + K^2 e^{-2\tau} y^6}{2(n-1) +  q_0} + \varepsilon + K e^{-\tau} y^2  \Big )
\end{equation}
where $\vartheta$ is a  sufficiently large  constant (depending only on  $\sigma_n$ and $n$).  

\sk

\subsubsection{The choice of  the parameter function $h$ in \eqref{eqn-QQ}.  }
Before we proceed  we will need to make an appropriate choice of solution $Q(y,\tau)$ of \eqref{eqn-Q} by  specifying the parameter function $h$
in \eqref{eqn-QQ}.  The subtlety comes from the fact that  $2(n-1)+q_0$ is very tiny ($= \ell_2 \, e^{-\tau/2}$)   at the intersection of the  intermediate  
and tip regions.  Hence to control the denominator $8(n-1) + 4 q^-$ in \eqref{eqn-q13} in terms of $2(n-1)+q_0$ 
one needs to guarantee  $8(n-1) + 4q_0 + 4 \vartheta Q > 2(n-1) + q_0$.  We will simply make $Q>0$.  

To this end, first  to simplify the notation, we set 
$
\cA(y,\tau) :=  \frac{\varepsilon^2 y^2 + K^2 e^{-2\tau} y^6}{2(n-1)}
$
and observe that since $|y| \leq Y_0 = \cO(e^{\tau/4})$, we have 
\begin{equation}		\label{eqn-AAA0} 
\cA(y,\tau) :=  \frac{\varepsilon^2 y^2 + K^2 e^{-2\tau} y^6}{2(n-1)} = \cO(e^{-\tau/2})
%, \qquad \cA(Y_0, \tau)= a_0 \, e^{-\tau/2}
\end{equation}
where $a_0=a_0(\sigma, n)=o(1)$.  Since $\cA(y,\tau) $ is of the form $y^2 \tilde h(y^2 e^{-\tau})=y^2 h (\tau-2 \ln y)$,  we will choose $h$ so that $y^2 h(\tau - 2\ln y)=
\cA \cdot f(\tau  - 2\ln y)$ for some appropriate  $f$.  With such a choice and after rearranging terms we express 
$Q=Q(y, \tau)$ as 
\[
Q := \cA \,  \big \{   \ln  \big ( e^{\frac \tau 2}  (2(n-1) +q_0  ) \big ) + \frac \tau 2 -  4 \ln |y|  + f(\tau - 2\ln |y|) \big  \} + (\varepsilon + K e^{-\tau} y^2).  
\]
We will choose $f$ so that $ \frac \tau 2 -  4 \ln |y|  + f(\tau - 2\ln |y|)  \sim \tau$ at $|y|=Y_0$.  For example, setting   $f := 2 (\tau - 2 \ln |y|)$ we guarantee that at $y \approx Y_0  = c_\sigma \,  e^{\tau/4} $ we have 
\begin{equation}		\label{eqn-tau6}
\frac \tau 2 -  4 \ln |y|  + f = \frac \tau 2 -  4 \ln |y| +  2 \tau - 4\ln|y|  =  \frac {5\tau}2 - 8 \ln |y| \approx \frac \tau2.  
\end{equation}

\sk
From now one \emph{we fix this choice of   $Q$}, namely  
\begin{equation}		\label{eqn-QQn}
Q(y,\tau) :=  \cA \cdot  \big \{   \ln  \big ( e^{\tau/2}  (2(n-1) +q_0  ) \big ) + \frac {5\tau}2 - 8 \ln |y|   \big  \} + (\varepsilon + K e^{-\tau} y^2)
\end{equation}
where
$\cA(y,\tau)$ is given by \eqref{eqn-AAA0}.  We claim the following:

\begin{claim}\label{claim-QQQ}  \emph{If $Q$ is given  by \eqref{eqn-QQn} we have }
\begin{equation}		\label{eqn-Q-linfty}
0 < Q(y,\tau) \leq  \cO_{\ell_1, \ell_2} \big  ( \tau e^{-\tau/2} \big ), \qquad \mbox{on}\,\,\,  \cI_{\ell_1, \ell_2}
\end{equation}
 provided $\ell_1, \ell_2 \gg 1$ and $\tau  \gg 1$.  

\end{claim} 

\begin{proof}[Proof of Claim \ref{claim-QQQ}] 
The claim follows  from  the condition $2(n-1) + q_0 \geq \ell_2 e^{-\tau/2}$, \eqref{eqn-AAA0} and  \eqref{eqn-tau6} that  hold in $\cI_{\ell_1, \ell_2}$.  
First, inserting these  bounds in \eqref{eqn-QQn} we get $Q(y, \tau) > \cA(y, \tau)  \, \big (  \ln  \ell_2 + \frac \tau 4  \big ) >0.  $
Furthermore, the same bounds imply $Q(y, \tau) \leq  \cA(y, \tau)  \, \big (  \ln  \ell_2 + \cO(\tau)  \big )$.  
\end{proof} 

\sk

\begin{lemma}\label{lem-sub-inter}  There exists  $\vartheta >0$ depending on $\sigma_n, n$, such that the function $q_1:=\vartheta Q$ satisfies \eqref{eqn-q12} in the intermediate region $ \cI_{\ell_1, \ell_2} $ 
provided $\ell_1, \ell_2$ and $\tau_0$ are sufficiently large.  Subsequently, $q^-=q_0+ q_1$ 
 is a subsolution of \eqref{eq:q} in the same region.  

 \end{lemma}

\begin{proof} Throughout the proof we will  repeatedly use that  $\varepsilon = \cO(e^{-\frac \tau2})$ and 
\begin{equation}		\label{eqn-est}
|y| \lesssim  e^{- \frac \tau4}, \,\, |g_{0y}| \lesssim e^{- \frac \tau4}, \,\,  |g_{0yy}| \lesssim e^{- \frac \tau2}, \,\, \ell_2 e^{-\frac \tau2} \leq 2(n-1) + q_0 \leq c_n.  
\end{equation}
We will also assume that $y >0$ since the $y <0$ is identical.

We have defined $q_1$ to satisfy \eqref{eqn-q15} and we have seen that $q^-=q_0+q_1$ is a subsolution on \eqref{eq:q} provided
\eqref{eqn-q13} holds.  We will now show that with  our choice of $q_1=\vartheta Q$, where  $Q$ is  given  by \eqref{eqn-QQ}  we indeed have that  \eqref{eqn-q13} holds.  To do so, we will need some estimates on $Q_{y}$ and $Q_{yy}$.  
Recall our notation  $\cA :=  \frac{\varepsilon^2 y^2 + K^2 e^{-2\tau} y^6}{2(n-1)}$ and observe that 
\begin{equation}		\label{eqn-cA} 
\cA= \cO(e^{-\tau/2}), \quad  \cA_y=\cO(\cA \, y^{-1} ) =\cO(e^{-3\tau/4}), \quad \cA_{yy}=\cO(\cA \, y^{-2} ) =\cO(e^{-\tau}).
\end{equation}
Differentiating $Q$ in $y$  gives  
\[
Q_{y} = \cA \,   \,  \big (   \frac{q_{0y}}{2(n-1) +q_0}   - \frac 8{y}    \big  )  + \cA_y \, \big \{    \ln \big ( e^{\tau/2}  (2(n-1) +q_0  ) \big ) + \frac{5\tau}2   - 8 \ln y   \big  \} + 2 e^{-\tau} y
\]
and by applying the estimates    \eqref{eqn-est} and \eqref{eqn-cA} we   get 
\begin{equation}		\label{eqn-q1y}
Q_y =o(e^{-\tau/4}). 
\end{equation}

Now let us look at $Q_{yy}$.  We claim the following.  

\begin{claim}\label{claim-Q12} \emph{On the whole intermediate region we have} 
\begin{equation}		\label{eqn-Q13} Q_{yy} = o(1), \qquad - \, Q_{yy} \leq  o \Big ( \frac{\cA}{2(n-1) + q_0}  \Big )
\end{equation}
{\em provided $\ell_1, \ell_2 \gg 1$.  } 

\end{claim}

\begin{proof}[Proof of Claim \ref{claim-Q12}]
Call $\cB:=   \ln \big ( e^{\tau/2}  (2(n-1) +q_0  ) \big ) + \frac{5\tau}2   - 8 \ln y  $. Then, formula \eqref{eqn-QQn} 
and  
\[
Q_{yy} = {\cA}   \cdot  \cB_{yy}  + 2 {\cA}_y \cdot  \cB_y  + {\cA}_{yy} \cdot \cB + 2 K e^{-\tau}.  
\]
Furthermore direct calculation shows
\[
 \cB_y = \big (   \frac{q_{0y}}{2(n-1) +q_0}   - \frac 8{y}    \big  ), \qquad \cB_{yy} = \frac {  q_{0yy}}{ 2(n-1) + q_0} -  \frac { q_{0y}^2}{(2(n-1) + q_0 )^2} + \frac 8{y^2}.
\]
One can easily check using the last two formulas,  \eqref{eqn-est} and \eqref{eqn-cA}, that 
$Q_{yy} = o(1)$ that is the first bound in \eqref{eqn-Q13} holds. 

Let us know concentrate in proving the  second bound in \eqref{eqn-Q13}. 
Dropping the last term in $Q_{yy}$ that has the correct sign and using  \eqref{eqn-cA} we get 
\begin{equation}		\label{eqn-Q14}
- \, Q_{yy} \leq -  {\cA} \,  \big (  \cB_{yy}  + 2 y^{-1}  \cB_y  + y^{-2} \cB \big )
\end{equation}
hence it is sufficient to show that
\[
 - \cB_{yy}  -  2 y^{-1}  \cB_y  -   y^{-2} \cB  \leq   o(\frac 1{2(n-1) + q_0}).
\]
To this end, we observe using \eqref{eqn-est} that  
\[
- \cB_{yy} = \frac{1}{2(n-1) + q_0} \, \big ( o(1) + o(1) + o(1) \big ) = o \big ( \frac{1}{2(n-1) + q_0}\big ).
\]
and
\[
 - y^{-1} \, \cB_y = o \big ( \frac{1}{2(n-1) + q_0}\big )
\]
and (since $e^{\tau/2}  (2(n-1) +q_0  ) \geq \ell_2 \gg 1$) the first and second terms in $\cB$
are positive, hence 
\[
- y^{-2} \, \cB \leq  8 \, y^{-2}  \, \ln y = o ( \frac{1}{2(n-1) + q_0}\big ).
\]

Combining the last three estimates with \eqref{eqn-Q14} yields that the second bound in \eqref{eqn-Q13} holds,
thus finishing the proof of Claim \ref{claim-Q12}. 
\end{proof}

\smallskip

We are now in position to finish the proof of the Lemma by showing  that \eqref{eqn-q13} holds. 
We will  combining   \eqref{eqn-est},  \eqref{eqn-q1y}, \eqref{eqn-Q13} and $8(n-1) +  4 q^-  > 4(n-1)+ 2q_0$.  These estimates imply  $(2+q^-_{yy})  (q^-_{y})^2 < 3 q_{0y}^2$ and 
$8(n-1) + 4q^- + (q^-_y)^2 > 4(n-1) + 2 q_0$, thus 
\begin{equation*}
  \frac{(2+q^-_{yy})  (q^-_{y})^2}{8(n-1) + 4q^- + (q^-_y)^2} < \frac{3 q_{0y}^2 }{ 4(n-1) + 2 q_0} <  B_1 \cdot 
 \frac{\varepsilon^2 y^2 + K^2 e^{-2\tau} y^6}{2(n-1) +  q_0}
 \end{equation*}
 for some absolute constant $B_1$.  Furthermore, \eqref{eqn-Q13} and $q_{0yy} = 2\varepsilon - 12 e^{-\tau} y^2$ give 
\[
- q^-_{yy} = - q_{0yy} -  \vartheta Q_{yy} \leq  - 2\varepsilon  + 12 K e^{-\tau} y^2 +  o \Big ( \frac{\varepsilon^2 y^2 + K^2 e^{-2\tau} y^6}{2(n-1) + q_0}  \Big ) 
\]
The last two estimates yield that \eqref{eqn-q13} holds provided $\vartheta > \max (B+1, 12) $.  
Consequently, $q_1:=\vartheta Q$ satisfies \eqref{eqn-q12} on 
$\cI_{\ell_1, \ell_2}$ and $q^-=q_0+ q_1$ 
 is a subsolution of \eqref{eq:q}.  
\end{proof}

\subsection{Subsolution in   the tip region}\label{sec-sub2} 
We will now construct a subsolution  of equation \eqref{eq:u} near the tip  $-Y_0(\tau)$ where $2(n-1)+q_0=0$ and match it with our subsolution in the intermediate  region.  
For a solution $u$ of \eqref{eq:u}, at the tip we  perform the change of coordinates 
\begin{equation}		\label{eqn-wxi} w(\xi, \tau) := \alpha(\tau) \, u(y,\tau), \quad y = -Y_0 + \alpha^{-1} \xi
\end{equation}
where $\alpha(\tau)$, $Y_0(\tau)$  are given by \eqref{eqn-atau}, \eqref{eqn-Ytau} respectively (see in subsection \ref{ss-stip}   for details).   

Thus  \eqref{eqn-w202}   implies that $w(\xi,\tau)$ satisfies 
\begin{equation}		\label{eqn-w55} 
\frac 1{\alpha^2}\, w_\tau + \frac 1{\alpha^2} \Big ( \frac{\alpha'}{\alpha}  + \frac 12 \Big )  \big (  \xi  w_\xi -   \, w \big ) = \frac{w_{\xi\xi}}{1 + w_\xi^2} -  \frac{n-1}w  - w_\xi.  
\end{equation}

We are interested  to construct a \emph{subsolution} $w^-(\xi,\tau)$ of \eqref{eqn-w55}  in the region $|\xi| \leq L$ for some large constant $L$.  Because equation \eqref{eqn-w55}
becomes degenerate at the tip, it is more convenient  to  switch coordinates, that is write $\xi = \xi (w,\tau)$.  Then, \eqref{eqn-w55} becomes
\begin{equation}		\label{eqn-xi0}
\frac 1{\alpha^2} \xi_\tau + \frac 1{\alpha^2} \Big ( \frac{\alpha'}{\alpha}  + \frac 12 \Big )  \big (  w \xi_w -   \, \xi \big )  = \frac{\xi_{ww}}{1 + \xi_w^2} + \frac{n-1}w  \, \xi_w  - 1.  
\end{equation}
%Call $\cE(\xi)$ the left hand side of \eqref{eqn-xi0} which will be treated as an error term.  

A subsolution of \eqref{eqn-w55} means a supersolution of \eqref{eqn-xi0}.  Following  \cite{AV} 
we seek for a supersolution of \eqref{eqn-xi0} of the form 
$\xi^+(w,\tau) = \cB(w)  + \tau \, \alpha(\tau)^{-2} \xi_1(w)$, (the  choice of scaling for $\xi_1$ will come apparent in the sequel)  
where  $\cB(w)$ denotes the profile of the Bowl soliton of speed one in the new coordinates that satisfies 
\begin{equation}		\label{eqn-Phi}
\frac{\cB''}{1 + (\cB')^2} + \frac{n-1}w  \, \cB'  - 1=0, \quad \cB(0)=\cB'(0)=0
\end{equation}
and $\xi_1(w)$ is an error term that will be taken   to solve 
\[
 \Big ( \frac {\xi_1'}{1+ (\cB')^2 } \Big )' + \frac{n-1}w  \, \xi_1'  = - C, \qquad \xi_1(0) =0=\xi_1'(0)=0
\]
for a constant $C >0$.  We will show the following: 

\begin{lemma} By choosing $C>0$ sufficiently large, 
$\xi^+(w,\tau) = \cB(w)  + \tau  \, \alpha^{-2} \,  \xi_1(w)$
becomes a supersolution of equation \eqref{eqn-xi1} on $0 \leq w \leq R$,
for any $R \gg 1$ and satisfies  $\xi^+(0,\tau)=0$.  

\end{lemma} 

\begin{proof} We plug $\xi^+(w,\tau) = \cB(w)  + \tau \, \alpha(\tau)^{-2} \xi_1(w)$ in \eqref{eqn-xi1}.  
By writing $\xi_{ww}/(1 + \xi_w^2) = (\arctan \xi_w)_w$ and expanding in powers of $\tau \alpha^{-2} \xi_1$
we find that the right hand side in \eqref{eqn-xi0} is 
\begin{equation*}
\begin{split}
\tau \alpha^{-2} \Big  \{  \Big ( \frac {\xi_1'}{1+ (\cB')^2 } \Big )' + \frac{n-1}w  \, \xi_1'  \Big   \} & +
\cO(\tau^2 \alpha^{-4} \big (  (\xi_1'')^2 +  (\xi_1')^2  \big ) \\&= - C \, \tau \alpha^{-2} + \cO(\tau^2 \alpha^{-4} \big (  (\xi_1'')^2 +  (\xi_1')^2 \big ).  
\end{split}
\end{equation*}
On the other hand, using $0 < \frac{\alpha'}{\alpha} + \frac 12 \leq \cO(n)$ (that follows by \eqref{eqn-atau}) and a direct calculation yields that the left hand side in \eqref{eqn-xi0} is
\[
 \alpha^{-2} \Big  \{  \cO(\tau \alpha^{-2} ( |\xi_1| + w |\xi'_1| )   +  \cO_n(1) \,  \big (  w \cB' -   \, \cB \big )  \Big \}.  
\]
We conclude that  by choosing $C >0$ sufficiently large (depending also on $ \cO_n(1)$), $\xi^+(w,\tau)$ satisfies 
\begin{equation}		\label{eqn-xi1}
\frac 1{\alpha^2} \xi^+_\tau + \frac 1{\alpha^2} \Big ( \frac{\alpha'}{\alpha}  + \frac 12 \Big )  \big (  w \xi^+_w -   \, \xi^+ \big )  
>   \frac{\xi_{ww}}{1 + \xi_w^2} + \frac{n-1}w  \, \xi_w  - 1 + \frac{ \tau }{\alpha^{2}} 
\end{equation}
that is $\xi^+$ is  a supersolution of equation  \eqref{eqn-xi0} in the region $0 < w < R$.  We also note that $\xi^+(0, \tau) = 0$,
since $\cB(0)=\xi_1(0)=0$.  

\end{proof} 

Define next  $w^-(\xi,\tau)$ to be the \emph{inverse}  of the function  $\xi=\xi^+(w,\tau)$ and compute (by direct calculation using \eqref{eqn-xi1})
\begin{equation}		\label{eqn-w52} 
\frac 1{\alpha^2}  w_\tau^- + \frac 1{\alpha^2} \Big ( \frac{\alpha'}{\alpha}  + \frac 12 \Big )  \big (  \xi  w_\xi ^- -  w^- \big ) <  \frac{w_{\xi\xi}^-}{1 + (w_\xi^-)^2} -  \frac{n-1}{w^-}  
- w_\xi^-  - \frac{\tau}{\alpha^2}  w_\xi^-.  
\end{equation}

\begin{lemma}\label{lem-sub-tip}
Let $w^-(\xi,\tau)$  be the inverse of the function  $\xi=\xi^+(w,\tau)$.  Then, for any large $\xi^* >1$,  $w^-$ is a sub-solution of equation \eqref{eqn-w55}
on $0 < \xi < \xi^*$.  In addition,  satisfies $w^-(0, \tau)=0$ and for all $0 < \xi < \xi^*$ we have 
\begin{equation}		\label{eqn-wmas}
w^-(\xi,\tau) = W(\xi) + \cO_{\xi^*}(\tau \alpha^{-2})
\end{equation}
where the dependence in $\xi^*$ is at most exponential.  It follows that for $\xi \gg 1$, $w^-(\xi, \tau)$ satisfies the asymptotics 
\[
w^-(\xi, \tau)^2  = 2(n-1) \Big \{   \xi + \frac 12   \ln  \big [ 2(n-1) \xi + \cO(\ln \xi)  \big ]    +\cO_{\xi^*} ( \tau \alpha^{-2} ) \Big \} 
\]
and
\[
\partial_\xi  ( w^-  (\xi, \tau)^2)     = 2 (n-1)  + (n-1) \xi^{-1}  + o(\xi^{-1})   + \cO_{\xi^*} ( \tau \alpha^{-2} ).  
\]

\end{lemma} 

\begin{proof} It is clear that $w^-(\xi,\tau)$ is a subsolution of \eqref{eqn-w55}.  The asymptotics 
simply follow from the definition of  $\xi=\xi^+(w,\tau)$  and  the asymptotic behavior
of the Bowl soliton $\cB(w)$, as $w \to +\infty$, namely \eqref{eqn-Ba}.  
Indeed, combining these two we get 
\[
\xi^+(w,\tau) =  \frac {w^2}{2(n-1)} - \frac 12 \ln w^2  + \cO(w^{-2}) + \cO_R(\tau \alpha^{-2})
\]
and therefore, setting $\xi=\xi^+(w,\tau) $ and inverting with respect to $w$ we get $w^2 \approx  2(n-1)\, \xi + O (\ln \xi)$  and 
\begin{equation*}\begin{split}
\frac{w^-(\xi,\tau)^2}{2(n-1)}   &=  \xi  + \frac 12 \ln w^2 +  \cO(w^{-2}) +  \cO_{\xi^*}(\tau \alpha^{-2}) \\
&=  \xi +   \frac 12 \ln [ 2(n-1) \, (\xi   + \cO (\ln \xi)] + 
 \cO_{\xi^*}(\tau \alpha^{-2}).  
 \end{split}
 \end{equation*}
Then, the lemma follows.

\end{proof}

\subsection{Matching the sub-solutions of  the intermediate  and tip regions}\label{sec-sub3}

Recall the definitions of $\alpha(\tau)$, $Y_0$ and $\cA(y,\tau)$  in \eqref{eqn-atau}, \eqref{eqn-Ytau} and \eqref{eqn-AAA0} from
which we easily see that 
\begin{equation}		\label{eqn-use}
\alpha(\tau) \sim Y_0(\tau) \sim e^{\tau/4}, \qquad  a_0:= \cA(Y_0,\tau)  \, \alpha(\tau)^{-2}  \sim 1.  
\end{equation}
We will use these in our computations below.

Let $w^-$ be the subsolution we constructed in Lemma \ref{lem-sub-tip} of the previous  section.  Recalling the change of variables \eqref{eqn-wxi}  (near the left tip $y=-Y_0$) and $q= u^2 - 2(n-1)$,
we define    
\begin{equation}		\label{eqn-qqq0} \tilde q^-(y,\tau) := \alpha^{-2} w^-(\xi + \beta(\tau), \tau)^2 - 2(n-1), \qquad    y  = - Y_0 + \alpha^{-1} \xi,
\end{equation}
or an appropriate correction term $\beta(\tau)$ to be defined momentarily.  This barrier is defined for $- \beta(\tau) < \xi \leq \xi^*$,   for some $\xi^*$ large.  

Our goal in this section is to match $\tilde q^-(y,\tau) $ with  the intermediate region barrier $q^-(y,\tau)$
when $y <0$.  \emph{The matching for $y >0$ is similar due to symmetry.  } 
We will see that the matching happens at  $y^*(\tau)  = -Y_0 + \alpha^{-1} \xi^*$,  where  \emph{$\xi^*$ is a fixed  large number.  }   As a result  the barrier
\begin{equation}		\label{eqn-barq}
\cQ^-_{\varepsilon,K} (y,\tau) := \begin{cases} q^-(y,\tau), \quad \mbox{on} \,\,   - Y_0 + \alpha^{-1} \xi^* \leq y \leq -\ell_1\\
\tilde q^-(y,\tau), \quad \mbox{on} \,\,   y  \ge - Y_0 + \alpha^{-1} \xi, \, - \beta(\tau)  <  \xi < \xi^*
\end{cases}
\end{equation}
will be a subsolution of equation \eqref{eq:q} on $\cI_{\ell_1, \ell_2} \cup \cT_{\ell_2}$ for $y <0$,  where $\ell_2 \approx 2(n-1) \xi^*$.

To this end,   \emph{we define}  $\beta(\tau)$  so that 
\begin{equation}		\label{eqn-qqq}
q^-( y^*, \tau) = \tilde q^-(y^*, \tau), \qquad \mbox{at some}\,\,\,  y^*(\tau)  = - Y_0 + \alpha^{-1} \xi^*
\end{equation}
where $\xi^*$ is a fixed sufficiently large number (depending also on our choice of $\ell_2$).  
Furthermore, we will see that with such a choice of $\beta(\tau)$,  $\tilde q^-(y,\tau)$ is a sub-solution 
of equation \eqref{eq:q} in the tip region.

\begin{lemma}[Matching of $q^-$ and $\tilde q^-$]\label{lem-match}  Define  $\beta(\tau)$ such that \eqref{eqn-qqq} holds.  
Then,  
\begin{equation}		\label{eqn-btau} \beta (\tau) =  \frac{\vartheta a_0}{4(n-1)}  \, ( \tau  + \cO(\ln \tau) +  \cO(1) )
\end{equation}
where $a_0(\tau) := \alpha^{-2} \cA(Y_0,\tau) \sim 1$.  
Furthermore, at  $y^*$ we have
\begin{equation}		\label{eqn-mder}
q^-_y (y^*,\tau)  > {\tilde q}^-_y(y^*,\tau).  
\end{equation}
Consequently  $\hat  q^-(y,\tau)$   defined by \eqref{eqn-barq} is a sub-solution of \eqref{eq:q}
on $\cI_{\ell_1, \ell_2} \cup \cI_{\ell_2} $.  
\end{lemma}

\begin{proof} 
Recall that $\alpha \sim e^{\tau/4}$.  First, we  observe  that near $-Y_0$ we have $q_{0y} >0$ and 
\begin{equation}		\label{eqn-asy1}
\alpha \, q_{0y}  =2(n-1) (1+ \cO(\alpha^{-1}) ), \quad \mbox{for}\,\,\, y  = -  Y_0 \,  (1+ \cO(\alpha^{-1}) ).  
\end{equation}
In fact,  $q_{0y} = 2\varepsilon y - 4Ke^{-\tau} y^3$, \eqref{eqn-Ytau}, and $\varepsilon = \sigma e^{-\tau/2}$ imply that for for   $y =- Y_0 \, (1+ \cO(e^{-\tau/4}))$ we have
\[
|q_{0y} | = 2 e^{-\frac \tau 2} |y| \, | \sigma - 2 K e^{-\frac \tau2} y^2 | 
= \sqrt{ \frac 2K}  \, \bar \sigma \, \sqrt{\sigma + \bar \sigma} \, e^{-\tau/4} \, \big (1+ \cO(e^{\frac \tau 4}) \big )
\]
which combined with \eqref{eqn-atau} yields
\begin{equation*}
\alpha \, |q_{0y}|  \approx  \frac{\sqrt{2K} (n-1) e^{\frac \tau 4} }{\brho \sqrt{\sigma + \brho}} \cdot \sqrt{\frac 2K} \, \brho \sqrt{\sigma + \brho} \, e^{-\frac  \tau4}
=2(n-1)  \, \big (1+ \cO(\alpha^{-1}) \big )
\end{equation*}
that is \eqref{eqn-asy1} holds.  

\sk

\emph{Matching of $q^-$ with  $\tilde q^-$ at $y^*$ and determining  $\beta(\tau)$:}  From the definition of $ \tilde q^-$ in \eqref{eqn-qqq0}  we have  
\[
q^-( y , \tau) = \tilde q^-(y, \tau) \quad \mbox{iff} \quad 
2(n-1) + q^-( y , \tau) = \alpha^{-2} w^-(\xi+\beta(\tau), \tau)^2
\]
 where $q^- = q_0 + \vartheta Q$, with $Q$ is given by \eqref{eqn-QQn}.  
%where $y^2 h (y^2 e^{-\tau})  = - A_n y^6 e^{-2\tau} \ln (y^2 e^{-\tau})$.  
Now,  \eqref{eqn-asy1}, $q_0(-Y_0, \tau)=-2(n-1)$, the definition of  the matching point $y^*$,  and the mean value theorem,   imply 
\begin{equation}		\label{eqn-q23}
2(n-1) + q_0(y^*,\tau) \approx  q_{0y}(y^*, \tau) \, (y^*+Y_0) \approx   2 (n-1) \alpha^{-2} \,  \xi^* \, (1 + \cO(1) ).  
\end{equation}
Inserting  the above in  \eqref{eqn-QQn} we get 
\[
Q (y^*, \tau) =   a_0 \, \alpha^{-2}  \Big ( \ln  \big ( e^{\tau/2} \cdot  2(n-1) \,  \alpha^{-2} \xi^* \big ) + \frac {5\tau}2 - 8 \, \ln (Y_0)  \Big ) + \varepsilon + K\, e^{-\tau} y^2
\]
where we used $a_0 := \cA(y^*,\tau)\, \alpha^2 \sim 1$ (see \eqref{eqn-use}).  
By  $a_0 \sim 1$,  $  \alpha^{-2} \sim e^{-\tau/2}$, $\ln Y_0 = \tau/4 + C_\sigma$,  and $\varepsilon + K\, e^{-\tau} y^2 = \cO(\alpha^{-2})$, we obtain 
\[
Q(y^*, \tau) = a_0\,  \alpha^{-2}  \big ( \ln \xi^*  + \frac \tau2 + \cO_\tau(1)  \big ).  
\]
Since $q^- = q_0 +  \vartheta \, Q$,  the above and \eqref{eqn-q23} yield  that at $y^*$ we have 
\[
2(n-1) + q^- = 2(n-1) + q_0 + \vartheta Q = \alpha^{-2} \Big ( 2(n-1) \xi^* + \vartheta a_0 \, (\ln \xi^* +\frac  \tau 2) +   \cO_\tau(1) \Big ).  
\]
Thus, to have  $2(n-1)+  q^- = \alpha^{-2} w^-(\xi^*+\beta(\tau),\tau)^2$  at  $y^*$ we must choose $\beta(\tau)$ such that
\[
 w^-(\xi^*+\beta(\tau),\tau)^2  =  2(n-1) \xi^* +  \frac{\vartheta a_0}2 \, ( \tau + 2 \ln \xi^* )  +  \cO_\tau(1).  
\]
On the other hand,  the asymptotics for  $ w^-(\xi,\tau)^2$ from Lemma \ref{lem-sub-tip}, give that at first order
$ w^-(\xi^*+\beta(\tau),\tau)^2 =  2(n-1) ( \xi^* + \beta(\tau)) + (n-1)    \ln  (\xi^* + \beta(\tau))   + \cO_\tau(1)$.  Hence, we must choose
$\beta(\tau)$ so that 
\[
  2(n-1) ( \xi^* + \beta(\tau)) + (n-1)   \ln  (\xi^* + \beta(\tau))     =  2(n-1) \xi^* + \frac{\vartheta a_0}2 \, ( \tau +2 \ln \xi^* ) +  \cO_\tau(1) 
\]
leading to \eqref{eqn-btau}.  

\sk

\emph{ Matching  the derivatives $q^-_y$ and  $\tilde q^-_y$ at $y^*$:}
Having determined $\beta(\tau)$,  will finally match  the derivatives of ${\tilde q}^-$ and $q^-$ at $y^* = - Y_0 + \alpha^{-1} \xi^*$ to show that
\begin{equation}		\label{eqn-qym} q^-_y(y^*, \tau ) > {\tilde q}^-_y(y^*, \tau ).  
\end{equation}

Assume that $y$ is near $y^* = - Y_0 + \alpha^{-1} \xi^*$, where  $q_{0y} >0$.  Then,  
\[
q_y^-= q_{0y} +  \vartheta Q_y \approx 2 (n-1)\,  \alpha^{-1}  (1 + \cO(\alpha^{-2}) ) +  \vartheta Q_y.  
\]
On the other hand by Lemma \ref{lem-sub-tip}, we have
\[
{\tilde q}^-_y= \alpha^{-1} \, \partial_\xi (w^-)^2  (\xi + \beta(\tau))   = 2 (n-1) \,  \alpha^{-1}  + (n-1) \alpha^{-1}  (\xi+\beta(\tau))^{-1} (1 + o(1)).  
\]
Thus for \eqref{eqn-qym} to hold, we need to have 
\begin{equation}		\label{eqn-qneed}
 \vartheta Q_y (y^*, \tau)  > (n-1) \alpha^{-1} (\xi^* + \beta(\tau))^{-1}.  
\end{equation}
But by \eqref{eqn-QQn} we have 
\[
Q_y (y^*, \tau) = \cA \big ( \frac{q_{0y}}{2(n-1) + q_0} -  \frac 8{y} \big ) + \cA^{-1} \cA_y \, Q + \cO(e^{-\frac{3\tau}4}).  
\]
For $y \approx -Y_0$, the second term $ \cA^{-1} \cA_y \, Q = \cO(e^{-3\tau/4})$,
 since   $\cA  \sim e^{-\tau/2}$,  $A_y \sim e^{-3\tau/4}$ and $Q \sim e^{-\tau/2}$.  
For the first term, we use  $\cA =  a_0 \, \alpha^{-2}$, for  $a_0 \sim 1$,
$2(n-1)+q_0 \approx  q_{0y} \cdot (y^*+Y_0)$  and $y^* <0$ to obtain 
%$$ \cA \big\,  ( \frac{q_{0y}}{2(n-1) + q_0} -  \frac 6{y^*} \big ) >  a_0 \, \alpha^{-2} \, (Y_0+y^*)^{-1} 
%= a_0 \alpha^{-1} (\xi^*)^{-1}.  $$
We conclude that 
\[
\vartheta Q_y (y^*, \tau) > a_0 \vartheta \alpha^{-1} (\xi^*)^{-1} + \cO(e^{-\frac {3\tau}4})
\]
 and hence condition \eqref{eqn-qneed} clearly holds since $\beta(\tau) >0$.  
We conclude that we can make $q^-_y (y^*,\tau)  > {\tilde q}^-(y^*,\tau)$.

\end{proof} 

Having determined in \eqref{eqn-barq} that our tip region subsolution $\tilde q^-(y,\tau)$ 
is defined in  the  region  $0  < \zeta:= \xi + \beta(\tau)  \leq \xi^* + \beta(\tau)$ which becomes  unbounded as $\tau \gg 1$, 
our final step is  to verify that Lemma \ref{lem-sub-tip} extends in this region, that is $\tilde q^-$ is indeed a subsolution.  
We do this next.

\begin{corollary}\label{cor-sub} Let $\tilde q^-(y,\tau) := \alpha^{-2} w^-(\xi + \beta(\tau), \tau)^2 - 2(n-1)$, $\xi   =  \alpha \, (Y_0 + y)$.  Then, 
$\tilde q^-(y,\tau) $ is a subsolution of \eqref{eq:q} in the region $0  < \xi + \beta(\tau)  \leq \xi^*$.  
\end{corollary} 
\begin{proof} We have seen that $u$ solves \eqref{eq:u} iff $q:= u^2-2(n-1)$ solves \eqref{eq:q}.  
Similarly, ${\tilde q}^-$ is a subsolution of \eqref{eq:q} iff 
$\tilde u^-(y,\tau):=  \alpha^{-1} w^-(\xi+\beta(\tau), \tau)$ is a subsolution of \eqref{eq:u} or equivalently ${\tilde w}^-(\xi,\tau):= w^-(\xi+\beta(\tau),\tau)$
is a subsolution of \eqref{eqn-w202}.  Note that by our choice of $\alpha(\tau) : = - (Y_0' - \frac 12 Y_0)$ in Lemma \ref{lemma-atau},  the last term in \eqref{eqn-w202}
is exactly equal to $-1 \cdot w_\xi$.  Consequently,  we need to have  that $\tw^-$ satisfies 
\begin{equation}		\label{eqn-tw1} 
\frac 1{\alpha^2}\, \tw^-_\tau + \frac 1{\alpha^2} \Big ( \frac{\alpha'}{\alpha}  + \frac 12 \Big )  \big (  \xi  \tw^-_\xi -   \tw \big ) < \frac{\tw_{\xi\xi}^-}{1 + (\tw^-_\xi)^2} -  \frac{n-1}{\tw^-}  - \tw_\xi^-  
\end{equation}

Call $\zeta:= \xi+ \beta(\tau)$ so that ${\tilde w}^-(\xi, \tau):= w^-(\zeta,\tau)$.  By \eqref{eqn-w52} $w^-(\zeta, \tau)$ satisfies 
\begin{equation*}
\frac 1{\alpha^2}  w_\tau^- + \frac 1{\alpha^2} \Big ( \frac{\alpha'}{\alpha}  + \frac 12 \Big )  \big (  \zeta  w_\zeta ^- -  w^- \big ) <  \frac{w_{\zeta\zeta}^-}{1 + (w_\zeta^-)^2} -  \frac{n-1}{w^-}  
- w_\zeta^-  - \frac{\tau}{\alpha^2}  w_\zeta^-.  
\end{equation*}
Since $w_\tau = \tw_\tau - \beta'(\tau) \tw_\xi$ and $\zeta = \xi + \beta(\tau)$,  after rearranging terms, the above equation  can be re-written as 
\begin{equation*}
\frac 1{\alpha^2}  \tw_\tau^- + \frac 1{\alpha^2} \Big ( \frac{\alpha'}{\alpha}  + \frac 12 \Big )  \big ( \xi   \tw_\xi^- -  \tw^- \big ) < 
 \frac{\tw_{\xi\xi}^-}{1 + (\tw_\xi^-)^2} -  \frac{n-1}{\tw^-}  
- \tw_\xi^-   - \frac 1{\alpha^2} \big ( \tau  - \beta'(\tau)   + \beta(\tau) ).  
\end{equation*}
Our choice of $\beta$ guarantees that $\tau  - \beta'(\tau)   + \beta(\tau) >0$ and hence the last formula implies \eqref{eqn-tw1}.  This finishes the proof of the corollary. 

\subsection{The proof of Proposition \ref{prop-sub-sup}} 
 Recall the definition of $\tau_2$ in \eqref{eqn-tau2}.
   First, it is clear that for  every $\varepsilon \geq 0$ and $K >0$,   the function $\cQ^+_{\varepsilon, K}:= \varepsilon \hhm_2(y)  - K e^{-\tau} \hhm_4(y)$  
   is a supersolution of equation \eqref{eq:q} on the set $\big  \{ y: |y| \geq \ell_1 \, \mbox{and} \,\, 2(n-1) + \cQ^+_{\varepsilon, K}(y, \tau) \}  >0\big \}$,
   for all and $\tau \in [\tau_1, \tau_2]$.  This simply follows from  $\partial_\tau \cQ^+_{\varepsilon, K} -  \cL \cQ^+_{\varepsilon, K} =0$ and the error term in \eqref{eq:q} negative. 
   Furthermore, note that that $ u^+_{\varepsilon, K}$ defined by $(u^+_{\varepsilon, K})^2 = 2(n-1) + \cQ^+_{\varepsilon, K}$ defines the profile of a hypersurface that is smooth at its tip 
   $ u^+_{\varepsilon, K}=0$ and hence a supersolution including  tip. 

The fact that $\cQ^-_{\varepsilon, K}$  is a subsolution of \eqref{eq:q} for all $\varepsilon >  0$ and $K >0$ was shown in subsections \ref{sec-sub1}  - \ref{sec-sub3} (see,  in particular, Lemmas \ref{lem-sub-inter}, \ref{lem-sub-tip}, \ref{lem-match} and Corollary \ref{cor-sub}).  
In addition, one can see that the arguments in subsections \ref{sec-sub1} - \ref{sec-sub3} 
also apply when  $\varepsilon =0$ (actually this is a simpler case), hence $\cQ^-_{0, K}$ is a subsolution on $[\tau_0, \tau_2]$, for all $K >0.$
Finally, it is simple  to show that the bound \eqref{eqn-bound-QQQQ} holds for $\cQ^-_{\varepsilon, K}$ and it obviously holds for $\cQ^+_{\varepsilon, K}$, by definition.

\end{proof}

Before we finish this section we will show the following bounds that will play a crucial role in  the next section.
Let us call $\bq(y,\tau) = \bu^2  - 2(n-1)$,
where $\bu(y,\tau)$ is the profile of the peanut solution  which satisfies the asymptotics $\bu(y,\tau) \approx \sqrt{2(n-1) - K  e^{-\tau} y^4}$
according to \cite{AV}.  Then,  $\bq (y,\tau) \approx - K e^{-\tau} y^4$.  

\begin{claim} \label{claim-peanut-toti}Assume that $\varepsilon = M_1 \, e^{-\tau_1}$, and $\tau_1 \gg 1$ so that all asymptotics holds.  
Then, for any $\eta >0$ small but fixed number we have 
\begin{equation}		\label{eqn-p1-p2}
\cQ^-_{\varepsilon, K + \eta } (y, \tau_1) \leq \bq(y, \tau_1) \leq \cQ^+_{\varepsilon,K-\eta}(y, \tau_1)
\end{equation}
on $|y| \geq \ell_2$, provided that $\ell_2  >  \sqrt{2 M_1 \eta^{-1}}$, $\ell_2 \gg 1$  and $\tau_1 \gg 1$.  
\end{claim} 
\begin{proof}[Proof of Claim] 
We will first explain why the estimates in \cite{AV}  imply that for any $\eta >0$
\begin{equation}		\label{eqn-comp-peanut}
- (K + \eta)  \,  e^{-\tau} y^4  \leq  \bq (y, \tau) \leq - (K-\eta) \,  e^{-\tau} y^4 
\end{equation}
holds on  $|y| \geq \ell_1 \gg 1$ and for $\tau \ \gg 1$.  
First, by the analogue of Theorem 7.1 in \cite{AV} we have that 
$\lim_{\tau \to +\infty} \bu (y, \tau) = 2(n-1) - K \, e^{-\tau} y^4$
and the convergence is uniform on any compact subset of $z:= e^{-\tau/4 } y  \in (-(2(n-1)K^{-1})^{1/4}, (2(n-1)K^{-1})^{1/4})$.  
In particular for $z:=  e^{-\tau/4 } y  \in (-z_0/2, z_0/2)$.  Now the barriers constructed in Section 5 in \cite{AV} (see in particular subsection 5.2 (intermediate barrier) and 
Lemmas 5.3 and 5.4 (tip barrier))  imply that the bounds \eqref{eqn-comp-peanut} extend all the way to the tip of $\bu_K$.  

Knowing \eqref{eqn-comp-peanut},  in order  to show \eqref{eqn-p1-p2} it is sufficient to establish the bounds 
\begin{equation}		\label{eqn-p2-p3}
\cQ^-_{\varepsilon, K +  2\eta } (y, \tau_1) \leq - (K + \eta)  \,  e^{-\tau} y^4 \quad \mbox{and} \quad 
- (K-\eta) \,  e^{-\tau} y^4  \leq \cQ^+_{\varepsilon,K-2\eta}(y, \tau_1)
\end{equation}
on $|y| \geq \ell_1 \gg 1$, provided that $\tau_1 \gg 1$.  The second bound is clear, since by definition $\cQ^+_{\varepsilon,K-2\eta}(y, \tau_1) := \varepsilon y^2 - (K-2\eta) e^{-\tau_1} y^4$.  
So, lets concentrate on the first bound.  Recall the definition of $\cQ^+_{\varepsilon, K+2\eta}$ in \eqref{eqn-barq} and lets see first that the desired inequality holds in $\cI_{\ell_2, \ell_2}$
where $\cQ^+_{\varepsilon,K+2\eta}(y, \tau_1) := \varepsilon y^2 - (K+2\eta) e^{-\tau_1} y^4 + \vartheta Q(y, \tau_1)$ with $Q$ as in \eqref{eqn-QQn} when $K$ is replaced by $K+2\eta$.  
In this case, the desired bound is equivalent to  $\varepsilon y^2  + \vartheta Q (y,\tau_1) \leq \eta \, e^{-\tau_1} y^4$.  
We have seen in the proof of Claim \ref{claim-QQQ} that  $Q(y, \tau_1) \leq  \cA(y, \tau_1)  \, \big (  \ln  \ell_2 + \cO(\tau_1)  \big ) =  \cO  (\tau_1 ( \varepsilon^2 y^2 +  e^{-2\tau_1} y^6) )
=  \cO(\tau_1 e^{-2\tau_1} y^6)$, where we used  $\varepsilon = M_1 e^{-\tau_1}$.  Hence $Q(y, \tau_1) < \frac 1 2 \eta \, e^{-\tau_1} y^4$ provided $\tau_1 \gg 1$.  
In addition, we can make  $\varepsilon y^2 = M_1 e^{-\tau_1} y^2 < \frac 12 \, \eta e^{-\tau_1} y^4$, 
provided $\ell_1^2  > 2 M_1 \eta^{-1}$.  Combining the two last bounds yields $ \varepsilon y^2  + \vartheta Q (y,\tau_1) \leq \eta \, e^{-\tau_1} y^4$ which shows that the first bound in \eqref{eqn-p2-p3} holds
on $\cI_{\ell_1, \ell_2}$.  

To finish the proof of the claim, it remains to show that  the first bound in  \eqref{eqn-p2-p3} holds,  near the tip, that is, for 
$ y  = - Y_0 + \alpha^{-1} \xi, \, 0  < \zeta:= \xi +\beta(\tau_1)< \xi^*+\beta(\tau_1)$.  Here $Y_0(\tau_1)$ denotes the vanishing point of $2(n-1) + \varepsilon y^2 - (K+2\eta) e^{-\tau_1} y^4$.  
According to Remark \ref{rem-atau} since $\varepsilon^2  = M_1^2  e^{-2\tau_1} \ll e^{-\tau_1}$ we have   $Y_0^4=  \frac {2(n-1)}{K+2\eta } \, e^{\tau_1} (1+o_{\tau_1}(1))$ and $\alpha(\tau_1) =  \frac 14 \, Y_0(\tau_1)$.  
 In this range \eqref{eqn-barq} and \eqref{eqn-qqq0} tell us that
$\cQ^+_{\varepsilon,K+2\eta}(y, \tau_1) :=  \alpha^{-2} w^-(\xi + \beta(\tau_1), \tau_1)^2 - 2(n-1) $ where $w^-(\xi, \tau_1) = W(\xi) (1+ o_{\tau_1}(1))$ by  \eqref{eqn-wmas}.  
Thus,  it is sufficient to prove that 
\[
\max_{  \xi \in [-\beta(\tau_1), \xi^*]}   \alpha^{-2}  W(\xi+\beta) < 
\min_{  \xi \in [-\beta(\tau_1), \xi^*]} ( 2(n-1) - (K+\eta)  e^{-\tau_1} y^4) (1+ o_{\tau_1}(1))
\]
for $y= - Y_0 + \alpha^{-1} (\xi + \beta)$.  The maximum on the left is attained at $\xi  =  \xi^*$ and in the definition of  $\cQ^+_{\varepsilon,K+2\eta}$ 
 in  \eqref{eqn-barq} we have  chosen $\xi^*$  so that \eqref{eqn-qqq} holds, implying  that 
$ \alpha^{-2}  W(\xi^*+\beta)  = ( 2(n-1) - (K+2\eta)  e^{-\tau_1} y^{*4}) (1+ o_{\tau_1}(1))$.  
On the other hand, the  minimum on the right is attained at $Y_0$.  Hence, it is sufficient to see that 
$2(n-1) - (K+2\eta)  e^{-\tau_1} y^{*4} <  \big ( 2(n-1) - (K+\eta)  e^{-\tau_1} Y_0^4 \big )  (1+ o_{\tau_1}(1))$ for $y^*= - Y_0 + \alpha^{-1} (\xi^* + \beta)$.  
Since $\beta = \cO(\tau_1)$ and $\alpha^{-1} = \cO(e^{- \frac {\tau_1}4})$ we have $y^{*4} = Y_0^4 - \cO(\tau_1 e^{-\frac{\tau_1}4})$, hence the
above clearly holds, provided that $\tau_1 \gg 1$.  
\end{proof}

%%%%%%%%%%%%%%%%%%%%%%%%%%%
%%%%%%%%%%%%%%%%%%%%%%%%%%%
\section{$L^2$ arguments to control the solution on compact sets}
\label{sec-right-bc}

For the purpose  of finding perturbations of the peanut solution that develop  nondegenerate neckpinch singularities, we consider the funnel defined by Definition \eqref{def-funnel}, where $M_1$ is an arbitrary
 uniform constant that will be chosen later.    Apply Proposition \ref{lemma-big-unstable} and Lemma \ref{lemma-homotopy} to find sufficiently big $\tau_0$ (we can choose it big enough so that Theorem \ref{prop-sphere} holds as well), so that  for every $\epsilon > 0$, and every $\bar \bom := (\bar \Omega_0, \bar \Omega_2)$, where $\bar \Omega_0^2+\bar  \Omega_2^2 = 1$, there exists an initial data 
$u_{\epsilon, \bom}(y,\tau_0)$  at time $\tau_0 \gg 1$, defined by \eqref{eq-initial},  and  time $\tau_1 > \tau_0$ so that 
\begin{equation}		\label{eq-u-tau1-omega}
\big(u_{\epsilon, \bom}(y,\tau_1) - \bu(y,\tau_1)\big)\, \eta (y,\tau_1)  =   e^{-\tau_1} \, M_1 \, \big ( \bar \Omega_0 +   \bar \Omega_2\, H_2(y) \big ) + o(e^{-\tau_1}) \big ),
\end{equation}
in the $L^2$-sense. Recall that \eqref{eq-initial} means that $u_{\epsilon, \bom}(\cdot,\tau_0)$ is an $\epsilon$-perturbation  of the peanut solution $\bar u(\cdot, \tau_0)$ at time $\tau_0$, where     ${\bf  \Omega} = (\Omega_0, \Omega_2)$  is the $2$ dimensional parameter that we need to choose at time $\tau_0$ so that at time $\tau_1$ we have  \eqref{eq-u-tau1-omega}. 
 
 Also, recall  that  $\eta(y,\tau)$ is the same cut off function as in Section \ref{sec-outline}, that is, 
 $\eta(y,\tau) := \eta_0 \left(\frac{y}{\rho e^{\tau/4}}\right)$, for some $\rho >0$, and $\eta_0(y)$ is a  cut off function defined   in \eqref{eqn-def-eta0}. 

%\smallskip 
%
%Observe next that  \eqref{eq-u-tau1-omega}, Lemma \ref{cor-expansion}  and the peanut asymptotics in \eqref{eq:um2} imply 
%{\color{blue} Are we using this ? \begin{equation}
%\label{eq-tau1}
%u_{\epsilon, \bom} (y,\tau_1) = \sqrt{2(n-1)} - K_0  H_4(y) e^{-\tau_1} + M_1\, e^{-\tau_1} \big ( \Omega_0 +  \, \Omega_2\, H_2(y) \big ) + o(e^{-\tau_1}).  
%\end{equation}
%on a large compact set $|y| \le 2\ell_1$.   } 

\smallskip
To simplify the notation, for the rest of the section  we use $u(y,\tau)$ instead of $u_{{\bf \Omega},\epsilon}(y,\tau)$. 
Define 
\begin{equation}		\label{eqn-qV}
q(y,\tau) := u(y,\tau)^2 - 2(n-1), \qquad  V(y,\tau) := q(y,\tau)\,  \eta(y,\tau)
\end{equation}
where $\eta(y,\tau)$ is the same cut off function that appears in \eqref{eq-u-tau1-omega}.   Recall that $q(y,\tau)$ satisfies equation \eqref{eq:q}.  An easy computation shows that $V(y,\tau)$ satisfies the equation of form 
\[
V_{\tau} = \mc L V + \eta^2 \, \bar{\mc{E}}_1(q) + \bar{\mc E}_2(\eta,q)
\]
where $\mc{E}_1(q) = -\frac{(q_{yy}+2)q_y^2}{8(n-1)+4q+q_y^2}$ and $\bar{\mc {E}_2}(\eta,q)$ is the error coming from the cut off function and its derivatives.  
Let $V_+ := \pi^+(V)$ where $\pi^+$ is the projection of $L^2(\mathbb R, e^{-y^2/4})$ onto $\langle \hm_0\rangle$.  
%in \eqref{eqn- defined as follows.  Choose a function $0 \le \eta_0 \in C_c^{\infty}(\mathbb{R})$ with $\eta_0(s) = 1$ for $s\le 1$ and $\eta_0(s) = 0$ for $s \ge 2$.  Define $\eta(y,\tau) = \eta_0\left(\frac{y e^{-{\tau/4}}}{\rho}\right)$, where $\rho > 0$ is the same constant as constant as in section \ref{sec-outline}.  
Equation \eqref{eq-u-tau1-omega} and the asymptotics for $\bu$ in \eqref{eq:um2} imply 
% \[
%\Big(u(y,\tau_1) - \sqrt{2(n-1)} + K_0 \, H_4(y) e^{-\tau_1} \Big) \, \eta (y,\tau_1) -   e^{-\tau_1} \big ( M_1\, \bar \Omega_0 + M_1 \, \bar \Omega_2\, H_2(y) \big) =   o(e^{-\tau_1})
%\]
\[
\big(u(y,\tau_1) - \sqrt{2(n-1)}  \big) \, \eta (y,\tau_1) -   e^{-\tau_1} \big ( M_1\, \bar \Omega_0 + M_1 \, \bar \Omega_2\, H_2(y) -
 K_0 \, H_4(y)  \big) =   o(e^{-\tau_1})
\]
where we also used that the $L^2$-norm of $K_0 \, H_4(y) e^{-\tau_1}$ outside the support of $\eta$ is $o(e^{-\tau_1})$. 
The last estimate in turn gives that  the following holds in the $L^2$ sense:
\begin{equation}		\label{eq-Qtau1}
f:= V(\cdot,\tau_1) -  2\sqrt{2(n-1)}  \, e^{-\tau_1}\, M_1   \, \big (  \bar \Omega_0  +  \bar \Omega_2 \, H_2 - K_0  \, H_4 \big ) = o(e^{-\tau_1}). 
\end{equation}
To show \eqref{eq-Qtau1}  we  split  
$\int f^2   e^{-y^2/4}  dy = \int_{|y| \leq   e^{\tau_1/8} }  f^2 \, e^{-y^2/4} \, dy +   \int_{|y| >  e^{\tau_1/8} }  f^2 \, e^{-y^2/4} \, dy$
and use the   bound (i)  in Lemma \ref{cor-prop-funnel} to get  $ \int_{|y| >  e^{\tau_1/8} }  f^2 \, e^{-y^2/4} \, dy = o(e^{-\tau_1}).   $ 
To bound the the integral on $|y| \leq   e^{\tau_1/8}$ we write   $q =   (u-\sqrt{2(n-1)}) \, (u + \sqrt{2(n-1)})$,  and use again (i)  in Lemma \ref{cor-prop-funnel}, to get 
$ u + \sqrt{2(n-1)} =    2 \sqrt{2(n-1)} + O(e^{-\tau_1/8})$, 
 holding in the $L^\infty$ sense on  $|y| \leq e^{\tau/8}$. Thus, $ \int_{|y| >  e^{\tau_1/8} }  f^2 \, e^{-y^2/4} \, dy = o(e^{-\tau_1})$,
 and \eqref{eq-Qtau1} holds. 

\smallskip
Now \eqref{eq-Qtau1} implies 
\begin{equation}		\label{eq-Q+-tau1}
V_+(\tau_1) = 2 \sqrt{2(n-1)}\, M_1\, \bar \Omega_0 e^{-\tau_1}\langle 1, 1\rangle + o(e^{-\tau_1}),
\end{equation}
 implying that 
\begin{equation}		\label{eq-V+theta}
|  V_+(\tau_1) - 2\sqrt{2(n-1)} M_1 \bar \Omega_0 e^{-\tau_1}\langle1,1 \rangle | < \theta \, e^{-\tau_1},
\end{equation}
where $\theta > 0$ is a small number, independent of the choice of $\bar \Omega_0$, and it  can be made very  small by taking $\ell_0$  very large (which can be seen by part (iii) of  Lemma \ref{cor-prop-funnel}).  
\smallskip

Let us fix  $\varepsilon$ and $\tau_2$  such that 
\begin{equation}		\label{eqn-fixve}
\varepsilon := 2\sqrt{2(n-1)}\, M_1 e^{-\tau_1}  \qquad \mbox{and} \quad  \tau_2 :=  2 \ln ( \varepsilon ^{-1} \sigma_n ) 
\end{equation}
where $\sigma_n$ is a   sufficiently large  constant that will be chosen later in section \ref{sec-main-thm} to depend  only on dimension and $K_0$.  Note that the same  choice of $\tau_2$ also appeared in the previous section in \eqref{eqn-tau2}.

\smallskip

 This section is dedicated to proving  Proposition \ref{prop-q-beh-bound}, by  employing   $L^2$-theory methods.  We will make the a'priori assumption that  the function $q(y,\tau)$ is defined  for $|y| \leq  2 \rho\,  e^{\tau}$ and $\tau \in [\tau_0, \tau_2]$,  for some $\rho >0$, and that {\em  in the time interval $[\tau_1, \tau_2]$} 
it  satisfies  the bounds 
\begin{equation}
 \label{eq-Linfty-qq}
|q(y,\tau)| +  |q_y(y,\tau)| + |q_{yy}(y,\tau)| \le  \varepsilon  \, \Lambda \, (1 + |y|^4), \qquad |y| \leq  2 \rho\,  e^{\tau}
\end{equation}
for some auxiliary constant $\Lambda >0$.  The constant $\Lambda$ will be  found  in Proposition \ref{prop-barriers} 
 where it will be  shown that the domain of  $q(\cdot,\tau)$ 
contains the interval $|y| \leq 2 \rho  e^\tau$,  for some $\rho >0$. It turns out that both $\Lambda$ and $\rho$ can be taken to depend only on $M_1$, $K_0$ and the  dimension $n$. 

\begin{proposition}
\label{prop-q-beh-bound} 
Assume that \eqref{eq-Linfty-qq} holds. Then
we  can  choose $\bar \Omega_0$ in \eqref{eq-u-tau1-omega} so that
\begin{equation}		\label{eq-q-asymp-help}
q(y,\tau) = \varepsilon \, H_2(y) - \bar{K}_0 e^{-\tau}\, H_4(y) + o(\varepsilon),
\end{equation}
holds in the $L^2$ sense, for all $\tau\in [\tau_1,\tau_2]$, where $\bar{K}_0 := 2\sqrt{2(n-1)} K_0$.
\end{proposition}

 The proof of Proposition \ref{prop-q-beh-bound}  follows from a series of Lemmas that follow. Similar computation  as in the proof of Lemma 4.1 in \cite{AV}, using \eqref{eq-Linfty-qq},  yields 
\begin{equation}		\label{eq-Q+}\frac{d}{d\tau} |  V_+(\tau) |  \ge  | V_+(\tau) |- B_1\, \Lambda^2 \varepsilon^2,
\end{equation}
where $B_1$ is a universal constant depending only on $n, m$ and $\rho$.  

If $u(\cdot,\tau)$, $\tau \in [\tau_0, \tau_2]$ and $\tau_2 >\tau_1$,  is the MCF solution as above such that at each time $\tau$ its domain contains  $[-2\rho e^{\tau/4}, 2\rho \, e^{\tau/4}]$, we will say that 
\begin{equation}		\label{eq-funnel-unstable}
u(y,\tau) \in \tilde{\mc F}_{\theta_0} (\tau) \quad  \mbox{if} \quad | V_+(\tau) | \le  2\theta_0 \,   \varepsilon,
\end{equation}
where $\theta_0 := \max\{2\theta, 4B_1 \Lambda^2  \varepsilon\}$, $\theta$ is the same small constant as in \eqref{eq-V+theta}, and $B_1$ and $\Lambda$ are as in \eqref{eq-Q+}.
%Note also that by the degree theory lemma we can assume that for any choice of $\bar \Omega_0$ and $\bar \Omega_2$, with $\bar \Omega_0^2 + \bar \Omega_2^2 = 1$, we can find perturbation of the peanut solution at time $\tau_0$ so that \eqref{eq-tau1} holds at some time $\tau_1 > \tau_0$.  

\begin{lemma}[Exit lemma]
\label{lemma-exit}
Let $u(y,\tau)$, where $\tau\in [\tau_0, \tau_2]$, and $\tau_2 > \tau_1$, is a solution to \eqref{eq:u} starting from a perturbation of the peanut solution $\bu(y,\tau)$ at time $\tau_0$, so that at time $\tau_1$ we have \eqref{eq-u-tau1-omega}, and $\tau_2$ is defined in \eqref{eqn-fixve}.   Assuming \eqref{eq-Linfty-qq}, if $u(\cdot,\bar \tau) \in \partial\tilde{\mc F}_{\theta_0}(\bar \tau)$, and $u(\cdot,\tau)\in \tilde{\mc F}_{\theta_0}(\tau)$ for all $\tau\in [\tau_1,\bar \tau]$ then
\[
\frac{d}{d\tau} |V_+(\tau)|_{\tau = \bar \tau} > 0.
\]
\end{lemma}

\begin{proof}
The statement immediately follows from \eqref{eq-Q+}, and the definition of $\tilde{\mc F}_{\theta_0}(\tau)$ given by \eqref{eq-funnel-unstable}. More precisely, as long as  $u(\cdot,\tau) \in \tilde{\mc F}_{\theta_0}(\tau)$, since we assume \eqref{eq-Linfty-qq}, inequality \eqref{eq-Q+} holds. 
This together with \eqref{eq-funnel-unstable} imply that at the first time $\bar{\tau} > \tau_1$ that it happens $u(\cdot,\bar{\tau}) \in \partial \tilde{\mc F}_{\theta_0}(\bar{\tau})$, we have that $\|V_+(\tau)\| = 2\theta_0 \, \varepsilon$, and 
\begin{equation}		\label{eq-exit}\frac{d}{d\tau}| V_+(\tau)|_{\tau = \bar{\tau}} \geq 2\theta_0 \,  \varepsilon  -  B_1\, \Lambda^2 \varepsilon^2  >0
\end{equation}
\end{proof}

%{\color{red} {\bf Toti: Lets discuss this.  I think we have $|\Omega_0| < 1$. Isn't this sufficient ? Can't we just choose $\theta_0 = 2\sqrt{2(n-1)}\delta $ ? }  

%Note that in order to have $u(\cdot,\tau_1) \in \tilde{\mc{F}}_{\theta_0} (\tau_1)$, since $\delta < 1$,  we can choose $\Omega_0$ so that  
%\begin{equation}
%\label{eq-first-cond}
% |2\sqrt{2(n-1)}M_1\Omega_0|\langle 1, 1\rangle \le \theta_0 
% \end{equation} } 
Denote by $\bar{\theta}_0 := \frac{\theta_0}{2\sqrt{2(n-1)}M_1\langle 1, 1\rangle}$. Our goal is to show that we can find at least one such $\bar \Omega_0$ so that our solution stays in $\tilde{\mc F}_{\theta_0}(\tau)$ for all $\tau\in [\tau_1,\tau_2]$. More precisely we have the following result.

\begin{lemma}
\label{lemma-omega0} 
Define the interval  $\mc I := [-\bar{\theta}_0, \bar{\theta}_0]$. There exists an $\bar \Omega_0\in \mc I$, for which the solution $u(y,\tau)$ stays in the funnel $\tilde{\mc F}_{\theta_0}(\tau)$, for all $\tau\in [\tau_1, \tau_2]$.
\end{lemma}

\begin{proof}
 If there exists an $\bar \Omega_0 \in \mc I$   for which \eqref{eq-Q+-tau1} holds,  and which stays inside the set $\tilde{\mc F}_{\theta_0}(\tau)$ for all $\tau\in [\tau_1,\tau_2]$, we are done. Hence, assume that for every $\bar \Omega_0 \in \mc I$, and a solution for which \eqref{eq-Q+-tau1} holds,  there exists the first exit time $\tau_{ex}(\bar \Omega_0) < \tau_2$ at which the solution hits the boundary of the set $\partial \tilde{\mc F}_{\theta_0}(\tau_{ex}(\bar \Omega_0))$, and therefor it exits the funnel  by \eqref{eq-exit}. Note also that for every $\bar \Omega_0 \in \mc I$  our solution  for which \eqref{eq-Q+-tau1} holds,  belongs to $\tilde{\mc F}_{\theta_0}(\tau_1)$.

 We claim that the exit time $\tau_{ex}(\bar \Omega_0)$ is a continuous function of $\bar \Omega_0$, on an interval $\bar \Omega_0\in \mc I$. To justify this claim we use Lemma \ref{lemma-exit} and argue similarly as in the proof of Lemma 3.2 in \cite{AV}.  
 
We define next 
\[
\mu(\bar \Omega_0) = \frac{\langle V, 1\rangle(\tau_{ex}(\bar \Omega_0))}{|\langle V, 1\rangle(\tau_{ex}(\bar \Omega_0))|},
\]
where we recall that  $V$ is the function  given in \eqref{eqn-qV}, it  satisfies  \eqref{eq-Q+-tau1},
and  whose exit time is $\tau_{ex}(\bar \Omega_0)$. Note that 
\[
\mu: \mc I  \to \{-1,1\}
\]
 is a continuous map and recall the definition $\mc I:=[-\theta_0, \theta_0]  $. 
Next we claim that $\mu(-\bar{\theta}_0) = -1$ and  $\mu(\bar{\theta}_0) = 1$, which would then obviously contradict the continuity of the map $\mu$ on a given interval. This would then conclude the proof of the Lemma.

Indeed, let us show  $\mu(\bar{\theta}_0) = 1$, since the other statement is proved similarly. Similar computation to the one to derive \eqref{eq-Q+}, using again \eqref{eq-Linfty-qq} yields
\[
\frac{d}{d\tau} V_+(\tau) \ge  V_+(\tau) - B_1 \Lambda^2 \, \varepsilon^2.
\]
If $\bar \Omega_0 = \bar{\theta}_0$, then our  definition $\bar{\theta}_0 := \frac{\theta_0}{2\sqrt{2(n-1)}M_1\langle 1, 1\rangle}$,   \eqref{eq-Q+-tau1} and $\|V_+(\tau)\| = 2\theta_0 \, \varepsilon$ 
imply  that $V_+(\tau_1) = \theta_0 \,  e^{-\tau_1} + o(e^{-\tau_1})$, thus integrating the equation for $V_+(\tau)$ from $\tau_1$ to $\bar{\tau} := \tau_{ex}(\bar{\theta}_0)$, and recalling that $\varepsilon =  2 \sqrt{2(n-1)} \,  M_1 e^{-\tau_1}$ yield
\[
V_+(\bar{\tau})\, e^{-\bar{\tau}} \ge e^{-\tau_1} V_+(\tau_1) - B_1 \Lambda^2  \varepsilon^2 e^{-\tau_1}> 0,
\]
for big enough $\tau_1$, since
\[
V_+(\tau_1)  =  \theta_0 \, e^{-\tau_1} + o(e^{-\tau_1})  \ge \frac{\theta_0}{2} e^{-\tau_1} \ge 2B_1 \Lambda^2 \varepsilon^2,
\]
where we have used  the definition of $\varepsilon$ in \eqref{eqn-fixve} and $\theta_0$ in \eqref{eq-funnel-unstable}.
This implies  $\mu(\bar{\theta}_0) = 1$. Similarly,  $\mu(-\bar{\theta}_0) = -1$, using
\[
\frac{d}{d\tau}V_+(\tau) \le V_+(\tau) + B_1 \Lambda^2 \varepsilon^2.
\]

The above argument shows that indeed there exists an $\bar \Omega_0$ in the interval  $\mc I$, for which the solution $u(y,\tau)$ stays inside $\tilde{\mc F}_{\theta_0}(\tau)$, for all $\tau\in [\tau_1, \tau_2]$, hence finishing the proof of the Lemma.
\end{proof}

%
% {\color{red} Lets discuss the next  sentence:  Based on the previous arguments, given that   the $L^{\infty}$ estimate  \eqref{eq-Linfty-qq} holds, we may assume that we have chosen an $\epsilon$ small so that for the $\epsilon$ perturbation of the peanut solution at time $\tau_0$, we have both,  \eqref{eq-u-tau1-omega} and \eqref{eq-Q+-tau1} holding at time $\tau_1$, and we can always choose a perturbation at time $\tau_0$ so that equations  \eqref{eq-u-tau1-omega} and \eqref{eq-Q+-tau1} hold for an arbitrary choice of $(\Omega_0,\Omega_2) \in \mathbb S^1$.}  Recall that by Lemma \ref{lemma-omega0},  we can choose parameter $\Omega_0 \in  \mc I$  so that in addition to above, $u(\cdot,\tau) \in \tilde{\mc F}_{\theta_0}(\tau)$ for all $\tau\in [\tau_1, \tau_2]$ as well. For this perturbation we have the following result which shows the asymptotic behavior of our solution from $\tau_1$ to $\tau_2$ (in the $L^2$-sense). 

\begin{lemma}
\label{lemma-L2-V}
Let $u(y,\tau)$ be the mean curvature flow solution for which $\bar \Omega_0$ is chosen so that \eqref{eq-Q+-tau1} holds, and so that $u(\cdot,\tau) \in \mc F_{\theta_0}(\tau)$ for all $\tau\in [\tau_1,\tau_2]$ $($the existence of such $\Omega_0$ is guaranteed by Lemma \ref{lemma-omega0}$)$. Assume also that  \eqref{eq-Linfty-qq} holds. Then, for all $\tau\in [\tau_1,\tau_2]$
we have
\[
V(y,\tau) = \varepsilon\, H_2 - \bar{K}_0 e^{-\tau} H_4 +  o(\varepsilon),
\]
in the $L^2$-sense, where $\varepsilon$ is defined in   \eqref{eqn-fixve},  and $\bar K_0 := 2\sqrt{2(n-1)} K_0$,  with  $K_0$
being the constant that appears in the  asymptotics of the peanut solution $\bu$.
\end{lemma}

%{\color{blue} I do not think that we need the $\delta$ since we are in the case where $\epsilon \gg  e^{-\tau}$ so the $\delta e^{-\tau} \ll = o(\epsilon)$. Am I correct ?
%}
\begin{proof}
Denote by $\beta(\tau) = \langle V, H_4\rangle$. The fact that $u(\cdot,\tau) \in \mc F_{\theta_0}(\tau)$ for all $\tau\in [\tau_1,\tau_2]$ yields  the domain of $q(\cdot, \tau)$ contains the interval $|y| \leq 2\rho e^{\tau}$. Since we also have that \eqref{eq-Linfty-qq} holds, similarly to the proof of Lemma 4.1 in \cite{AV},  we have
\[
-\beta(\tau) - B_1\Lambda^2 \varepsilon^2  \le \frac{d}{d\tau}\, \beta(\tau) \le -\beta(\tau) + B_1 \Lambda^2 \varepsilon^2.
\]
Integrating this from $\tau_1$ to $\tau\in [\tau_1,\tau_2]$, and using that
\[
 \beta(\tau_1) = -\bar{K}_0\, \langle H_4, H_4 \rangle \, e^{-\tau_1}  + o(e^{-\tau_1})
\]
(which follows from \eqref{eq-Qtau1}),  yields
\[
 -  B_1 \Lambda^2 \, \varepsilon^2  \le \beta(\tau) + \bar K_0 \, \langle H_4, H_4\rangle (1 + o(1))\, e^{-\tau} \le  B_1 \Lambda^2 \varepsilon^2,
\]
which implies
\[
\beta(\tau) = - \bar{K}_0\, \langle H_4, H_4\rangle \, e^{-\tau} + o(\varepsilon).
\]

Let us denote by $\alpha(\tau) := \langle V, H_2\rangle$. Then, similarly as in \cite{AV},
\[
\Big|\frac{d}{d\tau}\alpha(\tau)\Big| \le B_1 \Lambda^2\, \varepsilon^2,
\]
implying
\[
\alpha(\tau_1) - B_1  \Lambda^2\, \varepsilon^2 \, ( \tau -\tau_1) \le \alpha(\tau) \le \alpha(\tau_1) + B_1 \Lambda^2 \varepsilon^2 \, (\tau - \tau_1).
\]
My \eqref{eq-Qtau1}, we have $\alpha(\tau_1) = \varepsilon \, \bar \Omega_2\, \langle H_2, H_2\rangle  + o(\varepsilon)$,  $\varepsilon = 2\sqrt{2(n-1)}\, M_1 e^{-\tau_1}$, and thus
\[
\alpha(\tau) = \varepsilon\, \bar \Omega_2 \, \langle H_2, H_2\rangle+ o(\varepsilon),
\]
for all $\tau\in [\tau_1,\tau_2]$, provided $\tau_1$ is sufficiently large. Here we have used that for all $\tau\in [\tau_1,\tau_2]$,  we have $\varepsilon^2e^{\tau} \le \sigma_n$, and hence $\varepsilon^2 \tau = o(\varepsilon)$ in that time interval. Furthermore, since our solution stays in $\tilde{\mc F}_{\theta_0}(\tau)$ for all those times, we have $|V_+(\tau)| \le 2\theta_0 \,\varepsilon$, and $\theta_0$ is as in \eqref{eq-funnel-unstable}.

Finally, let us denote by $V^s$ the projection of $V$ onto the space $\langle H_6, H_8, \dots\rangle$. We claim that we have  $\|V^s(\tau)\| \le o( \varepsilon)$ for $\tau\in [\tau_1,\tau_2]$ as well. Indeed,
\[
\frac{d}{d\tau}\|V^s\| \le -2\, \|V^s\| + B_1 \Lambda^2 \varepsilon^2,
\]
implying
\[
\|V^s\| \le o(1) \, e^{-2\tau + \tau_1} + B_1 \Lambda^2 \varepsilon^2 = o(\varepsilon),
\]
where we have used that $\|V^s(\tau_1)\| = o(e^{-\tau_1})=o(\varepsilon)$.
Finally, we conclude that
\[
V(y,\tau) = \varepsilon\, H_2 - 2\sqrt{2(n-1)} \, \bar K_0 \, e^{-\tau} H_4 + o(\varepsilon),
\]
in the $L^2$-sense, for all $\tau\in [\tau_1,\tau_2]$, as claimed.
\end{proof}

We can now finally give a proof of Proposition \ref{prop-q-beh-bound}.

\begin{proof}[Proof of Proposition \ref{prop-q-beh-bound}]
The proof readily  follows by Lemma \ref{lemma-omega0} and Lemma \ref{lemma-L2-V}.
\end{proof}

The following Lemma follows from  Proposition \ref{prop-q-beh-bound}.

\begin{lemma}  
\label{lemma-asymp-2}
Assuming that the domain $q(y, \tau)$ contains  $|y| \leq  2\rho e^{\tau/4}$ for some $\rho >0$ and  \eqref{eq-Linfty-qq} hold,
for every $\ell \gg 1$ there exists $\tau_0 \gg 1$ so that for all $\tau\in [\tau_1,\tau_2]$, we have
\begin{equation}		\label{eq-asymp-2}
\begin{split}
&q(y,\tau) =  \epsilon \, H_2(y)  - \bar{K}_0\, e^{-\tau} \, H_4(y)  + o(\varepsilon),\\
&q_y(y,\tau) =  \epsilon \, H_2'(y) - \bar{K}_0\, e^{-\tau} \, H_4'(y) + o(\varepsilon)
\end{split}
\end{equation}
on $|y| \leq \ell$, where $\bar K_0= 2\sqrt{2(n-1)} K_0$, and $K_0$  is the  constant that appears in the asymptotics of the peanut solution. 

\end{lemma} 

\begin{proof}
The proof  similar to the proof of Lemma \ref{cor-expansion}, although here is simple.  Call $f(\cdot, \tau) = q(\cdot, \tau)  - \epsilon \, H_2  + \bar{K}_0\, e^{-\tau} \, H_4$.
Then  by \eqref{eq:q} we have that 
$f_\tau - \cL f = \cE_1(q)$, where $\mc{E}_1(q) = -\frac{(q_{yy}+2)q_y^2}{8(n-1)+4q+q_y^2}$,  and by 
\eqref{eq-Linfty-qq}, the estimate  $\mc{E}_1(q)  = O(\varepsilon^2) \, \ell^8$, holds for all $|y| \leq 4\ell$, $\tau \in [\tau_1, \tau_2]$.
Furthermore,  Proposition \ref{prop-q-beh-bound} shows that $\| f(\cdot, \tau) \|_{\hilb( [0, \rho e^\tau]) } = o(\varepsilon)$. 
Recall that  $\varepsilon :=  2\sqrt{2(n-1)}\, M_1 e^{-\tau_1} $. For any fixed $\ell \gg 1$, we have $\ell \ll \rho \, e^{\tau_1}$, provided $\tau_0 \gg 1$.  Hence, the last estimate implies that 
$\| f(\cdot, \tau) \|_{\hilb( [0, 4\ell]  } = o(\varepsilon)$, for $\tau \in [\tau_1, \tau_2]$. One can then employ standard $L^\infty$-estimates 
as in Lemma \ref{cor-expansion},  to conclude the first bound in \eqref{eq-asymp-2}, holding on $[0, 2\ell]$, for $\tau \in [\tau_1, \tau_2]. $
The second bound follows by standard derivative estimates.
Note  that we can make  $ \varepsilon \, \ell^8 =  o(1)$, by choosing $\tau_0 \gg 1$. 

\end{proof}

%%%%%%%%%%%%%%%%%%%%%%%%%%%%%%%%

\section{Proof of Theorem \ref{thm-main}}
\label{sec-main-thm}

 To conclude the proof of Theorem \ref{thm-main} we first show the following.
 
\begin{theorem}
\label{thm-cylinder}
Let $\bar{M}_0(t)$ be the peanut solution as discussed above, and let $\bar{T}$ be its first singular time. There exists a $t_0$ sufficiently close to $\bar{T}$, so that in every sufficiently small neighborhood of $\bar{M}_0(t_0)$, there exists a perturbation $\bar{M}_{\theta_c}(t_0)$ so that the MCF starting at $\bar{M}_{\theta_c}(t_0)$  develops a nondegenerate neckpinch singularity. Here $\theta_c$ can be chosen arbitrarily small.
\end{theorem}

To prove Theorem \ref{thm-cylinder} we will work below with the rescaled equation. Let $\tau_1$ and $\tau_2$ be the (rescaled)  times defined  in Sections  \ref{sec-sub-super}  and \ref{sec-right-bc}, where we also defined 
$\varepsilon := 2 \sqrt{2(n-1)} \, M_1 e^{-\tau_1} $. Recall that at time $\tau_1$, \eqref{eq-u-tau1-omega} holds,  and that $\tau_2 := 2 \ln \big ( \frac{\sigma_n}{\varepsilon} \big )$
was defined in Section \ref{sec-sub-super} 
(see   \eqref{eqn-tau2}), that is $\tau_2 = 2 \ln \big ( \frac{\sigma_n}{\varepsilon} \big )$, where $\sigma_n$ can be any fixed constant (as large as we wish).  

\smallskip

The main tool in the proof of Theorem \ref{thm-cylinder} will be the use of   super and sub solutions  $\cQ^+_{\varepsilon^+,K^+}(y,\tau)$ and  $\cQ^-_{\varepsilon^-, K^-}(y,\tau)$, that were  constructed in Section 
\ref{sec-sub-super} and are defined  for  $|y| \geq \ell_1$,  where $\ell_1$ is a large but fixed constant. 
The significance of these barriers  is that they allow us to extend the asymptotics for $q(y,\tau)$ 
 shown in \eqref{eq-asymp-2}    from $|y| = \ell_1$ (a fixed number) all the way to the tip.
As a conclusion we will prove that the $L^\infty$-bound  \eqref{eq-Linfty-qq} that we assumed in Section  \ref{sec-right-bc} holds  with a constant $\Lambda$ that depends only on $K_0$
and the dimension.

We will take $\varepsilon^-, \varepsilon^+$ very  close to $\varepsilon$ and $K^-, K^+$  near $\bar K_0:= 2 \sqrt{2(n-1)} K_0$. 
Consequently, these supersolutions and subsolutions   will  be defined on the same interval $[\tau_1, \tau_2]$ (since $\sigma_n$ in \eqref{eqn-tau2} can be any constant). 
We  remind the reader that  by definition  $\cQ^+_{\varepsilon^+,K^+}(y,\tau):= \varepsilon^+ \, H_2(y) - K^+ \, e^{-\tau} \, H_4(y)$
while $\cQ^-_{\varepsilon^-, K^-}(y,\tau)$ is given by \eqref{eqn-barq},  defined differently in the intermediate $\cI_{\ell_1, \ell_2}$ and tip $\cT_{\ell_2}$ regions. 
In the intermediate region $\cQ^-_{\varepsilon^-, K^-}(y,\tau) := \varepsilon^-  y^2 - K^- \, e^{-\tau} y^4 + q_1(y,\tau)$ where $q_1$ is of  lower order, while the  tip region happens at a much smaller scale around the tip. 
Hence, $\cQ^-_{\varepsilon^-, K^-}(y,\tau) \approx \varepsilon^-  y^2 - K^- \, e^{-\tau} y^4$ all the way up  to a very tiny neighborhood of the tip. We invite the reader to have this in mind, as we are doing the comparisons 
below. 
\smallskip 

In the previous section we assumed that the $L^\infty$-estimate \eqref{eq-Linfty-qq} holds on $|y| \leq 2\rho e^{-\gamma \tau}$, where $\Lambda$ is an auxiliary constant.
As a result, we saw in Lemma \ref{lemma-asymp-2} that asymptotics \eqref{eq-asymp-2} hold.  We will now show that the constant $\Lambda$ can be taken to 
 depend only on the initial data. 

From now on we assume that $\ell_0$  is sufficiently large constant so that the results in previous sections hold. In what follows we will take $\ell_1$ large so that $2\ell_0 \le \ell_1 \le 1000 \ell_0 < \rho e^{\tau_0/4}$.

\begin{proposition}
\label{prop-barriers}
Let $\varepsilon := 2\sqrt{2(n-1)} M_1\, e^{-\tau_1}$, and $\tau_1 \gg \tau_0$ be the exit time as before so that   \eqref{eq-u-tau1-omega} holds. There exists a uniform constant $C_0(n, K_0, M_1) > 0$  so that the $L^\infty$-estimate 
\begin{equation}		\label{eqn-linfty-q} 
|q(y,\tau)| +  |q_y(y,\tau)| +  |q_{yy}(y,\tau)| \le C_0 \,  \varepsilon \, (1 + |y|^4),
\end{equation}
holds   for $|y| \le 2 \rho \, e^{{\tau/4}}$  and $\tau\in [\tau_1,\tau_2]$.  
This verifies that  \eqref{eq-asymp-2} holds. In addition to that
we have 
\begin{equation}		\label{eqn-barriers-tau2}
\cQ^-_{(1-\bar{\eta})\varepsilon, (1+ \bar{\eta})\bar{K}_0}(y,\tau) \le q(y,\tau) \le \cQ^+_{(1+\bar{\eta})\varepsilon,(1-\bar{\eta})\bar{K}_0}(y,\tau)
\end{equation}
for all $|y| \ge \ell_1$, $\tau \in [\tau_1, \tau_2]$,  where $\bar \eta$ is  small.  Here $\bar{K}_0 := 2\sqrt{2(n-1)}\, K_0$. 
\end{proposition}

For the proof of this proposition we first need the following few lemmas.

\smallskip

\begin{lemma}
\label{lemma-initial-comp}
There exists $\eta > 0$ small so that
\begin{equation}		\label{eqn-ic}
\cQ^-_{\varepsilon (1-\eta), \bar{K}_0+\eta}(y,\tau_1)  \le q(y,\tau_1) \le  \cQ^+_{\varepsilon(1+\eta), \bar{K}_0-\eta}(y,\tau_1)
\end{equation}
on $|y| \geq \ell_1$.  We can take  $\eta$ so that   $\eta \ge C(M_1,K_0)\,  \ell_0^{-2}$. 
\end{lemma}

\begin{proof} 
Recall that $\ell_1$ satisfies  
$2\ell_0 \le \ell_1 \le 1000 \ell_0 < \rho e^{\tau_0/4}$, where $\ell_0$ is as in Proposition \ref{prop-inout}. 
The proof of Lemma \ref{cor-lll} implies that for any $\tilde \eta >0$ small, if we choose $\alpha$ such that $e^{-\alpha} = 1- \tfrac{\tilde \eta}4$, we have 
\[
\bu(y, \tau_1-\alpha) \leq u(y, \tau_1) \leq  \bu(y, \tau_1 + \alpha), \qquad \mbox{on} \,\, |y| \geq \ell_1.
\]
provided that  $C(M_1, K_0)  \ell^{-2}_0 < \tilde \eta$ and  $\tau_0 \gg 1$. 
The above combined with   the peanut asymptotics  and the definition of $\alpha$ 
yields  (similarly as in  Lemma \ref{cor-lll}) that 
\begin{equation}		\label{eqn-q34}(1- \tfrac{\tilde \eta}2) \,  K_0 \, y^4 \, e^{-\tau} \leq     \sqrt{2(n-1)} -  u(y,\tau) \leq (1+\tfrac{\tilde \eta}2) \,   K_0 \, y^4 \, e^{-\tau} 
\end{equation}
holds for $\ell_1 \leq |y|  \leq \rho \, e^{\tau_1/4}$, provided $\tilde \eta \geq  C(M_1,K_0)\,  \ell_0^{-2}$. 

\smallskip
Recall our notation    $q:= u^2 - 2(n-1)$ and set $\bar K_0 := 2\sqrt{2(n-1)} \, K_0$, $\bar K_0^{\pm} := (1\pm \bar  \eta) \bar K_0$ and
$q_{\bar K_0^+}:= (\bu_{(1 + \tilde \eta) K_0})^2- 2(n-1)$, $q_{\bar K_0^-}:= (\bu_{(1-\tilde \eta) K_0})^2- 2(n-1)$. 
Then, \eqref{eqn-q34} and the peanut  asymptotics  imply that 
\begin{equation}		\label{eqn-q35}  q_{\bar K_0^+}(y,\tau) \leq q(y,\tau) \leq q_{\bar K_0^-} (y,\tau) 
\end{equation}
holds for $\ell_1 \leq |y|  \leq \rho \, e^{\tau_1/4}$,  provided that  $C(M_1, K_0)  \ell^{-2}_0 < \tilde \eta$, $\ell_0 \gg 1$  and  $\tau_0 \gg 1$. 

\smallskip 

On the other hand,  by Claim \ref{claim-peanut-toti},  for 
any $\eta >0$ small, if $\cQ^-_{\varepsilon, \bar K^+_0 +  \eta }, \cQ^+_{\varepsilon,\bar K^-_0 - \eta}$ denote  the sub and super solutions constructed in Proposition \ref{prop-sub-sup},  then we have 
\begin{equation}		\label{eq-qQ}
\bar{q}_{\bar K^+_0}(y,\tau_1) \ge \cQ^-_{\varepsilon, \bar K^+_0  +\eta}(y,\tau_1) \quad \mbox{and} \quad \bar{q}_{\bar K^-_0}(y,\tau_1) \le \cQ^+_{\varepsilon,\bar K^-_0 - \eta}(y,\tau_1)
\end{equation}
for $|y| \ge \ell_1$,  provided $\ell_1 > \sqrt{4M_1 \eta^{-1}}$. 

Finally, combining  \eqref{eqn-q35} and \eqref{eq-qQ} while  taking $\tilde \eta <  c(n, K_0) \eta$, for some constant $c(n, K_0)$ depending only on $K_0$
and  $n$,  we conclude the desired bound 
\begin{equation}
 \cQ^-_{\varepsilon,\bar K_0+\eta}(y,\tau_1) \le q(y,\tau_1) \le \cQ^+_{\varepsilon,\bar K_0-\eta}(y,\tau_1)
\end{equation}
holding for all $|y| \ge \ell_1$,  where $\eta \ge C(M_1,K_0) \, \ell_0^{-2}$ can be taken to be a fixed constant. 
\end{proof}

\smallskip

\begin{lemma}
\label{lemma-Linfty-q}
Assume that \eqref{eq-Linfty-qq} holds on $|y| \leq 2\rho \, e^{\tau/4}$, $\tau \in [\tau_1, \tilde \tau_1]$, for some $\tilde \tau_1 \in (\tau_1, \tau_2)$. 
Then, there exists a uniform constant $C_1$   that depends only on $n$,  $K_0$ and $M_1$ such that 
\begin{equation}		\label{eqn-linfty-q1} 
|q(y,\tau)| +  |q_y(y,\tau)| +  |q_{yy}(y,\tau)| \le C_1 \, \varepsilon \, (1 + |y|^4),
\end{equation}
 holds for all $|y| \le 4 \rho \,  e^{{\tau/4}}$ and all $\tau\in [\tau_1,\tilde \tau_1]$.
\end{lemma}

\begin{proof} We will first prove the $L^\infty$-bound $|q(y,\tau)|  \le C_1 \,  \varepsilon \, (1 + |y|^4)$, on 
$|y| \le 4 \rho \,  e^{{\tau/4}}$, $\tau \in [\tau_1, \tilde \tau_1]$. By Lemma \ref{lemma-asymp-2} the following asymptotics  hold
\begin{equation}		\label{eq-q-100}
q(y,\tau) = \varepsilon \, H_2(y) - \bar{K}_0 e^{-\tau}\, \hm_4(y) + o(\varepsilon),
\end{equation}
in the $C^0$-sense, for $|y| \le 4\ell_1$ and all $\tau\in [\tau_1,\tilde{\tau}_1]$. By employing standard derivative estimates for parabolic equations, one can show that \eqref{eq-q-100} implies 
the bound \eqref{eqn-linfty-q1}  on $|y| \leq 2\ell_1$,  for some $C_1$ depending only on $K_0$ and $n$. 

Let us then concentrate next on $|y| \geq \ell_1$. Fix $\eta= C(M_1,K_0) \ell_0^{-2}$ so that \eqref{eqn-ic} holds, according to Lemma \ref{lemma-initial-comp}. 
 Asymptotics \eqref{eq-q-100}, using that we can make  $|o(\varepsilon)| \le  \frac{\eta}{2}$, by ensuring $\tau_1$ is large enough, yield
\[
\varepsilon (1- \frac{\eta}{2}) \hm_2(\ell_1) - (\bar{K}_0 + \frac{\eta}{2})e^{-\tau}\hm_4(\ell_1) \le q(\ell_1,\tau) \le \varepsilon (1+ \frac{\eta}{2}) \hm_2(\ell_1) - (\bar{K}_0 - \frac{\eta}{2})e^{-\tau}\hm_4(\ell_1)  
\]
and thus, using the definition of  $\cQ^-_{\varepsilon(1-\frac \eta 2),\bar{K}_0+\frac \eta 2}, \cQ^+_{\varepsilon (1+\frac \eta 2),\bar{K}_0-\frac \eta 2}$,   we get 
\begin{equation}		\label{eq-bd-sat}
\cQ^-_{\varepsilon(1-\eta),\bar{K}_0+\eta}(\ell_1,\tau) \le q(\ell_1,\tau) \le \cQ^+_{\varepsilon (1+\eta),\bar{K}_0-\eta}(\ell_1,\tau), \qquad \tau \in [\tau_1,\tilde{\tau}_1].
\end{equation}
Having \eqref{eqn-ic}  and  \eqref{eq-bd-sat},  we can apply the comparison principle with  boundary $|y|=\ell_1$  and conclude that 
\begin{equation}		\label{eq-tau1-sat}
\cQ^-_{\varepsilon (1-\eta),\bar{K}_0+\eta}(y,\tau) \le q(y,\tau) \le \cQ^+_{\varepsilon (1+\eta),\bar{K}_0-\eta}(y,\tau),
\end{equation}
for $|y| \ge \ell_1$, and $\tau\in [\tau_1,\tilde{\tau}_1]$. 

On the other hand, by the construction of our barriers we have 
$\cQ^\pm_{\varepsilon, \bar K_0}(y, \tau)  = \big ( \varepsilon \,  H_2(y)  - \bar K_0 \, e^{-\tau} H_4 (y)  \big ) \, (1+o(1))$ on $\ell_1 \leq |y| \leq 6 \rho \, e^{\gamma \tau}$, provided that the constant $\rho >0$ is chosen 
sufficiently small  so that the region $\ell_1 \leq |y| \leq 6 \rho \, e^{\gamma \tau}$ is away from the tip (for example  we can take $\rho >0$ so that $8 \rho e^{\gamma \tau} \leq Y_0(\tau)$). 
Hence, \eqref{eq-tau1-sat} implies  the $L^\infty$ bound 
\[
|q(y, \tau) | \leq C_1    \, \varepsilon (1+ |y|^4 ), \qquad \mbox{on} \,\,  \ell_1 \leq |y| \leq 6\rho \, e^{\tau/4}
\]
where $C_1$ is a uniform constant depending only on $n, \bar K_0$. 

To pass from the $L^\infty$ bound on $q$ on $\ell_1 \leq |y| \leq 6\rho \, e^{\tau/4}$ to the $L^\infty$ bounds on $q_y$ and $ q_{yy}$ on $\ell_1 \leq |y| \leq  4\rho \, e^{\tau/4}$,
one uses standard derivative estimates  following  the proof of Lemma 6.2 in \cite{AV}. 
\end{proof}

We can now finish the  proof of Proposition \ref{prop-barriers}.

\begin{proof}[Proof of Proposition  \ref{prop-barriers}]

To finish the proof of the Proposition we need to remove the  a'priori assumption  that bound \eqref{eq-Linfty-qq} holds. 
Let  $\ell_1 \geq \ell_0$ be sufficiently large but fixed so that Lemma \ref{lemma-Linfty-q} and all our previous results hold. By part (i) of Lemma \ref{cor-prop-funnel}
(applied to  $4\rho$ instead of $\rho$), since $\varepsilon = 2\sqrt{2(n-1)} M_1 e^{-\tau_1}$, we have that 
\[
|q(y,\tau_1)| + |q_y(y,\tau_1)| + |q_{yy}(y,\tau_1)| \le \bar{C}_0\,  \varepsilon \, (1 + |y|^4),
\]
for $|y| \le 4 \rho e^{{\tau/4}}$, where $\bar{C}_0$ depends on $n$ and   $K_0$. 
With no loss of generality we may assume this constant $\bar C_0 = C_1$, where $C_1$ is 
the constant in   \eqref{eqn-linfty-q1}.  Let $\tilde{\tau}_1 \leq  \tau_2$ be the maximal time so that we have
\[
|q(y,\tau)| +  |q_y(y,\tau)| +  |q_{yy}(y,\tau)| \le 2C_1 \, \varepsilon \, (1 + |y|^4),
\]
for $|y| \le 2  \rho \, e^{\tau/4}$, $\tau \in [\tau_0, \tilde \tau_1]$. If $\tilde \tau_1=\tau_2$ we conclude that \eqref{eqn-linfty-q}  holds  with $C_0 = 2C_1$. 
Otherwise,  using the above bound in place of  \eqref{eq-Linfty-qq},
by Lemma \ref{lemma-Linfty-q} we get that 
\[
|q(y, \tilde \tau_1)| +  |q_y(y, \tilde \tau_1)| +  |q_{yy}(y,\tilde \tau_1)| \le  C_1 \,  (\varepsilon + e^{- \tilde \tau_1}) \, (1 + |y|^4)
\]
on $ |y| \leq 4 \rho \, e^{\tilde \tau_1/4}$ and therefore the estimate can be extended beyond  $\tilde \tau_1$, contradicting its maximality.  
We then conclude that \eqref{eqn-linfty-q} holds on $|y| \leq 2 \rho e^{\tau/4}$, $\tau \in [\tau_1, \tau_2]$, with $C_0 = 2 C_1$. 
 \end{proof}

We can now finish the proof of Theorem \ref{thm-main}.

\begin{proof}[Proof of Theorem \ref{thm-main}] Proposition \ref{prop-sphere} shows that in every small neighborhood of peanut solution we can find a perturbation whose mean curvature flow develops a {\em spherical singularity. }

\smallskip 
Previous results ensure that at the same time in any small neighborhood of the peanut solution we can find a perturbation so that Proposition \ref{prop-barriers} holds. Our goal is to show that in this case the flow develops a \emph{nondegenerate neckpinch singularity.} By Proposition \ref{prop-barriers} and Lemma \ref{lemma-asymp-2}, for all $|y| \le \ell_1$ we have
\[
q(y,\tau_2) = \varepsilon \, H_2(y) - Ke^{-\tau_2}\, H_4(y) + o(\varepsilon),
\]
where $\varepsilon = 2\sqrt{2(n-1)} M_1 \, e^{-\tau_1}$. This implies that {\em at time $\tau_2$}  in a  neighborhood of $ y = 0$ we can put a Shrinking Doughnut whose inner radius is $\sqrt{2(n-1)}$  around our solution. Note that this is possible since $u(0,\tau_2) = \sqrt{2(n-1)} - 2\varepsilon + o(\varepsilon)$, and hence in a neighborhood of $y = 0$ we have that $u(y,\tau_2) < \sqrt{2(n-1)} - \varepsilon$. The doughnut whose inner radius is $\sqrt{2(n-1)}$ becomes singular at time $T_a  = T_a(n) < \infty$. 

\smallskip 
On the other hand, our barrier estimate  \eqref{eqn-barriers-tau2}, shown in Proposition \ref{prop-barriers},  implies that at time $\tau_2$   we have enough room so that we can put spheres 
$S^\pm_R$  of a large radius $R$  inside of our solution, on both sides, for $y > 1$ and $y < -1$. Here we use the fact that  $\cQ^-_{(1-\bar{\eta})\varepsilon, (1+ \bar{\eta})\bar{K}_0}  \approx \varepsilon \, H_2 -  \bar K_0 \, e^{-\tau_2} |y|^4$.  We choose   $R$ sufficiently large   so that the extinction  time $T_R$  of $S^\pm_R$ is much bigger than $T_a$. Since $T_a$ depends only on $n$, we can choose $R$ that also depends only on $n$. The comparison principle then guarantees that  at the singular time of the flow,  we have a  local singularity that  disconnects the manifold into two pieces none of which disappears at the singular time.   This ensures the singularity model cannot  be a round sphere $\mathbb{S}^n$.

Note that the height of our barriers at $\tau_2$ is approximately $ \sigma_n / (2\sqrt {\bar K_0}) $, where $\sigma_n$ can be taken as large
as we wish  and  defines $\tau_2$ through  \eqref{eqn-tau2}. (c.f.  Remark \ref{rem-height} and equation \eqref{eqn-height}). Hence, the shape of our barriers guarantee that by choosing $\sigma_n$ 
sufficiently large   our solution at time $\tau_2$ encloses $S^\pm_R$. Moreover, $\sigma_n$ depends  only on  $n$ and $K_0$, since $R$ depends only on $n$.

\smallskip 

We next  remark that  the peanut solutions considered in this paper are mean convex, and hence for each of them all sufficiently close perturbations are mean convex, which is a property that is preserved by the mean curvature flow. By a result of White (\cite{Wh}) and Brakke (\cite{Br}) we have that any tangent flow  at singularities of a compact, embedded, mean convex mean curvature flow is a unit multiplicity smooth mean convex shrinker with polynomial volume growth. Thus, by \cite{CM} it is either a round cylinder $\mathbb{S}^{n-1}\times\mathbb{R}$ or a round sphere $\mathbb{S}^n$. The latter case is excluded by the fact that singularity disconnects the manifold into two pieces none of which disappears at the singular time.  

Recall also that by Lemma \ref{lem-convexity} we can always choose $\ell_0$ big enough so that all our perturbations of the peanut solution $\bar{u}(y,\tau)$ are convex for $|y| \ge 2\ell_0$. 
By using rotational symmetry, reflection symmetry of our initial data, precise asymptotic description of our solution at time $\tau_2$, and Sturmian theorem for decreasing the number of critical points along the flow it immediately follows that   starting from some time, all the way up to the singular time, the profile function has only one maximum and  the solution is therefore convex, or the only local minimum starting from some time on is at $x = 0$, and the singularity in this case occurs at $x = 0$. In the former case the tangent flow is $\mathbb{S}^n$, which we know can not happen. This implies the latter case actually occurs, and a tangent flow at the origin is the round cylinder.
In other words, we have a neckpinch singularity at the origin, and at the singular time the surface disconnects into two pieces none of which disappears at that time. Direct adaptation of  arguments in the proof of Theorem 4.1 in \cite{SS}  shows that the singularity has to be nondegenerate, in a geometric sense, i.e. that every blow up limit around the origin is the round cylinder, and the singularity is Type I.
\end{proof}

%%%%%%%%%%%%%%%%%%%%%%%%%%
%%%%%%%%%%%%%%%%%%%%%%%%%%%

\section{Blow ups of families of mean curvature flows}

Let $\bar{M}_t$ be the peanut solution as discussed above, and let $\bar{T}$ be its first singular time. Consider  its rescaled profile $\bar u(y, \tau_0)$, where $\tau = - \log (\bar T-t)$. 
Assume that  $\tau_0$  large and $\eta_0(y)$ are  as in the previous sections. We recall that $\eta_0(y)$ is the cut off defined in \eqref{eqn-def-eta0} and   $\tau_0 = - \log (\bar T- t_0)$ {\em for some $t_0 $ close to $\bar T$.}
Let ${\bf \epsilon}_k = (\epsilon_k^1, \epsilon_k^2)$ be any sequence converging to $(0,0)$ as $k\to \infty$ and define the profile functions 
\[
u_k(y, \tau_0) := \bu(y ,\tau_0) + \eta_0 \, \big ( \frac{y}{\ell_0} \big ) \, (\epsilon_k^1 + \epsilon_k^2\, H_2).
\]
Theorem \ref{thm-main} guarantees that we can choose $ \epsilon_k = (\epsilon_k^1, \epsilon_k^2) \to (0,0)$, so that  the {\em unrescaled  mean curvature flow  $M_t^k := $},  with profile  
$ U_k(x, t) =   \sqrt{\bar T-t} \; u \Bigl( \frac x{\sqrt{\bar T-t}},\log \frac 1{\bar T-t}\Bigr) $
 develops spherical singularity at its first singular time  $T_k < \infty$.  

Having that the $\lim_{k\to \infty} M^k_{t_0} = \bar M_{t_0}$, and using the lower semi-continuity of the singular time of mean curvature flow in terms of its    initial data, on one hand, and the fact that all perturbations $M^k_{t_0}$  can be placed in the interior of a sphere of uniform radius $R$ that does not depend on $k$, by the comparison principle we see that 
\begin{equation}		\label{eq-sing-time}
\bar{T}/2 \le  T_k \le C_0, \qquad \mbox{for all} \,\,\,\,k \,\,\,\, \mbox{big enough}.
\end{equation}
 Here, $C_0$ is the extinction time of a sphere of radius $R$. The  proof of Proposition \ref{prop-sphere} shows that  for each $k$ big, there exists the first time $t_k^1 < \min\{T_k,\hat{T}\}$ at which the flow $M^k_t$ becomes convex. In this section we prove Theorem \ref{thm-families} by showing the following result.

\begin{theorem}
Appropriately rescaled sequence of any sequence of solutions whose initial data converge to the peanut solution, and all of which develop spherical singularities, converges to the Ancient oval solution constructed in \cite{HH, Wh}.  
\end{theorem}

\begin{proof}
Let $M_t^k$ be a sequence of mean curvature flows, as discussed above. Let us define  by $a_k(t) := \max_{x\in M_t^k} |x_{n+1}|$ the major radius, and by $b_k(t) := \max_{x\in M^k_t}\Big(\sum_{i=1}^n x_i^2\Big)^{1/2}$ the minor radius of $M_t^k$.  Let $t_k^1 < \min\{T_k,\bar{T}\}$ be {\em the first time at which solution $M^k_t$ becomes convex}. Then we have the following claim.

\begin{claim}
\label{claim-akbk}
We have that the $\lim_{k\to\infty} t_k^1 = \bar{T}$ and the $\lim_{k\to\infty} \frac{a_k(t_k^1)}{b_k(t_k^1)} = \infty$.
\end{claim}

\begin{proof}
Let $\tau_k^1 := -\ln (\bar{T} - t_k^1)$. Note that by the scaling invariance property of the quotient we have that $\frac{a_k(\tau_k^1)}{b_k(\tau_k^1)}  = \frac{a_k(t_k^1)}{b_k(t_k^1)}$, where the first quotient is for the rescaled flow and the latter quotient is for the unrescaled mean curvature flow. Furthermore, by Lemma \ref{lem:alpha}, by Lemma \ref{lem-convexity}, and  by Lemma \ref{cor-prop-funnel} (which guarantees the assumptions in the statement of Proposition \ref{prop-inout} are satisfied all the way to time $\tau_k^1$) we conclude that for  $|y| \ge 2\ell_0$ (where $\ell_0$ is taken as in Lemma \ref{lem:alpha}) we have  
\begin{equation}		\label{eq-comparison}
\bu(y,\tau_k^1 - \alpha_k) \le u_k(y,\tau_k^1) \le \bu(y,\tau_k^1+\alpha_k)
\end{equation} 
where   Lemma \ref{lem:alpha} guarantees that $\alpha_k$ satisfies  $0 < \alpha_k(\tau) \le \ln  2$. This implies that
\begin{equation}		\label{eq-ak}
c_0\, e^{\tau_k^1/4} \le a_k(\tau_k^1) \le C_0\, e^{\tau_k^1/4},
\end{equation}
for uniform positive constants $c_0, C_0$, since the diameter of the peanut  solution with profile function $\bu(y,\tau)$ at time $\tau$ is of the order $e^{{\tau/4}}$. Estimate \eqref{eq-comparison} and the asymptotics of peanut solution also  imply imply that
\begin{equation}		\label{eq-bk}
\frac{\sqrt{2(n-1}}{2} \le b_k(\tau_k^1) \le 2\sqrt{2(n-1)}.
\end{equation}

\smallskip

We {\em claim } that the $\lim_{k\to\infty} \tau_k^1 = \infty$, thus implying that the $\lim_{k\to\infty} t_k^1 = \bar{T}$. To show this we argue by contradiction. Assume that the sequence $\{\tau_k^1\}$ is uniformly bounded in $k$, implying that along a subsequence we have the $\lim_{k\to\infty} \tau_k^1 = \bar{\tau}$, and thus the $\lim_{k\to\infty} t_k^1 = \bar{t} < \bar{T}$. By the continuous dependence of the mean curvature flow on the initial data, we would then have that the $\lim_{k\to \infty} A^k_{ij}(t_k^1) = \bar{A}_{ij}(\bar{t})$, where $A^k_{ij}$ and $\bar{A}_{ij}$  are the second fundamental forms of $M_t^k$ and $\bar{M}_t$, respectively. Since $A^k_{ij}(t_k^1) > 0$ for all $k$, we would have that $\bar{A}_{ij}(\bar{t}) \ge 0$. This leads to contradiction, because $\bar{M}_t$ is the peanut solution that never becomes convex before it extincts to a point. Hence, we indeed have that the $\lim_{k\to\infty} t_k^1 = \bar{T}$.

Estimates \eqref{eq-ak} and \eqref{eq-bk} imply 
\begin{equation}		\label{eq-ratio-t1}
\lim_{k\to\infty} \frac{a_k(t_k^1)}{b_k(t_k^1)} = \infty,
\end{equation}
concluding the proof of the Claim. 
\end{proof}

Next, we claim that for every solution $M_t^k$, there exists the first time $t_k$ satisfying $T_k > t_k \ge t_k^1$  such that 
\begin{equation}		\label{eq-ratio}
\frac{a_k(t_k)}{b_k(t_k)} = 2.
\end{equation}
Indeed, since every solution $M_t^k$ develops a spherical singularity at time $T_k < \infty$, we have that the $\lim_{t\to T_k} \frac{a_k(t_k)}{b_k(t_k)} = 1$.  On the other hand, by Claim \ref{claim-akbk} we have the $\lim_{k\to\infty}\frac{a_k(t_k^1)}{b_k(t_k^1)} = \infty$.
We conclude now that for every $k$ there exists a $t_k \in (t_k^1, T_k)$ so that \eqref{eq-ratio} holds. 

Denote by $\hat{M}^k:= M^k_{t_k^1}$. By our construction, $\hat{M}^k$ is very close to a round cylinder on compact sets. Recall that $T_k$ is the singular time of $M^k_t$ and that $T_k$ satisfies  \eqref{eq-sing-time}. Choose the scaling factor $r_k$ so that the mean curvature flow $\bar{N}^k_t$, with initial data $N^k := r_k\, \hat{M}^k$ becomes singular at time $T_{\max} = 1$. This implies $N_k  \in Z_s$, where $Z_s$ is the set the authors defined in \cite{ADS2}, i.e.
\[
Z_s = \{ C \,\,\,\, \mbox{is a closed, convex set} \,\,\,|\,\,\, C = -C, \,\,\, T_{\max} = 1\}.
\]
By our construction  we have that the $\lim_{k\to\infty} \mc H(N^k) = \mc{H}(\mathbb{S}^{n-1}\times\mathbb{R})$, where $\mc{H}$ is the Huisken's energy. By the results in \cite{ADS2} we have that orbits $N^k(\tau)$ under the RMCF (rescaled mean curvature flow) of $N^k$, with rescaling corresponding to $T_{\max} = 1$, 
stay in a compact subset of $Z_s$. Choose $\tau_k$ to be so that $N^k(\tau_k)$ has the property that $\frac{a_k(\tau_k)}{b_k(\tau_k)} = 2$. Such a $\tau_k$ exists, since this quotient is scaling invariant, and since we have \eqref{eq-ratio}.

We claim that the $\lim_{k\to\infty} \tau_k = \infty$. Indeed, if it were uniformly bounded, we would have along a subsequence that the $\lim_{k\to\infty} \tau_k = \tau_{\infty} < \infty$. By \eqref{eq-sing-time} the sequence $r_k$ (that we used in the definition of $N^k$ above) is uniformly bounded in $k$ from above and below.  Consider then the sequence of solutions $\bar{N}^k_{\tau} := N^k(\tau+\tau_k)$.  As $k\to\infty$, the  hypersurfaces $\hat{M}^k$ converge to a round cylinder, and hence the same holds for hypersurfaces $N^k$. Call the limiting round cylinder $N_{\infty}$. Furthermore, by the compactness of set $Z_s$ we would have that the sequence of flows $N^k(\tau)$ converges to the limiting flow $N_{\infty}(\tau)$, uniformly on compact sets. That would mean that $N(\tau_{\infty}) = N_{\infty}$,  is still the cylinder,  which would certainly violate \eqref{eq-ratio}, due to scaling invariant property of the ratio. 

Finally, due to results in \cite{ADS2} we have that there exists a subsequence $\bar{N}^k(\tau)$ that converges to an ancient solution that must be an Ancient oval due to its uniqueness (\cite{ADS}, \cite{CHHW22}).

\end{proof}

%%%%%%%%%%%%%%%%%%%%%%%%%%%%%%%%%%%%%
%%%%%%%%%%%%%%%%%%%%%%%%%%%%%%%%%%%%

\appendix
\section{Gaussian weighted Hilbert spaces and the drift Laplacian}
\label{sec:appendix hilbert spaces}
\subsection{The Hilbert space $L^2(\R, e^{-y^2/4}dy)$}
 $\cL$ is the \emph{drift Laplacian}
\begin{equation}		\label{eq:L}
\cL v := v_{yy } - \frac {y } 2\, v_y + v .  
\end{equation}
The domain of $\cL$ is
\[
\mathrm{dom}(\cL):= \{ v \in \hilb \;|\; (1+|y|)v_y\in\hilb, (1+y^2) v_{yy} \in \hilb \}.  
\]
Since we only consider even functions, the spectrum of $\cL$ is given by the sequence of simple eigenvalues
\[
\lambda_k = 1-\frac k2 = -k\gamma, \qquad k=0,2,4,6, \dots
\]
and the corresponding eigenfunctions are Hermite polynomials $\hhm_k$.  We use the following normalizations: $\hhm_k$ is the Hermite polynomial normalized so that its leading coefficient is $1$, i.e.  
\begin{equation}		\label{eq:Hermite polynomials explicit form}
\hm_k(y) =y^k - \frac{k(k-1)}{1!}y^{k-2} + \frac{k(k-1)(k-2)(k-3)}{2!} y^{k-4}-\cdots \; .  
\end{equation}
% We write
% \[
% \hat\hm_k(y) = \frac{\hhm_k(y)}{\; \|\hhm_k(y)\| }
% \]
% for the multiple of $\hhm_k$ that has norm $1$ in $\hilb$.  \textcolor{BrickRed}{Check if we ever use this notation now}

\section{Constructing an $m$-peanut}		\label{sec:remember the peanut}
Let $K_0>0$ be fixed.  We choose our initial surface by perturbing the superellipsoid
\begin{equation}		\label{eq:superellipsoid defined}
U_{\rm out} (y,\tau_0) \stackrel{\rm def}= \sqrt{2(n-1) - K_0 y^me^{-m\gamma\tau_0}}
\end{equation}
both in the parabolic region $|y|\lesssim \rho e^{\gamma\tau_0}$, and the tip region where $u=O(e^{-\gamma\tau_0})$.   

In the tip region we replace the surface with a rescaled copy of the bowl soliton.  This modification will allow us to verify the monotonicity of the peanut solution in the region $|y|\geq \ell_{\rm int}$, $\tau\geq \tau_0$.  Monotonicity of the peanut solution lets us use this solution as barriers which control the perturbations of the peanut solutions that are the subject of this paper.  

\begin{figure}[h]\centering
 \includegraphics[width=\textwidth]{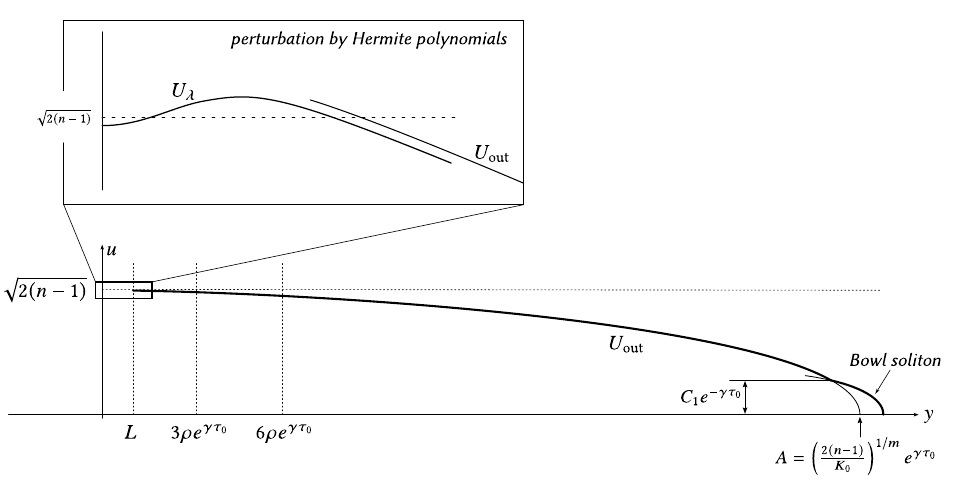}
 \caption{Construction of the initial peanut }
\end{figure}

In the parabolic region we perturb the superellipsoid by pasting in a linear combination of Hermite polynomials.  This perturbation must contain enough parameters ($\lambda_0, \dots, \lambda_{m-2}$) to guarantee that at least one choice of the parameters leads to a peanut solution.  Following \cite{AV} we consider the following $m/2$-parameter family of perturbations of the cylinder:
\begin{equation}		\label{eq:U lambda defined}
U_\lambda(y, \tau) \stackrel{\rm def}= \sqrt{2(n-1)} + \sum_{j=0}^{\frac m2-1} \lambda_{2j} \hhm_{2j}(y)
- \frac{K_0}{2\sqrt{2(n-1)}}e^{-m\gamma\tau_0}\hhm_m(y).  
\end{equation}
with $\lambda_0, \lambda_2, \dots, \lambda_{m-2}\in\R$.  

The initial condition that leads to a peanut solution is given by gluing together the parametrized family $U_\lambda(y, \tau_0)$ and the modified superellipsoid  $\hat U_{\rm out}$.  Thus we set
\begin{equation}		\label{eq:u lambda defined} 
u_\lambda(y, \tau_0) 
	= \zeta\Bigl(\frac{y}{\rho e^{\gamma\tau_0}}\Bigr) \; U_\lambda(y, \tau_0) 
	+ \Bigl\{ 1- \zeta\Bigl(\frac{y}{\rho e^{\gamma\tau_0}}\Bigr)\Bigr\} \hat U_{\rm out}(y, \tau_0).  
\end{equation}
Here $\zeta:\R\to\R$ is a smooth even cut-off function with $\zeta(s)=1$ for $|s|\leq 1$ and $\zeta(s)=0$ for $|s|\geq 2$.

The function $u_\lambda(\cdot, \tau_0)$ defines a  hypersurface, whose evolution by rescaled MCF~\eqref{eq:u} is given by a function $u_\lambda(y, \tau)$ that is defined for
\[
|y|\leq y_{\max,\lambda}(\tau), \qquad \tau_0\leq \tau < \tau_{\max, \lambda}.  
\]
The unstable component at time $\tau$ of the solution is defined to be
\begin{equation}		\label{eq:Psi u def}
\Psi_\lambda^u(y, \tau) := \pi^u \Bigl[ \zeta\Bigl(\frac{y}{\rho e^{\gamma\tau}}\Bigr) \Bigl\{u_\lambda(y, \tau)-\sqrt{2(n-1)}\Bigr\}\Bigr]
\end{equation}
In the exit lemma of \cite{AV} it was shown that if $M_0>0$ is appropriately chosen, and $\tau_0$ is sufficiently large, then for each $\lambda_0, \lambda_2, \dots, \lambda_{m-2}\in \R$ one either has $\tau_{\max, \lambda} =\infty$ and
\begin{equation}		\label{eq:no exit}
\|\Psi_\lambda^u(\cdot, \tau)\|  <  M_0e^{-2m\gamma\tau}\text{ for all }\tau\in[\tau_0, \infty),
\end{equation}
or else there is a first $\tau_1=\tau_1(\lambda)\geq \tau_0$ such that
\begin{equation}		\label{eq:exit condition}
\|\Psi_\lambda^u(\cdot, \tau_1)\|  = M_0e^{-2m\gamma\tau_1}.  
\end{equation}
In the first case, where \eqref{eq:no exit} holds, the unrescaled solution forms a singularity as $t\nearrow T$ whose parabolic blow-up is the cylinder.  

In the second case, i.e.~when there is a first $\tau_1\geq \tau_0$ at which \eqref{eq:exit condition} holds, one has
\begin{equation}
\|\Psi_\lambda^u(\cdot, \tau)\|  <  M_0e^{-2m\gamma\tau}\text{ for all }\tau\in[\tau_0, \tau_1).  
\end{equation}
and
\begin{equation}		\label{eq:exit strict}
\frac{d}{d\tau}\Bigl(e^{2m\gamma\tau}\|\Psi_\lambda^u(\cdot, \tau)\|\Bigr)\bigg|_{\tau=\tau_1} >0.  
\end{equation}
One can verify that the initial condition satisfies
\[
\|\Psi_\lambda^u(\cdot, \tau_0)\|^2
\geq \sum_0^{\frac m2} \lambda_{2j}^2 \|\hhm_{2j}\|^2 - O\bigl(e^{-c \rho^2 e^{2\gamma\tau_0}}\bigr)
\]
for some small constant $c>0$.  It follows from a shooting argument (see \cite{AV}) that one can choose
\begin{equation}		\label{eq:good lambdas}
\bar\lambda_0, \bar\lambda_2, \dots, \bar\lambda_{m-2}\in\R
\end{equation}
with
\begin{equation}		\label{eq:good lambdas bound}
\sqrt{ \bar \lambda_{0}^2 \|\hhm_{0}\|^2
  +  \bar \lambda_{2}^2 \|\hhm_{2}\|^2+\cdots
  + \bar \lambda_{m}^2 \|\hhm_{m}\|^2}
\leq 2M_0 e^{-2m\gamma\tau_0}
\end{equation}
so that the solution starting from $u_{\bar\lambda}(\cdot, \tau_0)$ exists for all $\tau\geq\tau_0$ and satisfies \eqref{eq:no exit}.

The shooting argument in \cite{AV} is robust with respect to small perturbations of the family of initial data \eqref{eq:u lambda defined}.  In particular, one can replace \eqref{eq:U lambda defined} by
\[
\tilde U_\lambda(y, \tau_0) = \sqrt{2(n-1)} -  e^{-m\gamma\tau_0}\hhm_m(y) + \sum_{j=0}^{m-1}c_j\hhm_j(y) + f(y)
\]
where $f$ can be any sufficiently small smooth function with support in $|y|\leq \rho e^{\gamma\tau_0}$.  This implies that there is an infinite dimensional family (parametrized by the function $f$) of solutions that have the behavior \eqref{eq:um2}, \eqref{eq:um2-derivs}, \eqref{eq:u-outer}.  It was suggested in \cite{AV} and recently proved in \cite{AOSUN} that the set of initial data $\bu(y, \tau_0)$ that lead to solutions satisfying \eqref{eq:um2}, \eqref{eq:um2-derivs}, \eqref{eq:u-outer} is a submanifold of codimension $m$.

\section{Proof of monotonicity}
\label{sec:proof of monotonicity}

Since $\bu_\tau$ satisfies a linear parabolic equation, obtained by differentiating the equation \eqref{eq:u} for $\bu$ with respect to $\tau$, we can use the maximum principle.  At first sight it looks like this approach runs into difficulties because the equation for $\bu_\tau$ degenerates at the tip, i.e.  ~at $y=y_{\max}(\tau)$.  However, we can avoid this issue by considering the normal velocity
\[
\mathscr V = H+\frac 12 X\cdot N,
\]
which is related to $\bu_\tau$ by
\[
\mathscr V =\frac{\bu_\tau}{\sqrt{1+\bu_y^2}},
\]
and which satisfies a nondegenerate parabolic equation on the surface, namely
\[
\partial_t^\perp \mathscr V = \Delta\mathscr V - \tfrac 12 \nabla_{X^\top}\mathscr V + \bigl(|A|^2 + \tfrac 12\bigr)\mathscr V.  
\]
We shall verify that $\mathscr V\geq 0$ initially, i.e.~for $y\geq \ell_{\rm int}$, $\tau=\tau_0$, and on the boundary $y=\ell_{\rm int}, \tau\geq \tau_0$.  
The maximum principle then implies that $\mathscr V\geq0$ whenever $y\geq \ell_{\rm int}, \tau\geq \tau_0$.  Since $\mathscr V$ and $\bu_\tau$ have the same sign, this shows that $\bu_\tau\geq 0$ whenever  $y\geq \ell_{\rm int}, \tau\geq \tau_0$.  And, for the same reason, we can verify $\mathscr V\geq 0$ by showing that $\bu_\tau\geq 0$.  

The fact that $\bu_\tau>0$ at $y=\ell_{\rm int}$ for all $\tau\geq \tau_0$  follows directly from the construction of the peanut.  Indeed, the asymptotic expansions of the peanut solution from \cite {AV} imply \eqref{eq:utau positive} for all $y\in [\frac 12 \ell_{\rm int}, 2\ell_{\rm int}]$ and all $\tau\geq \tau_0$.  

To complete the proof we therefore have to verify $\bu_\tau\geq 0$ at $\tau=\tau_0$ for all $y\geq \ell_{\rm int}$.  Because of the way we have defined $\bu(y, \tau_0)$, we have to split the region $y\geq \ell_{\rm int}$ into three parts.  First we consider the near parabolic region where $\ell_{\rm int}\leq y\leq 3\rho e^{\tau_0/4}$, then the gluing region $3\rho e^{\tau_0/4}\leq y\leq 6\rho e^{\tau_0/4}$, and finally we deal with the remaining region in which $y\geq 6\rho e^{\tau_0/4}$.

\subsection{The region $\ell_{\rm int}\leq y\leq 3\rho e^{\gamma\tau_0}$, $\tau=\tau_0$}
\label{sec:near region}

The initial value for the peanut is defined by gluing together the superellipsoid $U_{\rm out}$ and the parametrized perturbation $U_\lambda$ of the formal solution.  See~\eqref{eq:u lambda defined}.  We estimate these two approximate solutions when $y\in[\ell_{\rm int}, 6\rho e^{-\tau_0/4}]$.  In this region the Hermite polynomials $\hhm_j(y)$ are comparable with the monomial $y^j$, i.e.  
\[
cy^j \leq \hhm_j(y) \leq Cy^j,
\]
where this inequality may be differentiated.  Since the coefficients $\lambda_j$ in the definition of $U_\lambda$ are bounded by \eqref{eq:good lambdas bound}, we have $|\lambda_j|\lesssim e^{-2m\gamma\tau_0}$, and hence
\[
\left|\lambda_0\hhm_0(y) + \lambda_1\hhm_2(y) + \cdots + \lambda_{\frac m2 -1}\hhm_{m-2}(y)\right| \lesssim e^{-2m\gamma\tau_0}y^{m-2}.  
\]
By definition $\hhm_m(y) = y^m +O(y^{m-2})$, so
\[
e^{-m\gamma\tau_0} \hhm_m(y) =e^{-m\gamma\tau_0}y^m + O\bigl(e^{-m\gamma\tau_0}y^{m-2}\bigr).
\]
Abbreviating 
\[
K_1 = \frac{K_0}{2\sqrt{2(n-1)}},
\]
we therefore have
\begin{align*}
U_\lambda(y, \tau_0) 
	&= \sqrt{2(n-1)} + \sum_{j=0}^{\frac{m}{2}-1}\lambda_{j}\hhm_{2j}(y) - K_1 e^{-m\gamma\tau_0}\hhm_m(y)  \\
	&= \sqrt{2(n-1)} -  K_1e^{-m\gamma\tau_0}y^m + O\bigl(e^{-2m\gamma\tau_0}y^{m-2} \bigr).  
\end{align*}
For the superellipsoid we have
\begin{align*}
  U_{\rm out}(y, \tau_0)
  	&= \sqrt{2(n-1) - K_0 e^{-m\gamma\tau_0}y^m}\\
 	&= \sqrt{2(n-1)} - K_1 e^{-m\gamma\tau_0}y^m + O\bigl(e^{-2m\gamma\tau_0}y^{2m}\bigr)
\end{align*}
Subtract these to get
\begin{equation}		\label{eq:Uout minus Ulambda}
U_{\rm out}(y, \tau_0) - U_\lambda(y, \tau_0) = O\bigl(e^{-2m\gamma\tau_0}y^{m-2} + e^{-2m\gamma\tau_0}y^{2m}\bigr).  
\end{equation}
The initial profile of the peanut is then given by
\begin{equation}		\label{eq:ulambda defined again}
\bu(y, \tau_0)=
U_\lambda(y, \tau_0)
+ \Bigl\{1-\zeta\Bigl(\frac{y}{6\rho e^{\gamma\tau_0}}\Bigr) \Bigr\} \bigl(U_{\rm out}(y, \tau_0) - U_\lambda(y, \tau_0)\bigr)
\end{equation}
To verify  $\bu_\tau(y, \tau_0) > 0$ we linearize equation~\eqref{eq:u} around the cylinder radius $\sqrt{2(n-1)}$,
\begin{equation}		\label{eq:Gu def}
\bu_\tau= \cG[\bu]
		: = \cL[\bu-\sqrt{2(n-1)}]
			- \frac{(\bu-\sqrt{2(n-1)})^2}{2\bu}
			- \frac{\bu_y^2\bu_{yy}}{(1+\bu_y^2)}.  
\end{equation}
Recall that $\cL$ is the drift Laplacian \eqref{eq:L}.  

One can verify that $\bu_{yy}(y, \tau_0)<0$ for $y\geq \ell_{\rm int}$ (if $\ell_{\rm int}$ is large enough), so we have
\begin{equation}		\label{eq:Gbu lower bound}
                \cG[\bu] > \cL[\bu-\sqrt{2(n-1)}] + O\bigl((\bu-\sqrt{2(n-1)})^2\bigr) .  
\end{equation}
When $\ell_{\rm int}\leq y\leq 3\rho e^{\gamma\tau_0}$ the cut-off function $\zeta$ is $1$, so that $\bu = U_\lambda$.  Keeping in mind that $\cL\hhm_{2j} = (1-j)\hhm_{2j}$, we then get
\[
\cL[\bu] = \sum_{j=0}^{\frac{m}{2}-1}   \lambda_j (1-j)  \hhm_{2j}(y)
-K_1 \bigl(1-\frac{m}{2}\bigr)e^{-m\gamma\tau_0}\hhm_m.
\]
The first term is bounded by
\[
\left|\sum_{j=0}^{\frac{m}{2}-1}\lambda_j \bigl(1-j\bigr) \hhm_{2j}(y)\right|
\lesssim e^{-2m\gamma\tau_0}y^{m-2}.  
\]
We can estimate the last term by using $\hhm_m(y) = y^m +O(y^{m-2})$, which gives us
\[
-K_1 \Bigl(1-\frac{m}{2}\Bigr)e^{-m\gamma\tau_0}\hhm_m(y)
	= \Bigl(\frac{m}{2}-1\Bigr)K_1 e^{-m\gamma\tau_0}y^m + O\bigl(e^{-m\gamma\tau_0}y^{m-2}\bigr)
\]
Therefore, using $y\geq \ell_{\rm int}$ and thus $y^{m-2}\leq \ell_{\rm int}^{-2}y^m$,
\begin{align}		\label{eq:1600}
  \cL[\bu-\sqrt{2(n-1)}]
  	& > \Bigl(\frac{m}{2}-1\Bigr)K_1 e^{-m\gamma\tau_0}y^m + O\bigl(e^{-m\gamma\tau_0}y^{m-2}\bigr)\\
	& = \Bigl(\bigl(\tfrac m2 -1\bigr)K_1 - O(\ell_{\rm int}^{-2})\Bigr) e^{-m\gamma\tau_0}y^m .  \notag
\end{align}
We also have
\[
\bu - \sqrt{2(n-1)} = O\bigl( e^{-m\gamma\tau_0}y^m \bigr).  
\]
Continuing from \eqref{eq:Gbu lower bound}, and using $y\lesssim \rho e^{\gamma\tau_0}$, we therefore find
\begin{align*}
  \cG[\bu]
	&> \Bigl(\bigl(\tfrac m2 -1\bigr)K_1 - O(\ell_{\rm int}^{-2})\Bigr) e^{-m\gamma\tau_0}y^m
        + O\bigl(e^{-2m\gamma\tau_0}y^{2m}\bigr)\\
  	&> \Bigl(\bigl(\tfrac m2 -1\bigr)K_1 - O(\ell_{\rm int}^{-2})- O( \rho^m)\Bigr) e^{-m\gamma\tau_0}y^m.  
\end{align*}
If $\ell_{\rm int}$ is sufficiently large, and $\rho$ sufficiently small, then we have shown that $\cG[\bu]>0$ for $\ell_{\rm int}\leq y\leq 3\rho e^{\gamma\tau_0}$.  

\subsection{The region $3\rho e^{\gamma\tau_0}\leq y\leq 6\rho e^{\gamma\tau_0}$, $\tau=\tau_0$}
\label{sec:gluing region}

In the gluing region we have to take the cutoff functions into account when we estimate $\cG[\bu]$.  Considering \eqref{eq:Gu def} one finds that the two last terms on the right again satisfy
\[
\Bigl(\bu-\sqrt{2(n-1)}\Bigr)^2 = O\bigl(e^{-2m\gamma\tau_0}y^{2m}\bigr)
\text { and }
-\frac{\bu_y^2\bu_{yy}}{1+\bu_y^2} >0.  
\]
so that, in view of \eqref{eq:Gbu lower bound},
\begin{equation}		\label{eq:Gu in between}
\cG[\bu] > \cL[\bu-\sqrt{2(n-1)}] + O\bigl(e^{-2m\gamma\tau_0}y^{2m}\bigr).  
\end{equation}
We estimate the first, linear, term $\cL[\bu-\sqrt{2(n-1)}]$ as in \eqref{eq:1600}, which generates the following extra terms coming from the cutoff function:
\begin{multline*}
\cL\bigl[(1-\zeta)(U_{\rm out}-U_\lambda)\bigr]  \\
	=\cL[1-\zeta](U_{\rm out}-U_\lambda) +(1-\zeta)\bigl(\partial_y^2 - \frac y2 \partial_y\bigr)(U_{\rm out}-U_\lambda) - \zeta_y(U_{\rm out}-U_\lambda)_y\;.  
\end{multline*}
To bound these terms for $3\rho e^{\gamma\tau_0}\leq y \leq 6\rho e^{\gamma\tau_0}$ we use \eqref{eq:Uout minus Ulambda}, i.e.  $U_{\rm out} - U_\lambda = O\bigl(e^{- 2m\gamma\tau_0}y^{m-2} + e^{-2m\gamma\tau_0}y^{2m}\bigr)$.  We also use $\rho e^{\gamma\tau_0}>1$, which implies $e^{-\gamma\tau_0} < \rho$.  This leads to
\begin{align*}
  \cL[1-\zeta] &= O(1) \\
  \zeta_y&= O(e^{-\gamma\tau_0}) = O(\rho) \\
  U_{\rm out}-U_\lambda &= O\bigl(e^{-(m+2)\gamma\tau_0}\rho^{m-2} + \rho^{2m}\bigr) = O(\rho^{2m})\\
  \partial_y(U_{\rm out}-U_\lambda )&= O\bigl(e^{-(m+3)\gamma\tau_0}\rho^{m-3} + e^{-\gamma\tau_0}\rho^{2m-1}\bigr) = O(\rho^{2m})\\
  \bigl(\partial_y^2 - \frac y2 \partial_y\bigr)(U_{\rm out}-U_\lambda) &= O\bigl(e^{-(m+2)\gamma\tau_0}\rho^{m-2} + \rho^{2m}\bigr) = O(\rho^{2m})
\end{align*}
and thus
\[
\cL\bigl[(1-\zeta)(U_{\rm out}-U_\lambda)\bigr] = O(\rho^{2m}).
\]
Hence, using $y=O(\rho e^{\gamma\tau_0})$ and $e^{-m\gamma\tau_0}y^m = O(\rho^m)$,
\begin{align*}
  \cG[\bu]
  > \Bigl(\bigl(\tfrac m2 -1\bigr)K_1 - O(\ell_{\rm int}^{-2})\Bigr) \rho^m    
    + O(\rho^{2m})
  > \Bigl(\bigl(\tfrac m2 -1\bigr)K_1 - O\bigl(\ell_{\rm int}^{-2}+  \rho^m\bigr) \Bigr) \rho^m .  
\end{align*}
We see that if $\rho$ is small enough, $\ell_{\rm int}$ large enough, and $\tau_0$ large enough, then $\cG[\bu]>0$ in the region $3\rho e^{\gamma\tau_0} \leq y\leq 6\rho e^{\gamma\tau_0}$.

\subsection{The region $y\geq 6\rho e^{\gamma\tau_0}$, $\tau=\tau_0$}
\label{sec:far region}
We represent surfaces  as graphs where $u$ is the independent variable.  In this case $y=y(u, \tau)$ represents a solution to RMCF if
\begin{equation}
	y_\tau =\cF[y] := \frac{y_{uu}}{1+y_u^2} + \Bigl(\frac{n-1}{u}-\frac{u}{2}\Bigr)y_u + \frac 12 y
\end{equation}
\begin{prop}		\label{prop-AY-subsol}
The surface given by \eqref{eq:u lambda defined} satisfies $\cF[y]>0$ in the region
\[
	y\geq 6\rho e^{\gamma\tau_0}, \qquad u\geq C_1e^{-\gamma\tau_0}
\]
provided $C_1$  and $\tau_0$ are sufficiently large constants.  
\end{prop}

\begin{proof}
On the surface given by \eqref{eq:superellipsoid defined} we then have
\[
y= A Y(u),
\]
where
\[
Y(u) := \left(1-\frac{u^2}{2(n-1)}\right)^{1/m}
\qquad
A = \left(\frac{2(n-1)}{K_0}\right)^{1/m}e^{\gamma\tau_0}.
\]
The function $Y(u)$ satisfies
\[
\Bigl(\frac{n-1}{u}-\frac{u}{2}\Bigr)Y'(u) + \frac 1m Y(u) = 0
\]
so that
\[
\cF[A Y] = A \left\{\frac{Y_{uu}}{1+A^2Y_u^2} + \Bigl(\frac 12 - \frac 1m\Bigr)Y\right\}
> A \left\{\frac{Y_{uu}}{A^2Y_u^2} + \Bigl(\frac 12 - \frac 1m\Bigr)Y\right\}
\]
because $Y_{uu}<0$.  
We see that $\cF[A Y] > 0$ holds if
\begin{equation}		\label{eq-FetaY-positive}
\frac{-Y_{uu}}{A^2Y_u^2} < \Bigl(\frac 12 - \frac 1m\Bigr)Y\text{ i.e.  if }
A^2 > \Big(\frac{2m}{m-2}\Big) \frac{-Y_{uu}}{YY_u^2}.  
\end{equation}
For any fixed $u\in (0, \sqrt{2(n-1)})$ this holds if $A$ is large enough.  When $u=o(1)$ we have
\begin{equation}		\label{eq:Y for small u}
Y = 1 - \frac{u^2}{2m(n-1)} + O(u^4),		\quad
Y_u = -\frac{u}{m(n-1)}+O(u^3),			\quad
Y_{uu} = \frac{-1}{m(n-1)} + O(u^2).  
\end{equation}
Thus \eqref{eq-FetaY-positive} holds if $u\geq C/A$ for some constant $C$; in view of the definition of $A$ this means that \eqref{eq-FetaY-positive} holds if $u\geq C_1e^{-\gamma\tau_0}$ for some constant $C_1$.  

At the other end of the $u$ interval, where $u=\sqrt{2(n-1)} -o(1)$ we have
\begin{align*}
Y &= (C+o(1)) \bigl(\sqrt{2(n-1)}-u\bigr)^{1/m}\\
Y_u &= (C+o(1))\bigl(\sqrt{2(n-1)}-u\bigr)^{1/m-1}\\
Y_{uu} &= -(C+o(1))\bigl(\sqrt{2(n-1)}-u\bigr)^{1/m-2}
\end{align*}
for generic constants $C$.  Hence
\[
-\frac{Y_{uu}}{YY_u^2} = (C+o(1)) \bigl(\sqrt{2(n-1)} - u\bigr)^{-2/m} =(C+o(1)) Y^{-2}
\]
and therefore \eqref{eq-FetaY-positive} holds if $\sqrt{2(n-1)} - u > CA^{-m}= Ce^{-m\gamma\tau_0}$.  
If $y(u)\geq \rho e^{\gamma\tau_0}$ then $y(u)= AY(u)$ implies
\[
Y(u)\geq \frac{\rho}{A}e^{\gamma\tau_0} = \Bigl(\frac{K_0}{2(n-1)}\Bigr)^{1/m}\rho =C\rho
 \implies 
-\frac{Y_{uu}}{YY_u^2}\leq  \frac{C}{\rho^2}.  
\]
Therefore \eqref{eq-FetaY-positive} holds in the region $y(u)\geq \ell_{\rm int} $ provided $\ell_{\rm int}$ is large enough.  
\end{proof}

Up to this point we have shown that $\bu_\tau>0$ on almost all of the initial surface, but not on the whole surface.  There is a hole at the tip where we still have to check that $\bu_\tau>0$.  That's what we do in the next two subsections.

\subsection{Monotonicity at $\tau=\tau_0$, $u < C_1 e^{-\gamma\tau_0}$} 

The subsolution \eqref{eq:u lambda defined} has a ``hole'' in the region \(u=\cO(e^{-\gamma\tau_0})\), which we now fill.  Let \(B\) be the standard unit speed translating bowl soliton, i.e.  \(\cB:[0, \infty)\to\R\) is the unique solution of
\begin{equation}		\label{eq:bowl soliton ode}
\frac{\cB''(z)}{1+\cB'(z)^2} + \frac{n-1}{z}\cB'(z) = 1, \qquad \cB(0)=\cB'(0)=0.  
\end{equation}
For large \(z\) it is well known (see e.  g.  ~\cite{AV}) that as $z\to\infty$,
\begin{subequations}	\label{eq-bowl soliton asymptotics}
\begin{align}
\cB(z) &= \frac{z^2}{2(n-1)} - 2 \ln z + C + o(1), \\
\cB'(z) &=\frac{z}{n-1}+O(z^{-1})
\end{align}
\end{subequations}
for some constant \(C\).  It is also known that
\begin{equation}		\label{eq:B-starshaped}
z\cB'(z) <\cB(z)\text{ for all }z>0.  
\end{equation}

For any \(a>0\) we consider
\[
y_{aA}(u) := A - \frac{1}{aA} \cB(aAu),
\]
where, as before, $A=\bigl(2(n-1)/K_0\bigr)^{1/m}e^{\gamma\tau_0}$, and we test if \(y_{aA}(u)+b\) is a subsolution for any $b\in\R$:
\[
\cF[y_{aA}+b] = -\frac{aA\cB''(aAu)}{1+\cB'(aAu)^2}
			-\Bigl(\frac{n-1}{u} - \frac{u}{2}\Bigr)\cB'(aAu)
			+ \frac 12 A -\frac{1}{2aA}\cB(aAu) + \frac 12 b.  
\]
Using \eqref{eq:bowl soliton ode}, \eqref{eq-bowl soliton asymptotics},
\eqref{eq:B-starshaped}, we get
\[
\cF[y_{aA}+b]
	= \Bigl(\frac 12 - a\Bigr)A + \frac{1}{2aA}\bigl(aAu\, \cB'(aAu)-\cB(aAu)\bigr) + \frac 12 b 
	\geq \Bigl(\frac 12 - a\Bigr)A + \frac 12 b.  
\]
\subsection{Patching the two subsolutions}
Consider
\[
y_-(u) =
\begin{cases}
y_{aA}(u)+b	& \bigl(0\leq u \leq C_1 e^{-\gamma\tau_0}\bigr)  
\\[1ex]
AY(u)		& \bigl(C_1e^{-\gamma\tau_0} \leq u\leq \sqrt{2(n-1)-K_0 e^{-m\gamma\tau_0}\ell_{\rm int}^m}\bigr)
\end{cases}
\]
where, as before, $A=(2(n-1)/K_0)^{1/m}e^{\gamma\tau_0}$.  We always choose the
coefficient $b\in\R$ so that $y_-$ is continuous at $u=C_1$, i.e.  ~for given $a, C_1,
\tau_0$, we choose
\[
  b = AY(C_1 e^{-\gamma\tau_0}) - y_{aA}(C_1 e^{-\gamma\tau_0}) 
\]
which, in view of the definitions of $Y(u)$ and $y_{aA}$ can be expanded as
\begin{align*}
  b &= A\left\{\Bigl(1-\frac{C_1^2 e^{-2\gamma\tau_0}}{2(n-1)}\Bigr)^{\frac 1m}
  	-1 + \frac{1}{aA^2}\cB(aA C_1 e^{-\gamma\tau_0})\right\}  \\
    &=A\left\{
	-\frac{C_1^2 e^{-2\gamma\tau_0}}{2m(n-1)} + O(e^{-4\gamma\tau_0})
	+ \frac 1a \frac{K_0}{2(n-1)} e^{-2\gamma\tau_0} \cB\bigl(a({2(n-1)/K_0})^{1/m} C_1\bigr)
    \right\} \\
    &= O(e^{-2\gamma\tau_0}).  
\end{align*}
We now show that if  $a\in \bigl(\frac 1m, \frac 12\bigr)$, and if $C_1$ is large enough,
then $y_-$ is a subsolution, i.e.  ~one has $\cF[y_-]\geq 0$,  even in the viscosity sense
at $u=C_1e^{-\gamma\tau_0}$.  

We have already shown that $\cF[y_-]\geq 0$ when $u\neq C_1e^{-\gamma\tau_0}$, so we only have to check that
\[
\lim_{\epsilon\searrow 0} y_-'\bigl(C_1e^{-\gamma\tau_0}+\epsilon\bigr) \geq \lim_{\epsilon\searrow 0} y_-'\bigl(C_1e^{-\gamma\tau_0}-\epsilon\bigr),
\]
i.e.  
\begin{equation}		\label{eq-y-convex-corner}
A Y'\bigl(C_1e^{-\gamma\tau_0}\bigr) \geq  y_{aA}'\bigl(C_1e^{-\gamma\tau_0}\bigr).  
\end{equation}
Since $y_{aA}'(u)= -\cB'(aAu)$, we have 
\[
y_{aA}'(C_1e^{-\gamma\tau_0}) = -\cB'(aAC_1e^{-\gamma\tau_0})
= -\frac{a}{n-1}\Bigl(\frac{2(n-1)}{K_0}\Bigr)^{\frac{1}{m}} C_1 +  O(C_1^{-1}).  
\]
Using \eqref{eq:Y for small u} we find 
\begin{align*}
  y_-'(C_1e^{-\gamma\tau_0}) = AY'(C_1e^{-\gamma\tau_0}) 
&= -A \frac{C_1e^{-\gamma\tau_0}}{m(n-1)}+O(AC_1^3e^{-3\gamma\tau_0})  \\
&= -\Bigl(\frac{2(n-1)}{K_0}\Bigr)^{\frac 1m} \frac{C_1}{m(n-1)} + O\bigl(C_1^3 e^{-2\gamma\tau_0}\bigr)
\end{align*}
Therefore, at $u=C_1 e^{-\gamma\tau_0}$ 
\[
AY'(u) - y_{aA}'(u)
  = \Bigl(a - \frac 1m\Bigr) \Bigl(\frac{2(n-1)}{K_0}\Bigr)^{\frac 1m} \frac{C_1}{n-1}
    + O\bigl(C_1^{-1} + C_1^3 e^{-2\gamma\tau_0}\bigr).  
\]
We have chosen $a>\frac 1m$, so if $C_1$ is large enough, then the $O(C_1^{-1})$ term is small compared to the first $C_1$ term.  If $\tau_0$ is sufficiently large then we can also ignore the $O(C_1^3 e^{-2\gamma\tau_0})$ term.  Therefore \eqref{eq-y-convex-corner} holds if $C_1$ and $\tau_0$ are large enough.  This completes the proof that $u_\tau(y, \tau) > 0$ for $y>\ell_{\rm int}$, $\tau>\tau_0$.

%%%%%%%%%%%%%%%%%%%%%%%%%%

\end{document}